\documentclass{amsart}
\usepackage{epsfig,amsmath,amsfonts,amssymb,rotating,graphicx,mathtools,amsthm}
\usepackage{tikz,pgfplots}
\usetikzlibrary{arrows,snakes,backgrounds}
\usepackage{comment}
\usepackage{algorithm}			%Floating environment for algorithms
\usepackage{algpseudocode}
\usepackage{subfigure}
\usepackage{placeins}
\usepackage{enumitem}
\usepackage[nocompress]{cite}

\theoremstyle{plain}             % oder andere
\newtheorem{theorem}{Theorem}[section]
\newtheorem{lemma}[theorem]{Lemma}

\newtheorem{definition}[theorem]{Definition}
\newtheorem{remark}[theorem]{Remark}

\newcommand{\red}[1]{\textcolor{red}{#1}}
\newcommand{\blue}[1]{\textcolor{blue}{#1}}

\newcommand{\children}{\operatorname{child}}
\newcommand{\isdef}{\mathrel{\mathrel{\mathop:}=}}
\newcommand{\defis}{\mathrel{=\mathrel{\mathop:}}}

\newcommand{\rank}{\operatorname{rank}}

\newcommand{\spann}{\operatorname{span}}
\newcommand{\de}{\operatorname{d}\!}

\newcommand{\peq}{{\text{\tt+=}}}

\newcommand\numberthis{\addtocounter{equation}{1}\tag{\theequation}}
\newcommand{\scalefig}[1]{\scalebox{0.75}{#1}}
\newcommand{\sumexpr}{\texttt{sum}}
\newcommand{\true}{\mathrm{true}}
\newcommand{\false}{\mathrm{false}}
\newcommand{\adm}{\operatorname{adm}}
\newcommand{\parent}{\operatorname{parent}}
\newcommand{\level}{\operatorname{level}}
\newcommand{\depth}{\operatorname{depth}}
\newcommand{\Csp}{C_{\operatorname{sp}}}
\newcommand{\Cid}{C_{\operatorname{id}}}

\newcommand{\tikzlrmat}[2]{
	\draw (#1,#2-0.25) rectangle (#1+0.1,#2+0.25);
	\draw (#1+0.2, #2+0.18) node {$\cdot$};
	\draw (#1+0.3,#2+0.15) rectangle (#1+0.8,#2+0.25);
	}
\newcommand{\tikzHmatS}[3]{
	\foreach \tikzHmatSx in {0,1,2,3,4} {
		\draw (#1+#3*0.5*\tikzHmatSx/4,#2-#3*0.25) -- (#1+#3*0.5*\tikzHmatSx/4,#2+#3*0.25);
		\draw (#1,#2-#3*0.25+#3*0.5*\tikzHmatSx/4) -- (#1+#3*0.5,#2-#3*0.25+#3*0.5*\tikzHmatSx/4);
	}
}
\newcommand{\tikzHmat}[2]{
	\tikzHmatS{#1}{#2}{1}
}
\newcommand{\tikzHmatbig}[2]{
	\tikzHmatS{#1}{#2}{4}
}
\newcommand{\tikzHHmat}[2]{
	\tikzHmat{#1}{#2}
	\draw (#1+0.6,#2) node {$\cdot$};
	\tikzHmat{#1+0.7}{#2}
}

\definecolor{mycolor1}{rgb}{0.000000,0.000000,0.901961}
\definecolor{mycolor2}{rgb}{0.281863,0.000000,0.620098}
\definecolor{mycolor3}{rgb}{0.563726,0.000000,0.338235}
\definecolor{mycolor4}{rgb}{0.845588,0.000000,0.056373}
\definecolor{mycolor5}{rgb}{0.901961,0.225490,0.000000}
\definecolor{mycolor6}{rgb}{0.901961,0.507353,0.000000}
\definecolor{mycolor7}{rgb}{0.901961,0.789216,0.000000}

\begin{document}
\title{On the best approximation of the hierarchical matrix product}
\author{J\"urgen D\"olz}
\address{
J\"urgen D\"olz,
Technical University of Darmstadt,
Graduate School CE,
Dolivostra\ss{}e 15, 64293 Darmstadt, Germany}
\email{doelz@gsc.tu-darmstadt.de}
\author{Helmut Harbrecht}
\address{
Helmut Harbrecht,
University of Basel,
Department of Mathematics and Computer Science,
Spiegelgasse 1, 4051 Basel, Switzerland}
\email{helmut.harbrecht@unibas.ch}
\author{Michael D. Multerer${}^\dagger$}
\address{
Michael D. Multerer, born as Michael D.~Peters,
ETH Zurich,
Department of Biosystems Science and Engineering,
Mattenstrasse 26, 4058 Basel, Switzerland}
\email{michael.peters@bsse.ethz.ch}
{\let\thefootnote\relax\footnote{${}^\dagger$~Michael D.~Multerer was born as Michael D.~Peters}}
\begin{abstract}
The multiplication of matrices is an important arithmetic operation in 
computational mathematics. In the context of hierarchical matrices, 
this operation can be realized by the multiplication 
of structured block-wise low-rank matrices, resulting in an almost linear cost.
However, the computational efficiency of the algorithm is based on a 
recursive scheme which makes the error analysis quite involved. In this 
article, we propose a new algorithmic framework for the multiplication 
of hierarchical matrices. It improves currently known implementations 
by reducing the multiplication of hierarchical matrices towards finding a suitable
low-rank approximation of sums of matrix-products. We propose several 
compression schemes to address this task. As a consequence, we are able to 
compute the best-approximation of hierarchical matrix products. A 
cost analysis shows that, under reasonable assumptions on the 
low-rank approximation method, the cost of the framework is 
almost linear with respect to the size of the matrix. Numerical 
experiments show that the new approach produces indeed the 
best-approximation of the product of hierarchical matrices for 
a given tolerance. They also show that the new multiplication 
can accomplish this task in less computation time than 
the established multiplication algorithm without error control.
\end{abstract}
%===============================================================================
\maketitle
%===============================================================================

%\tableofcontents
%!TEX root = ./paper.tex
\section{Introduction}
%=======================================
Hierarchical matrices, $\mathcal{H}$-matrices for short,
historically originate from the discretization 
of boundary integral operators.
They allow for a data sparse approximation in terms of a block-wise
low-rank matrix. 
As first shown in \cite{Hack1}, the major advantage of the $\mathcal{H}$-matrix
representation over other data sparse formats for non-local operators is that 
basic operations, like addition, multiplication and inversion, can be performed
with nearly linear cost.
This fact enormously stimulated the research on 
$\mathcal{H}$-matrices, see e.g.\ \cite{buchbeb,BH,GH03,GHK,HK00,Hac15} 
and the references therein, and related hierarchical matrix formats like 
HSS, see \cite{HSS2,HSS3,HSS1}, HODLR, see \cite{HODLR1,HODLR2}, and 
$\mathcal{H}^2$-matrices, cf.\ \cite{boerm,HKS}. 

The applications for $\mathcal{H}$-matrices 
are manifold: They
have been used for solving large scale algebraic matrix {R}iccati 
equations, cf.\ \cite{riccati}, for solving Lyapunov equations, cf.\ \cite{baur}, 
for preconditioning, cf.\ \cite{precond1,precond2} and for the second moment 
analysis of partial differential equations with random data, cf.\ \cite{DHP15,DHP17}, 
just to name a few. 

In this context, the matrix-matrix multiplication of 
\(\mathcal{H}\)-matrices is an essential operation. 
Based on the hierarchical block structure, the matrix-matrix multiplication is performed
in a recursive fashion. To that end, in each recursive call, the two factors are considered as 
block-matrices which are multiplied by a block-matrix product. The resulting matrix-block
is then again compressed to the $\mathcal{H}$-matrix format. To limit the computational cost,
the block-wise ranks for the compression are usually priorily
bounded by a user-defined threshold. 
For this thresholding procedure, no a-priori error estimates exist. 
This fact and the recursive structure of the matrix-matrix multiplication render the error analysis 
difficult, in particular, since there is no guarantee that intermediate results 
provide the necessary low-rank structure.

To reduce the number of these time-consuming and error-introducing truncation 
steps, different modifications have been proposed in the literature:
In \cite{Boe17}, instead of applying each low-rank update immediately to an 
$\mathcal{H}$-matrix, multiple updates are accumulated in an auxiliary low-rank 
matrix. This auxiliary matrix is propagated as the algorithm traverses the hierarchical 
structure underlying the $\mathcal{H}$-matrix. 
This greatly improves computation times, although the computational cost is not improved. 
Still, also in this approach, multiple truncation steps are performed.
Thus, it does not
lead to the best approximation of the low-rank block under consideration.

As an alternative, in \cite{DHP15}, it has been proposed to directly 
compute the low-rank approximation of the output matrix block by using 
the truncated singular value decomposition, realized by means of a Krylov 
subspace method. This requires only the implementation of matrix-vector 
multiplications and is hence easy to realize. Especially, it yields 
the best approximation of the low-rank blocks to be computed. But 
contrary to expectations, it does not increase efficiency since the 
eigensolver converges very slowly in case of a clustering of the 
eigenvalues. Therefore, computing times have not been satisfactory.

In the present article, we pick up the idea from \cite{DHP15} and provide
an algorithm that facilitates the direct computation of any matrix block 
in the product of two \(\mathcal{H}\)-matrices. This algorithm is based on
a sophisticated bookkeeping technique in combination with 
a compression based on basic matrix-vector products.
This new algorithm will naturally lead to the 
best approximation of the \(\mathcal{H}\)-matrix product
within the \(\mathcal{H}\)-matrix format. In particular, we cover the 
cases of an optimal fixed rank truncation and of an optimal adaptively chosen rank
based on a prescribed accuracy.

Our numerical experiments show that both, the fixed rank and the 
adaptive versions of the used low-rank techniques, are 
significantly more efficient than the traditional arithmetic with fixed rank. 
In particular, the numerical experiments also validate that the desired
error tolerance can indeed be reached when using the adaptive algorithms.

For the actual compression of a given matrix block, any compression technique
based on matrix-vector products is feasible. Exemplarily, we shall consider here the
adaptive cross approximation, see \cite{Beb,GTZ},
the Golub-Kahan-Lanczos bidiagonalization procedure,
see \cite{golub1965},
and the randomized range approximation, cf.\ \cite{HMT11}.
We will employ these methods to either compute approximations with fixed ranks
or with adaptively chosen ranks.
We remark that, in the fixed rank case, 
a similar algorithm for randomized range approximation, 
has successively been applied to directly compute the approximation 
of the product of two HSS or HODLR matrices in \cite{martinsson11, martinsson16}.

The rest of this article is structured as follows. In Section
\ref{sec:pre}, we briefly recall the construction and structure
of $\mathcal{H}$-matrices together with their matrix-vector
multiplication. Section \ref{sec:Hmult} is dedicated to the
new framework for the matrix-matrix multiplication. The three
example algorithms for the efficient low-rank approximation are then the topic of
Section~\ref{sec:approx}. Section~\ref{sec:cost} is concerned
with the analysis of the computational cost of the new multiplications, which shows
that the new matrix-matrix multiplication has asymptotically the same
computational cost as the
standard matrix-matrix multiplication, i.e., almost linear in 
the number of degrees of freedom. Nonetheless, as the 
numerical results in Section~\ref{sec:results} show,
the constants in the estimates are significantly lower for
the new approach, resulting in a remarkable speed-up.
Finally, concluding remarks are stated in Section
\ref{sec:conclusion}.

%!TEX root = ./paper.tex
\section{Preliminaries}\label{sec:pre}
The pivotal idea of hierarchial matrices is to introduce a tree structure
on the cartesian product $\mathcal{I}\times\mathcal{I}$, where \(\mathcal{I}\) is a
suitable index set. 
The tree structure is then used
to identify sub-blocks of the matrix which are suitable for low-rank representation.
We follow the the monograph \cite[Chapter 5.3 and A.2]{Hac15} and first
recall a suitable definition of a tree.

\begin{definition}\label{def:prel.tree}
Let $V$ be a non-empty finite set, call it \emph{vertex set}, and let $\children$
be a mapping from $V$ into the power set $\mathcal{P}(V)$, i.e., $\children\colon
V\to\mathcal{P}(V)$. For any $v\in V$, an
element $v'$ in $\children(v)$ is called \emph{child}, whereas we call $v$ the
\emph{parent} of $v'$.
We call the structure $T(V,\children)$ a \emph{tree}, if the following properties
hold.
\begin{enumerate}
\item There is exactly one element $r\in V$ which is not a child of a vertex, i.e.,
\[
\bigcup_{v\in V}\children(v)=V\setminus\{r\}.
\]
We call this vertex the \emph{root} of the tree.
\item All $v\in V$ are \emph{successors} of $r$, i.e., there is a
$k\in\mathbb{N}_0$, such that $v\in\children^k(r)$. We define $\children^k(v)$
recursively as
\[
\children^0(v)=\{v\}\quad\text{and}\quad\children^k(v)=\bigcup_{v'\in\children^{k-1}(v)}\children(v').
\]
\item Any $v\in V\setminus\{r\}$ has exactly one parent.
\end{enumerate}
Moreover, we say that the number $k$ is the \emph{level} of $v$. The
\emph{depth} of a tree is the maximum of the levels of its vertices.
We define the set of \emph{leaves} of $T$ as
\[
\mathcal{L}(T)=\{v\in V\colon \children(v)=\emptyset\}.
\]
\end{definition}

We remark that for any $v\in T$, there is exactly one path from $r$ to $v$, see
\cite[Remark A.6]{Hac15}.

\begin{definition}\label{def:prel.clusterTree}
	Let $\mathcal{I}$ be a finite index set. The \emph{cluster tree}
	$\mathcal{T}_\mathcal{I}$ is a tree with the following properties.
	\begin{enumerate}
		\item $\mathcal{I}\in\mathcal{T}_\mathcal{I}$ is the root of the tree
		$\mathcal{T}_\mathcal{I}$,
		\item for all non-leaf elements $\tau\in\mathcal{T}_\mathcal{I}$ it
		holds
		\[
		\dot{\bigcup_{\sigma\in\children(\tau)}}\sigma =\tau,
		\]
		i.e., all non-leaf clusters are the disjoint union of their children,
		\item all $\tau\in\mathcal{T}_\mathcal{I}$ are non-empty.
	\end{enumerate}
	The vertices of the cluster tree are referred to as \emph{clusters}.
\end{definition}
To achieve almost linear cost for the following algorithms, we
shall assume that the depth of the cluster tree is bounded by
$\mathcal{O}(\log(\#\mathcal{I}))$ and that the cardinality of the leaf clusters
is bounded by $n_{\min}$. Various ways to construct a cluster tree
fulfilling this requirement along with different kinds of other properties
exist, see \cite{Hac15} and the references therein.

Obviously, by applying the second requirement of the definition recursively, it
holds $\tau\subset\mathcal{I}$ for all $\tau\in\mathcal{T}_\mathcal{I}$. Consequently, 
the leaves of the cluster tree form a partition of $\mathcal{I}$.

\begin{definition}\label{def:admissibility}
An \emph{admissibility condition} for $\mathcal{I}$ is a mapping
\[
\adm\colon\mathcal{P}(\mathcal{I})\times\mathcal{P}(\mathcal{I})\to\{\true,\false\}
\]
which is symmetric, i.e., for
$\tau\times\sigma\in\mathcal{P}(\mathcal{I})\times\mathcal{P}(\mathcal{I})$
it holds
\[
\adm(\tau,\sigma)=\adm(\sigma,\tau),
\]
and monotone, i.e., if $\adm(\tau,\sigma)=\true$, it holds
\[
\adm(\tau',\sigma')=\true,\quad\text{for all}~\tau'\in\children(\tau),\sigma'\in\children(\sigma).
\]
\end{definition}
Different kinds of admissibility exist, see \cite{Hac15} for a thorough
discussion and examples. Based on the admissibility condition and the
cluster tree, the block-cluster tree forms a partition of the index set
$\mathcal{I}\times\mathcal{I}$.

\begin{algorithm}
\caption{Construction of the block-cluster tree \(\mathcal{B}\), cf.~\cite[Definition 5.26]{Hac15}}
\label{alg:prel.buildFuN}
\begin{algorithmic}
\Function{BuildBlockClusterTree}{block-cluster $b=\tau\times\sigma$}
\If {$\adm(\tau,\sigma)=\true$}
\State $\children(b)\isdef\emptyset$
%\State Add $b$ to $\mathcal{F}$
\Else
\If {$\tau$ and $\sigma$ have sons}
\State $\operatorname{sons}(b)\isdef
\{\sigma'\times\tau'\colon\tau'\in\children(\tau),\sigma'\in\children(\sigma)\}$
\For {$b'\in\children(b)$}
\State \Call{BuildBlockClusterTree}{$b'$}
\EndFor
\Else
%\State Add $b$ to $\mathcal{N}$
\State $\children(b)\isdef\emptyset$
\EndIf
\EndIf
\EndFunction
\end{algorithmic}
\end{algorithm}

\begin{definition}
Given a cluster-tree $\mathcal{T}_{\mathcal{I}}$, the tree structure
$\mathcal{B}$ constructed by Algorithm~\ref{alg:prel.buildFuN} invoked
with $b=\mathcal{I}\times\mathcal{I}$ is referred to as
\emph{block-cluster tree}.
\end{definition}

For notational purposes, we write $p=\depth(\mathcal{B})$ and
refer to $\mathcal{N}$ as the set of
inadmissible leaves of $\mathcal{B}$ and call it the \emph{nearfield}.
In a similar fashion, we will refer to $\mathcal{F}$ as the set of
admissible leaves of $\mathcal{B}$ and call it the \emph{farfield}.
We remark that the depth of the block-cluster tree is bounded by the
depth of the cluster tree and that our definition of the
block-cluster tree coincides with the notion of a \emph{level-conserving
block-cluster tree} from the literature.

\begin{definition}
	For a block-cluster $b=\tau\times\sigma$  and $k\leq\min\{\#\tau,\#\sigma\}$,
	we define the set of \emph{low-rank matrices} as
	\[
	\mathcal{R}(\tau\times\sigma,k)=\big\{\mathbf{M}\in\mathbb{R}^{\tau\times\sigma}\colon\rank(\mathbf{M})\leq k\big\},
	\]
	where all elements $\mathbf{M}\in\mathcal{R}(\tau\times\sigma,k)$ are stored in
	low-rank representation, i.e.,
	\[
	\mathbf{M}=\mathbf{L}_{\mathbf{M}}\mathbf{R}_{\mathbf{M}}^{\intercal}
	\]
	for matrices $\mathbf{L}_{\mathbf{M}}\in\mathbb{R}^{\tau\times k}$ and
	$\mathbf{R}_{\mathbf{M}}\in\mathbb{R}^{\sigma\times k}$.
\end{definition}
Obviously, a matrix in $\mathcal{R}(\tau\times\sigma,k)$ requires
$k(\#\tau+\#\sigma)$ units of storage instead of \mbox{$\#\tau\cdot\#\sigma$},
which
results in a significant storage improvement if $k\ll\min\{\#\tau,\#\sigma\}$.
The same consideration holds true for the matrix-vector multiplication.

With the definition of the block-cluster tree at hand, we are 
in the position to introduce hierarchical matrices.

\begin{definition}
	Given a block-cluster tree $\mathcal{B}$, the set of \emph{hierarchical
		matrices}, in short \emph{$\mathcal{H}$-matrices}, of maximal block-rank
	$k$ is given by
	\[ 
	\mathcal{H}(\mathcal{B},k):=\Big\{{\bf H}\in\mathbb{R}^{\#\mathcal{I}\times\#\mathcal{I}}:
	\mathbf{H}|_{\tau\times\sigma}\in\mathcal{R}(\tau\times\sigma,k)\text{ for all }\tau\times\sigma\in\mathcal{F}\Big\}.
	\]
	A tree structure is induced on each element of this set by the tree
	structure of the block-cluster tree. Note that all nearfield blocks
	${\bf H}|_{\tau\times\sigma}$, $\tau\times\sigma\in\mathcal{N}$, are
	allowed to be dense matrices.
\end{definition}

The tree structure of the block-cluster tree provides the following useful
recursive block matrix structure on $\mathcal{H}$-matrices. Every
matrix block $\mathbf{H}|_{\tau\times\sigma}$, corresponding to a non-leaf
block-cluster $\tau\times\sigma$, has the structure
\begin{align}\label{eq:blockHMat}
	\mathbf{H}\big|_{\tau\times\sigma}=
	\begin{bmatrix}
		\mathbf{H}\big|_{\children(\tau)_1\times\children(\sigma)_1}&\ldots &\mathbf{H}\big|_{\children(\tau)_1\times\children(\sigma)_{\#\children(\sigma)}}\\
		\vdots&&\vdots\\
		\mathbf{H}\big|_{\children(\tau)_{\#\children(\tau)}\times\children(\sigma)_1}&\ldots&\mathbf{H}\big|_{\children(\tau)_{\#\children(\tau)}\times\children(\sigma)_{\#\children(\sigma)}}
	\end{bmatrix}.
\end{align}
If the matrix block $\mathbf{H}|_{\tau'\times\sigma'}$, $\tau'\in\children(\tau)$,
$\sigma'\in\children(\sigma)$, is a leaf of $\mathcal{B}$, the matrix block is
either a low-rank matrix, if $\tau'\times\sigma'\in\mathcal{F}$, or a dense
matrix, if $\tau'\times\sigma'\in\mathcal{N}$. If the matrix block is not
a leaf of $\mathcal{B}$, it exhibits again a similar block structure as
$\mathbf{H}|_{\tau\times\sigma}$. The required ordering of the clusters relies
on the order of the indices in the clusters. A possible block structure of an
$\mathcal{H}$-matrix is illustrated in Figure~\ref{fig:hmatrix}. Of course, the
structure may vary depending on the used cluster tree and admissibility
condition.

%\begin{figure}
%	\includegraphics[width=0.4\textwidth]{Figures/Hmatrix}
%\caption{\label{fig:hmatrix}Illustration of the recursive block structure of an $\mathcal{H}$-matrix.}
%\end{figure}

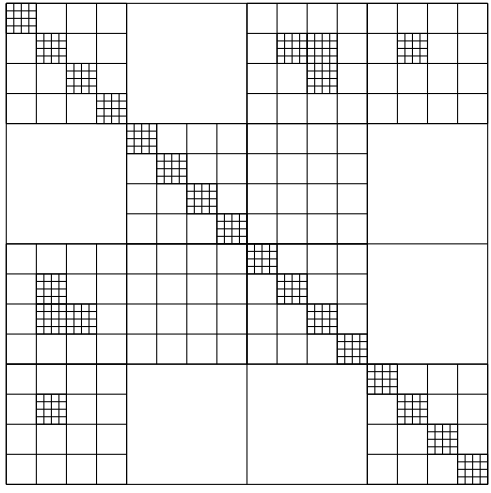
\begin{figure}
\centering
\begin{tikzpicture}[scale=0.8]
\tikzHmatS{0}{0}{16}
\foreach \x in {0,2,4,6} {
	\tikzHmatS{0+\x}{3-\x}{4}
}
\tikzHmatS{4}{3}{4}
\tikzHmatS{4}{1}{4}
\tikzHmatS{6}{3}{4}
\tikzHmatS{0}{-1}{4}
\tikzHmatS{2}{-1}{4}
\tikzHmatS{0}{-3}{4}
\foreach \x in {0,...,15} {
	\tikzHmatS{0+\x/2}{3.75-\x/2}{1}
}
\tikzHmatS{4.5}{3.25}{1}
\tikzHmatS{5}{3.25}{1}
\tikzHmatS{5}{2.75}{1}
\tikzHmatS{6.5}{3.25}{1}

\tikzHmatS{0.5}{-0.75}{1}
\tikzHmatS{0.5}{-1.25}{1}
\tikzHmatS{1}{-1.25}{1}
\tikzHmatS{0.5}{-2.75}{1}
\end{tikzpicture}
\caption{\label{fig:hmatrix}Illustration of the recursive block structure of an $\mathcal{H}$-matrix. Only the smallest blocks are allowed to be dense matrices while all
other blocks are represented by low-rank matrices.}
\end{figure}

Having the block structure \eqref{eq:blockHMat} available, an algorithm for
the matrix-vector multiplication, as listed in Algorithm~\ref{alg:prel.hVmult},
can easily be derived.
\begin{algorithm}
	\caption{\label{alg:prel.hVmult}$\mathcal{H}$-matrix-vector multiplication $\mathbf{y}\peq\mathbf{H}\mathbf{x}$, 
		see \cite[Equation (7.1)]{Hac15}}
	\begin{algorithmic}
		\Function{$\mathcal{H}$timesV}{$\mathbf{y}|_{\tau}$, $\mathbf{H}|_{\tau\times\sigma}$, $\mathbf{x}|_{\sigma}$}
		\If{$\tau\times\sigma\notin\mathcal{L}(\mathcal{B})$}
		\For{$\tau'\times\sigma'\in\children(\tau\times\sigma)$}
		\State\Call{$\mathcal{H}$timesV}{$\mathbf{y}|_{\tau'}$, $\mathbf{H}|_{\tau'\times\sigma'}$, $\mathbf{x}|_{\sigma'}$}
		\EndFor
		\Else
		\State $\mathbf{y}|_{\tau}\peq\mathbf{H}|_{\tau\times\sigma}\mathbf{x}|_{\sigma}$
		\EndIf
		\EndFunction
	\end{algorithmic}
\end{algorithm}
Note that the matrix-vector multiplication for the leaf block-clusters involves either
a dense matrix or a low-rank matrix. In accordance with \cite[Lemma 7.17]{Hac15},
the matrix-vector multiplication of an $\mathcal{H}$-matrix block
$\mathbf{H}|_{\tau\times\sigma}$ with $\mathbf{H}\in\mathcal{H}(\mathcal{B},k)$ can
be accomplished in at most
\[
N_{\mathcal{H}\cdot v}(\tau\times\sigma,k)\leq 2\Csp \max\{k,n_{\min}\}\Big(\big(\depth(\#\tau)+1\big)\#\tau+\big(\depth(\#\sigma)+1\big)\#\sigma\Big)
\]
operations, where the constant $\Csp=\Csp(\mathcal{B})$ is given as
\[
\Csp(\mathcal{B})\isdef
\max_{\tau\in\mathcal{T}_{\mathcal{I}}}\#\big\{\sigma\in\mathcal{T}_{\mathcal{I}}\colon \tau\times\sigma\in\mathcal{B}\big\}.
\]
Given a cluster $\tau\in\mathcal{T}_{\mathcal{I}}$, the quantity $\Csp$ is an upper bound on the
number of corresponding block-clusters $\tau\times\sigma\in\mathcal{B}$. Thus, it is also an
upper bound for the number of corresponding matrix blocks $\mathbf{H}|_{\tau\times\sigma}$
in the tree structure of an $\mathcal{H}$-matrix corresponding to $\mathcal{B}$.

%!TEX root = ./paper.tex
%%%%%%%%%%%%%%%%%%%%%%%%%%%%%%%%%%%%%%%%%%%%%%%%%%%%%%%%%%%%%%%%%%%%%%%%%%%%%%%%
\section{The multiplication of $\mathcal{H}$-matrices}
\label{sec:Hmult}
%%%%%%%%%%%%%%%%%%%%%%%%%%%%%%%%%%%%%%%%%%%%%%%%%%%%%%%%%%%%%%%%%%%%%%%%%%%%%%%%

Instead of restating the $\mathcal{H}$-matrix multiplication in its original
form from \cite{Hack1}, we directly introduce our new framework. The connection
between the new framework and the traditional multiplication will be discussed
later in this section.

We start by introducing the
following \emph{\sumexpr-expressions}, which will simplify the presentation
and the analysis of the new algorithm.

\begin{definition}\label{def:sumexpr}
Let $\tau\times\sigma\in\mathcal{B}$.
For a finite index set $\mathcal{J}_{\mathcal{R}}$, the expression
\[
\mathcal{S}_\mathcal{R}(\tau,\sigma)=\sum_{j\in\mathcal{J}_{\mathcal{R}}}\mathbf{A}_j\mathbf{B}_j^{\intercal},
\]
is called a \sumexpr-\emph{expression of low-rank matrices}, if it is represented
and stored as a set of factorized low-rank matrices
\[
\big\{ \mathbf{A}_j\mathbf{B}_j^{\intercal}\in\mathcal{R}(\tau\times\sigma,k_j)\colon
j\in\mathcal{J}_{\mathcal{R}}\}.
\]
Similarly, for a finite index set $\mathcal{J}_{\mathcal{H}}$, the expression
\[
\mathcal{S}_\mathcal{H}(\tau,\sigma)=\sum_{j\in\mathcal{J}_{\mathcal{H}}}\mathbf{H}_j\mathbf{K}_j
\]
is called a \sumexpr-\emph{expression of $\mathcal{H}$-matrices}, if it is
represented and stored as a set of pairs of $\mathcal{H}$-matrix blocks
\[
\big\{\big(\mathbf{H}_j,\mathbf{K}_j\big)=\big(\mathbf{H}|_{\tau\times\rho_j},\mathbf{K}|_{\rho_j\times\sigma}\big)\colon\tau\times\rho_j,\rho_j\times\sigma\in\mathcal{B},j\in\mathcal{J}_{\mathcal{H}}\big\},
\]
with $\mathbf{H},\mathbf{K}\in\mathcal{H}(\mathcal{B},k)$ and
 $\mathbf{H}|_{\tau\times\rho_j}$, $\mathbf{H}|_{\tau\times\rho_j}$, $j\in\mathcal{J}_\mathcal{H}$,
being either dense matrices or providing the block-matrix structure
\eqref{eq:blockHMat}.

The expression
\[
\mathcal{S}(\tau,\sigma)=\mathcal{S}_\mathcal{R}(\tau,\sigma)+\mathcal{S}_\mathcal{H}(\tau,\sigma)
\]
is called a \sumexpr-\emph{expression} and a combination of the two
previously introduced expressions. In particular, we
require that $\mathcal{S}_\mathcal{R}$ is stored as a \sumexpr-expression of low-rank matrices
and $\mathcal{S}_\mathcal{H}$ is stored as a \sumexpr-expression of $\mathcal{H}$-matrices.
\end{definition}

\begin{figure}[htb]
\centering
\begin{tikzpicture}
\draw (0,0) node{$\mathcal{S}_\mathcal{R}(\tau,\sigma)$};
\draw (1,0) node{$=$};
\tikzlrmat{1.5}{0}
\draw (2.6,0) node {$+$};
\tikzlrmat{2.9}{0}
\draw (4.0,0) node {$+$};
\tikzlrmat{4.3}{0}
\draw (5.4,0) node {$+$};
\tikzlrmat{5.7}{0}

\draw (0,-1) node{$\mathcal{S}_\mathcal{H}(\tau,\sigma)$};
\draw (1,-1) node{$=$};
\tikzHHmat{1.5}{-1}
\draw (3.0,-1) node {$+$};
\tikzHHmat{3.3}{-1}
\draw (4.8,-1) node {$+$};
\tikzHHmat{5.1}{-1}

\draw (0,-2) node{$\mathcal{S}(\tau,\sigma)$};
\draw (1,-2) node{$=$};
\tikzlrmat{1.5}{-2}
\draw (2.6,-2) node {$+$};
\tikzHHmat{2.9}{-2}
\draw (4.4,-2) node {$+$};
\tikzHHmat{4.7}{-2}
\draw (6.2,-2) node {$+$};
\tikzlrmat{6.5}{-2}
\end{tikzpicture}
\caption{\label{fig:sumexpr}Examples for the introduced
	 \sumexpr-expressions. $\mathcal{S}_\mathcal{R}(\tau,\sigma)$ is a sum of low-rank
	 matrices, $\mathcal{S}_\mathcal{H}(\tau,\sigma)$ a sum of $\mathcal{H}$-matrix
	 products, and $\mathcal{S}(\tau,\sigma)$ a mixture of both.}
\end{figure}

Examples of \sumexpr-expressions
are illustrated in Figure~\ref{fig:sumexpr}. A \sumexpr-expressions may be considered as a kind
of a queuing system to store the sum of low-rank matrices and/or
$\mathcal{H}$-matrix products for subsequent operations.

%We remark that the sums and products from Definition~\ref{def:sumexpr} are well-defined.
%and that the \sumexpr-expressions are essentially a characterization of the representation
%of a sum of low-rank matrices and $\mathcal{H}$-matrix products.
We remark that the sum of two \sumexpr-expressions is again a \sumexpr-expression and
shall now make use of this fact to devise an algorithm for the multiplication
of $\mathcal{H}$-matrices. For simplicity, we assume that all
involved $\mathcal{H}$-matrices are built upon the same block-cluster tree
$\mathcal{B}$.

\subsection{Relation between $\mathcal{H}$-matrix products and \sumexpr-expressions}
We start with two $\mathcal{H}$-matrices $\mathbf{H},\mathbf{K}\in\mathcal{H}
(\mathcal{B},k)$ and want to represent their product
$\mathbf{L}\isdef\mathbf{H}\mathbf{K}$ in $\mathcal{H}(\mathcal{B},k)$.
To that end, we rewrite the $\mathcal{H}$-matrix product as a
\sumexpr-expression
\[
\mathbf{L}
=
\mathbf{H}\mathbf{K}
\defis
\mathcal{S}_\mathcal{H}(\mathcal{I},\mathcal{I})
\defis
\mathcal{S}(\mathcal{I},\mathcal{I}).
\]
The task is now to find a suitable low-rank approximation to
$\mathbf{L}|_{\tau\times\sigma}$ in $\mathcal{R}(\tau\times\sigma, k)$
for all admissible leafs $\tau\times\sigma$ of $\mathcal{B}$.
If $\tau\times\sigma$ is a
child of the root, i.e.,
$\tau\times\sigma\in\children(\mathcal{I}\times\mathcal{I})$, we have that
\begin{align*}
\mathbf{L}|_{\tau\times\sigma}={}&
\mathcal{S}_\mathcal{H}(\mathcal{I},\mathcal{I})|_{\tau\times\sigma}\\
={}&
\sum_{\rho\in\children(\mathcal{I})}
\mathbf{H}|_{\tau\times\rho}
\mathbf{K}|_{\rho\times\sigma}\\
={}&\sum_{\substack{
			\rho\in\children(\mathcal{I})\colon\\
			\tau\times\rho\in\mathcal{F}\\
			\text{or}\\
			\rho\times\sigma\in\mathcal{F}
		}}
		\mathbf{H}|_{\tau\times\rho}\mathbf{K}|_{\rho\times\sigma}+
		\sum_{\substack{
				\rho\in\children(\mathcal{I})\colon\\
				\tau\times\rho\in\mathcal{B\setminus F}\\
				\rho\times\sigma\in\mathcal{B\setminus F}
			}}
			\mathbf{H}|_{\tau\times\rho}\mathbf{K}|_{\rho\times\sigma},
\end{align*}
due to the block-matrix structure \eqref{eq:blockHMat} of $\mathbf{H}$ and
$\mathbf{K}$, see also Figure~\ref{fig:Hmatprod} for an illustration.

The pivotal idea is now that $\mathbf{L}|_{\tau\times\sigma}$ can be
written as a \sumexpr-expression itself, for which we treat the two remaining sums as follows:
\begin{itemize}
\item
The products in the first sum involve at least one low-rank matrix, such that
the product in low-rank representation can easily be computed using
matrix-vector multiplications. Having
these low-rank matrices computed, we can store the first sum as a
\sumexpr-expression of low-rank matrices $\mathcal{S}_\mathcal{R}(\tau,\sigma)$.
\item
Both factors of the products in the second sum correspond to inadmissible
block-clusters. Thus, they are either dense matrices or $\mathcal{H}$-matrices.
Since a dense matrix is just a special case of an $\mathcal{H}$-matrix, the
second sum can be written as a \sumexpr-expression of $\mathcal{H}$-matrices
$\mathcal{S}_\mathcal{H}(\tau,\sigma)$.
\end{itemize}
It follows that $\mathbf{L}|_{\tau\times\sigma}$ can be represented as a
\sumexpr-expressions by setting
\[
\mathbf{L}|_{\tau\times\sigma}=\mathcal{S}_\mathcal{R}(\tau,\sigma)+\mathcal{S}_\mathcal{H}(\tau,\sigma)\defis\mathcal{S}(\tau,\sigma),
\]
see also Figure~\ref{fig:Hmatprod} for an illustration.
\begin{figure}
\centering
\subfigure[Matrix blocks of $\mathbf{H}$ and $\mathbf{K}$ which have to be taken into
account for $\mathbf{L}|_{\tau\times\sigma}$.]{
\begin{tikzpicture}
\draw (1,1.2) node {$\mathbf{L}$};
\tikzHmatbig{0}{0}
\tikzHmat{0}{0.75}
\tikzHmat{0.5}{0.25}
\tikzHmat{1}{0.25}
\tikzHmat{1}{-0.25}
\tikzHmat{1.5}{-0.75}
\draw (2.5,0) node {$=$};
\draw (4,1.2) node {$\mathbf{H}$};
\tikzHmatbig{3}{0}
\tikzHmat{3}{0.75}
\tikzHmat{3.5}{0.25}
\tikzHmat{3}{-0.25}
\tikzHmat{4}{-0.25}
\tikzHmat{4.5}{-0.75}
\draw (6.5,1.2) node {$\mathbf{K}$};
\draw (5.25,0) node {$\cdot$};
\tikzHmatbig{5.5}{0}
\tikzHmat{5.5}{0.75}
\tikzHmat{6}{0.25}
\tikzHmat{6.5}{0.25}
\tikzHmat{6.5}{-0.25}
\tikzHmat{7}{-0.75}

\draw[very thick] (0,-0.5) rectangle (0.5,0);
\draw[blue,very thick] (0,-0.5) -- (0,0);
\draw[red,very thick] (0,0) -- (0.5,0);
\draw[blue] (-0.125,-0.25) node {$\tau$};
\draw[red] (0.25,0.125) node {$\sigma$};

\draw[very thick] (3,-0.5) rectangle (5,0);
\draw[blue,very thick] (3,-0.5) -- (3,0);
\draw[blue] (3-0.125,-0.25) node {$\tau$};

\draw[very thick] (5.5,-1) rectangle (6,1);
\draw[red,very thick] (5.5,1) -- (6,1);
\draw[red] (5.75,1.125) node {$\sigma$};
\end{tikzpicture}
}

\vspace{0.5cm}
\subfigure[Computation of the \sumexpr-expression for
$\mathbf{L}|_{\tau\times\sigma}$.]{
\begin{tikzpicture}
\draw (0.25,-0.6) node {$\mathbf{L}|_{\blue{\tau}\times\red{\sigma}}$};
\draw (0,-0.25) rectangle (0.5,0.25);
\draw[blue,very thick] (0,-0.25) -- (0,0.25);
\draw[red,very thick] (0,0.25) -- (0.5,0.25);
\draw[blue] (-0.125,0) node {$\tau$};
\draw[red] (0.25,0.375) node {$\sigma$};
\draw (0.75,0) node {$=$};
\tikzHHmat{1}{0}
\draw (2.5,0) node {$+$};
\tikzlrmat{2.8}{0}
\draw (3.8,0) node {$\cdot$};
\tikzlrmat{4}{0}
\draw (5.1,0) node {$+$};
\tikzHmat{5.4}{0}
\draw (6.1,0) node {$\cdot$};
\tikzlrmat{6.3}{0}
\draw (7.5,0) node {$+$};
\tikzlrmat{7.8}{0}
\draw (8.8,0) node {$\cdot$};
\tikzlrmat{9}{0}

\draw [decoration={brace,mirror,raise=0.5cm},decorate] (2.8,0) -- (4.8,0);
\draw [decoration={brace,mirror,raise=0.5cm},decorate] (5.4,0) -- (7.2,0);
\draw [decoration={brace,mirror,raise=0.5cm},decorate] (7.8,0) -- (9.8,0);

\tikzlrmat{3.4}{-1}
\tikzlrmat{5.9}{-1}
\tikzlrmat{8.4}{-1}

\draw [decoration={brace,mirror,raise=0.5cm},decorate] (1,-1) -- (2.2,-1);
\draw [decoration={brace,mirror,raise=0.5cm},decorate] (2.8,-1) -- (9.8,-1);

\draw (1.6,-2) node {$\mathcal{S}_\mathcal{H}(\blue{\tau}\times\red{\sigma})$};
\draw (2.5,-2) node {$+$};
\draw (6.3,-2) node {$\mathcal{S}_\mathcal{R}(\blue{\tau}\times\red{\sigma})$};
\draw (10.1,-2) node {$=$};
\draw (11,-2) node {$\mathcal{S}(\blue{\tau}\times\red{\sigma})$};

\end{tikzpicture}
}
\caption{\label{fig:Hmatprod}Illustration why a block $\mathbf{L}|_{\tau\times\sigma}$ on the coarsest level of the target $\mathcal{H}$-matrix can be represented as a \sumexpr-expression $\mathcal{S}(\tau,\sigma)$.}
\end{figure}
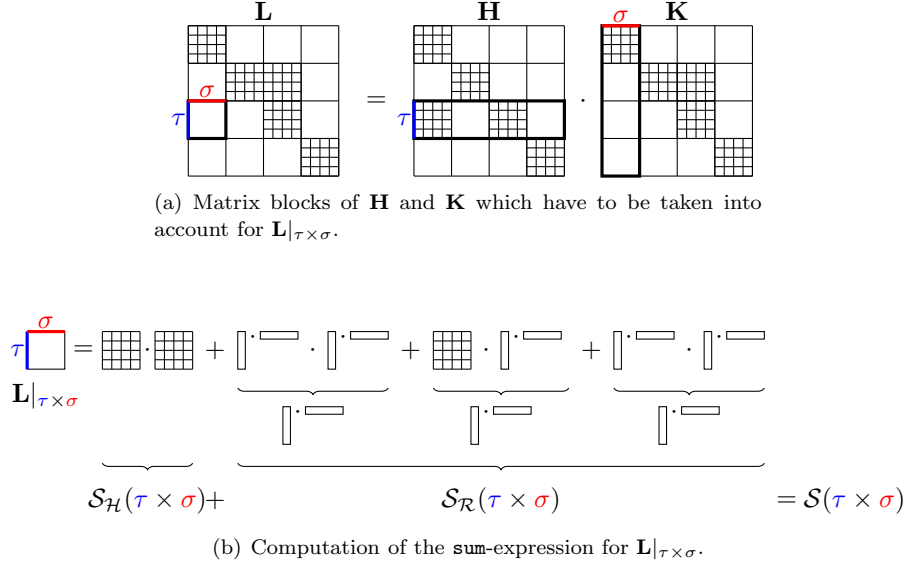

We can thus represent all children of the root of the block-cluster tree
by \sumexpr-expressions. However, we will require to represent all leaves
of the block-cluster tree as \sumexpr-expressions.
It thus remains to discuss how to represent matrix blocks
$\mathbf{L}|_{\tau\times\sigma}$ when $\tau\times\sigma$ is not a
child of the root.

\begin{remark}
The representation of any $\mathbf{L}|_{\tau\times\sigma}$ in terms
of \sumexpr-expressions is not unique. For example, assuming that
$\tau\times\sigma$ is on level $j$, one may refine the block-matrix
structure of $\mathbf{H}$ and $\mathbf{K}$ by modifying the corresponding
admissibility condition. When the admissibility condition is set to
$\false$ on all levels smaller or equal to $j$, one can construct a finer
partitioning for $\mathbf{H}$ and $\mathbf{K}$ to which one can apply the
same strategy as above to obtain a \sumexpr-expression for
$\mathbf{L}|_{\tau\times\sigma}$. In particular, the conversions to the
finer partitioning can be achieved without introducing any additional
errors.
However, we will show in Section~\ref{sec:cost} that we require the
following more sophisticated strategy to obtain an $\mathcal{H}$-matrix
multiplication in almost linear complexity.
\end{remark}

\subsection{Restrictions of \sumexpr-expressions}
The main difference between the \sumexpr-expressions for $\mathbf{L}$ and
its restriction $\mathbf{L}|_{\tau\times\sigma}$ from the previous subsection
is the presence of
$\mathcal{S}_\mathcal{R}(\tau,\sigma)$.
Given a block-cluster $\tau'\times\sigma'\in\children(\tau\times\sigma)$,
it then holds
\begin{align*}
\mathbf{L}|_{\tau'\times\sigma'}={}&
\big(\mathbf{L}|_{\tau\times\sigma}\big)|_{\tau'\times\sigma'}\\
={}&
\mathcal{S}(\tau,\sigma)|_{\tau'\times\sigma'}\\
={}&
\mathcal{S}_\mathcal{R}(\tau,\sigma)|_{\tau'\times\sigma'}+
\mathcal{S}_\mathcal{H}(\tau,\sigma)|_{\tau'\times\sigma'},
\end{align*}
where $\mathcal{S}_\mathcal{H}(\tau,\sigma)|_{\tau'\times\sigma'}$ can be rewritten as
\begin{align*}
\mathcal{S}_\mathcal{H}(\tau,\sigma)|_{\tau'\times\sigma'}
={}&
\sum_{\substack{
\rho\in\mathcal{T}_{\mathcal{I}}\colon\\
\tau\times\rho\in\mathcal{B\setminus F}\\
\rho\times\sigma\in\mathcal{B\setminus F}
}}
\big(\mathbf{H}|_{\tau\times\rho}
\mathbf{K}|_{\rho\times\sigma}\big)\big|
_{\tau'\times\sigma'}\\
={}&
\sum_{\substack{
\rho\in\mathcal{T}_{\mathcal{I}}\colon\\
\tau\times\rho\in\mathcal{B}\setminus\mathcal{F}\\
\rho\times\sigma\in\mathcal{B}\setminus\mathcal{F}
}}
\sum_{\rho'\in\children(\rho)}
\mathbf{H}|_{\tau'\times\rho'}\mathbf{K}|_{\rho'\times\sigma'}\\
={}&
\sum_{\substack{
\rho'\in\mathcal{T}_{\mathcal{I}}\colon\\
\tau'\times\rho'\in\mathcal{B}\\
\rho'\times\sigma'\in\mathcal{B}
}}
\mathbf{H}|_{\tau'\times\rho'}\mathbf{K}|_{\rho'\times\sigma'}.
\end{align*}
Each of the products in the last sum can be treated in the same manner as for the root. Thus,
$\mathcal{S}_\mathcal{H}(\tau,\sigma)|_{\tau'\times\sigma'}$ can be represented as
a \sumexpr-expression
\[
\mathcal{S}_\mathcal{H}(\tau,\sigma)|_{\tau'\times\sigma'}=\mathcal{S}(\tau',\sigma')=\mathcal{S}_\mathcal{R}(\tau',\sigma')+\mathcal{S}_\mathcal{H}(\tau',\sigma'),
\]
where $\mathcal{S}_\mathcal{R}(\tau',\sigma')$ and $\mathcal{S}_\mathcal{H}(\tau',\sigma')$ may be both non-empty.

The restriction of $\mathcal{S}_\mathcal{R}(\tau,\sigma)$ to $\tau'\times\sigma'$ can
be accomplished by the restriction of the corresponding low-rank matrices.
In actual implementations, the restriction of the low-rank matrices can be
realized by index-shifts, and thus without further arithmetic operations.

Since the sum of
two \sumexpr-expressions is again a \sumexpr-expression, we have shown that
each matrix block $\mathbf{L}|_{\tau\times\sigma}$ can be represented as a
\sumexpr-expression.
 A recursive algorithm for their construction is listed
in Algorithm~\ref{alg:dosumexpr}. When the algorithm is initialized with
$\mathcal{S}(\mathcal{I},\mathcal{I})=\mathcal{S}_\mathcal{H}(\mathcal{I},\mathcal{I})
=\mathbf{HK}$ and is applied recursively to all elements of
$\mathcal{B}$, it creates a \sumexpr-expression for each block-cluster in
$\mathcal{B}$, in particular for all leaves of the farfield and the nearfield.
\begin{algorithm}
\caption{\label{alg:dosumexpr}Given $\mathcal{S}(\tau,\sigma)$, construct $\mathcal{S}(\tau',\sigma')$ for $\tau'\times\sigma'\in\children(\tau\times\sigma$).}
\begin{algorithmic}
\Function{$\mathcal{S}(\tau',\sigma')=$\,restrict}{$\mathcal{S}(\tau,\sigma)$, $\tau'\times\sigma'$}
\State{Set $\mathcal{S}_\mathcal{R}(\tau',\sigma')=\sum_i \big(\mathbf{A}_i\mathbf{B}_i^{\intercal}\big)\big|_{\tau'\times\sigma'}$,
given $\mathcal{S}_\mathcal{R}(\tau,\sigma)=\sum_i \mathbf{A}_i\mathbf{B}_i^{\intercal}$}
\State{Set $\mathcal{S}_\mathcal{H}(\tau',\sigma')$ as empty}
\For{$\mathbf{H}|_{\tau\times\rho}
	\mathbf{K}|_{\rho\times\sigma}\in\mathcal{S}_\mathcal{H}(\tau,\sigma)$}
\For{$\rho'\in\children(\rho)$}
\If{$\tau'\times\rho'\in\mathcal{F}$ or $\rho'\times\sigma'\in\mathcal{F}$}
\State{Compute low-rank matrix
$\mathbf{A}\mathbf{B}^{\intercal}=\mathbf{H}|_{\tau'\times\rho'}\mathbf{K}|_{\rho'\times\sigma'}$}
\State{Set $\mathcal{S}_\mathcal{R}(\tau',\sigma')=\mathcal{S}_\mathcal{R}(\tau',\sigma')+
\mathbf{A}\mathbf{B}^{\intercal}$}
\Else
\State{Set $\mathcal{S}_\mathcal{H}(\tau',\sigma')=\mathcal{S}_\mathcal{H}(\tau',\sigma')+
\mathbf{H}|_{\tau'\times\rho'}\mathbf{K}|_{\rho'\times\sigma'}$}
\EndIf
\EndFor
\EndFor
\State{Set $\mathcal{S}(\tau',\sigma')=\mathcal{S}_\mathcal{R}(\tau',\sigma')
+\mathcal{S}_\mathcal{H}(\tau',\sigma')$}
\EndFunction
\end{algorithmic}
\end{algorithm}

\subsection{$\mathcal{H}$-matrix multiplication using \sumexpr-expressions}
The algorithm from the previous section provides us, when applied
recursively, with exact
representations in terms of \sumexpr-expressions for each matrix block
$\mathbf{L}|_{\tau\times\sigma}$ for all block-clusters
$\tau\times\sigma\in\mathcal{B}$. In order to compute an $\mathcal{H}$-matrix
approximation of $\mathbf{L}$, we only have to convert these \sumexpr-expressions
to dense matrices or low-rank matrices. This leads to the $\mathcal{H}$-matrix
multiplication algorithm given in Algorithm~\ref{alg:theHmult}, which is
initialized with
$\mathcal{S}(\mathcal{I},\mathcal{I})=\mathcal{S}_\mathcal{H}(\mathcal{I},\mathcal{I})
=\mathbf{HK}$. The \textsc{evaluate}()-routine in the algorithm computes the
representation of the corresponding \sumexpr-expression as a full matrix,
whereas the $\mathfrak{T}()$-routine is a generic low-rank approximation or
\emph{truncation} operator.
\begin{algorithm}
%\caption{\label{alg:theHmult}Compute $\mathcal{H}$-matrix approximation
%$\mathbf{L}|_{\tau\times\sigma}$ of $\mathcal{S}(\tau,\sigma)$,
%$\tau\times\sigma\in\mathcal{B}$}
\caption{\label{alg:theHmult}$\mathcal{H}$-matrix product: Compute
$\mathbf{L}|_{\tau\times\sigma}=(\mathbf{H}\mathbf{K})|_{\tau\times\sigma}$ from $\mathcal{S}(\tau,\sigma)$}
\begin{algorithmic}
\Function{$\mathbf{L}|_{\tau\times\sigma}=\mathcal{H}$mult}
{$\mathcal{S}(\tau,\sigma)$}
\If{$\tau,\sigma$ is not a leaf}
\Comment{$\mathbf{L}|_{\tau\times\sigma}$ is an $\mathcal{H}$-matrix}
\For{$\tau'\times\sigma'\in\children(\tau\times\sigma)$}
\State{Set $\mathcal{S}(\tau',\sigma')=$ \Call{restrict}{$\mathcal{S}(\tau,\sigma)$, $\tau'\times\sigma'$}}
\State{$\mathbf{L}|_{\tau'\times\sigma'}=$ \Call{$\mathcal{H}$mult}{$\mathcal{S}(\tau',\sigma')$}}
\EndFor
\Else
\If{$\tau\times\sigma\in\mathcal{F}$}
\Comment{$\mathbf{L}|_{\tau\times\sigma}$ is low-rank}
\State{$\mathbf{L}|_{\tau\times\sigma}=\mathfrak{T}(\mathcal{S}(\tau,\sigma))$}
\Else
\Comment{$\mathbf{L}|_{\tau\times\sigma}$ is a dense}
\State{$\mathbf{L}|_{\tau\times\sigma}=$ \Call{evaluate}{$\mathcal{S}(\tau,\sigma)$}}
\EndIf
\EndIf
\EndFunction
\end{algorithmic}
\end{algorithm}

The algorithm can be seen as a general framework for the multiplication of
$\mathcal{H}$-matrices, although several special cases are stated
in the literature
using different algorithmic implementations.

\subsubsection{No truncation}
In principle, the truncation operator could act as an identity. For this
implementation, it was shown in \cite{GH03} that the rank $\tilde{k}$ of low-rank
matrices in the product is bounded by
\[
\tilde{k}\leq\Cid\Csp(p+1)k,
\]
where the constant $\Cid = \Cid(\mathcal{B})$ is given by
\begin{align*}
C_{\operatorname{id}}(\tau\times\sigma)
\isdef{}&
\#\{\tau'\times\sigma'\colon\tau'\in\operatorname{successor}(\tau),\sigma'\in\operatorname{successor}(\sigma)~\text{such that}\\
&\qquad\qquad\qquad\quad
\text{there exists}~\rho'\in\mathcal{T}_{\mathcal{I}}~\text{such that}\\
&\qquad\qquad\qquad\quad
\tau'\times\rho'\in\mathcal{B},\rho'\times\sigma'\in\mathcal{B}\},\\
C_{\operatorname{id}}(\mathcal{B})\isdef{}&\max_{\tau\times\sigma\in\mathcal{L}(\mathcal{B})}
C_{\operatorname{id}}(\tau\times\sigma).
\end{align*}

Although the rank of the product is bounded from above, the constants in the
bound might be large. Hence one is interested in truncating the low-rank
matrices to lower rank in a best possible way. Depending on the employed
truncation operator $\mathfrak{T}$, different implementations of the
multiplication evolve.

\subsubsection{Truncation with a single low-rank SVD}
Traditionally, the used truncation operators are based on the singular value
decomposition, from which several implementations have evolved. The most
accurate implementation is given by computing the exact product in low-rank
format and truncating it to a lower rank by using a singular value
decomposition for low-rank matrices as given in Algorithm~\ref{alg:lrsvd}.
\begin{algorithm}[tbp]
	\caption{\label{alg:lrsvd}SVD of a low-rank matrix
		$\mathbf{L}\mathbf{R}^\intercal$, see \cite[Algorithm 2.17]{Hac15}}
	\begin{algorithmic}
		\Function{${\bf U\Sigma V}^{\intercal}$=LowRankSVD}{$\mathbf{L}\mathbf{R}^\intercal$}
		\State ${\bf Q_LR_L}{}=$ QR-decomposition of ${\bf L}$, ${\bf Q_L}\in\mathbb{R} ^{\tau\times\tilde{k}}$, ${\bf R_L}\in\mathbb{R} ^{\tilde{k}\times\tilde{k}}$
		\State ${\bf Q_RR_R}{}=$ QR-decomposition of ${\bf R}$, ${\bf Q_R}\in\mathbb{R} ^{\sigma\times\tilde{k}}$, ${\bf R_R}\in\mathbb{R} ^{\tilde{k}\times\tilde{k}}$
		\State ${\bf \tilde{U}\Sigma\tilde{V}}^{\intercal}{}= \operatorname{SVD}({\bf R_LR_R}^{\intercal})$
		\State ${\bf U}={\bf Q_L\tilde{U}}$
		\State ${\bf V}={\bf Q_R\tilde{V}}$
		\EndFunction
	\end{algorithmic}
\end{algorithm}
The number of operations for
the $\mathcal{H}$-matrix multiplication, assuming $n_{\min}\leq k$, is then
bounded by
\[
43\Cid^3\Csp^3k^3(p+1)^3\max\{\#\mathcal{I},\#\mathcal{F}+\#\mathcal{N}\},
\]
see \cite{GH03}.
However, it turns out that, for more complex block-cluster trees of practical
relevance, the numerical effort for this implementation of the multiplication is
quite high.

\subsubsection{Truncation with multiple low-rank SVDs --- fast truncation}\label{sec:fasttrunc}
Therefore, one may replace the the above truncation by the fast truncation
of low-rank matrices, which aims at accelerating the truncation of sums of
low-rank matrices by allowing a larger error margin. The basic idea is that in
many cases
\[
\mathbf{M}_n\mathbf{N}_n^{\intercal}=
\mathfrak{T}\bigg(\sum _{i=1}^n\mathbf{A}_i\mathbf{B}_i^{\intercal}\bigg)
\]
can be sufficiently well be approximated by computing
\begin{align*}
\mathbf{M}_2\mathbf{N}_2^{\intercal}={}&\mathfrak{T}\big(\mathbf{A}_1\mathbf{B}_1^{\intercal}+
\mathbf{A}_2\mathbf{B}_2^{\intercal}\big)\\
\mathbf{M}_i\mathbf{N}_i^{\intercal}={}&\mathfrak{T}\big(\mathbf{M}_{i-1}\mathbf{N}_{i-1}^{\intercal}+
\mathbf{A}_i\mathbf{B}_i^{\intercal}\big),\qquad i=3,\ldots,n.
\end{align*}
If the fast truncation is used as a truncation operator, the number of operations
for the $\mathcal{H}$-matrix multiplication, assuming $n_{\min}\leq k$, is
bounded by
\[
56\Csp^2\max\{\Cid,\Csp\}k^2(p+1)^2\#\mathcal{I}+184\Csp\Cid k^3(p+1)(\#\mathcal{F}+\#\mathcal{N}),
\]
see \cite{GH03}.
Numerical experiments confirm that the $\mathcal{H}$-matrix multiplication using
the fast truncation is indeed faster, but also slightly less accurate than the
previous version of the multiplication.

\subsubsection{Truncation with accumulated updates}
Recently, in \cite{Boe17}, a new truncation operator was introduced, to which
we will refer to as truncation with accumulated updates. Therefore,
we slightly modify the definition of the \sumexpr-expression and denote the new
object by $\mathcal{S}^{a}$.
\begin{definition}
For a given block-cluster $\tau\times\sigma\in\mathcal{B}$, we say that the sum
of a low-rank matrix
$\mathbf{A}\mathbf{B}^{\intercal}\in\mathcal{R}(\tau\times\sigma,k)$ and a
\sumexpr-expression of $\mathcal{H}$-matrices $\mathcal{S}_\mathcal{H}(\tau,\sigma)$ is
a \sumexpr-\emph{expression with accumulated updates} and write
\[
\mathcal{S}^{a}(\tau,\sigma)=\mathbf{A}\mathbf{B}^{\intercal}+\mathcal{S}_\mathcal{H}(\tau,\sigma).
\]
In particular, we write
$\mathcal{S}_\mathcal{R}^{a}(\tau,\sigma)=\mathbf{A}\mathbf{B}^{\intercal}$ and, for a
second low-rank matrix
$\tilde{\mathbf{A}}\tilde{\mathbf{B}}^{\intercal}\in\mathcal{R}(\tau\times\sigma,k)$,
we define the sum of these expressions with a low-rank-matrix as
\begin{align*}
\mathcal{S}_\mathcal{R}^{a}(\tau,\sigma)+\tilde{\mathbf{A}}\tilde{\mathbf{B}}^{\intercal}={}&\mathfrak{T}\big(\mathbf{A}\mathbf{B}^{\intercal}+\tilde{\mathbf{A}}\tilde{\mathbf{B}}^{\intercal}\big)\\
\mathcal{S}^{a}(\tau,\sigma)+\tilde{\mathbf{A}}\tilde{\mathbf{B}}^{\intercal}={}&\mathfrak{T}\big(\mathbf{A}\mathbf{B}^{\intercal}+\tilde{\mathbf{A}}\tilde{\mathbf{B}}^{\intercal}\big)+\mathcal{S}_\mathcal{H}(\tau,\sigma),
\end{align*}
i.e., instead of adding the new low-rank matrix to the list of low-rank matrices
in $\mathcal{S}^{a}(\tau,\sigma)$, we perform an addition of low-rank matrices
with subsequent truncation.
\end{definition}
Obviously, every \sumexpr-expression with accumulated updates is also a
\sumexpr-expression in the sense of Definition~\ref{def:sumexpr}. The key point
is that the addition with low-rank matrices is treated differently. By replacing
the \sumexpr-expressions in Algorithm~\ref{alg:theHmult} by \sumexpr-expressions
with accumulated updates, we obtain the $\mathcal{H}$-matrix multiplication as
stated in \cite{Boe17}.
The number of operations for this algorithm is bounded by
\[
3C_{\operatorname{mm}}\Csp^2k^2(p+1)^2\#\mathcal{I}.
\]
The constant $C_{\operatorname{mm}}$ consists of several other constants which
exceed the scope of this article and we refer to \cite{Boe17} for more
details. However, the numerical experiments in \cite{Boe17} indicate that the
truncation operator with accumulated updates is faster than the fast truncation
operator.

\bigskip
An issue of both, the fast truncation and the truncation with accumulated
updates, is a situation where the product of $\mathcal{H}$-matrices has to be
converted into a low-rank matrix. Here, both implementations rely on a
\emph{hierarchical approximation} of the product of the $\mathcal{H}$-matrices.
That is, the product is computed in $\mathcal{H}$-matrix format and then,
starting from the leaves, recursively converted into low-rank format, which
is a time-consuming task and requires several intermediate truncation steps.
This introduces additional truncation errors, although the truncation with
accumulated updates somehow reduces the number of conversions.

A slightly different approach than the fast truncation with hierarchical
approximation was proposed in \cite{DHP15}.
There, the $\mathcal{H}$-matrix products have been truncated to low-rank matrices
using an iterative eigensolver based on matrix-vector multiplications
before applying the fast truncation operator. The numerical experiments
prove this approach to be computationally efficient, while providing even
a best approximation to the product of the $\mathcal{H}$-matrices
in low-rank format.

We summarize by remarking that all of the common $\mathcal{H}$-matrix
multiplication algorithms are variants of Algorithm~\ref{alg:theHmult},
employing different truncation operators. Therefore, in order to improve the
accuracy and the speed of the $\mathcal{H}$-matrix multiplication, efficient
and accurate truncation operators have to be used.

Since approaches based on the singular value decomposition of dense or
low-rank matrices
have proven to be less promising, we focus in the following on low-rank
approximation methods based on matrix-vector multiplications. The idea behind
this approach is that the multiplication of a \sumexpr-expression
$\mathcal{S}(\tau,\sigma)$ with a vector $\mathbf{v}$ of length $\#\sigma$
can be computed efficiently by
\[
\mathcal{S}(\tau,\sigma)\mathbf{v}
=
\sum_{j\in\mathcal{J}_{\mathcal{R}}}\mathbf{A}_j\big(\mathbf{B}_j^{\intercal}\mathbf{v}\big)+
\sum_{j\in\mathcal{J}_{\mathcal{H}}}\mathbf{H}_j\big(\mathbf{K}_j\mathbf{v}\big).
\]
Although this idea has already been mentioned in \cite{DHP15}, the used
eigensolver in \cite{DHP15} seemed to be less favourable for this task.
In the next section, we will discuss several approaches to compute low-rank
approximations to \sumexpr-expressions using matrix-vector multiplications.
In particular, we will discuss
\emph{adaptive} algorithms, which compute low-rank approximations to
\sumexpr-expressions up to a prescribed error tolerance. The adaptive
algorithms will allow us to compute the best-approximation of
$\mathcal{H}$-matrix products.

%!TEX root = paper.tex
\section{Low-rank approximation schemes}\label{sec:approx}

In addition to the well known hierarchical approximation for the approximation
of \(\mathcal{H}\)-matrices by low-rank matrices, we consider here three different
schemes for the low-rank approximation of a given matrix. All of them
can be implemented in terms of elementary matrix-vector products and
are therefore well suited for the use in our new \(\mathcal{H}\)-matrix multiplication.
In what follows, let \({\bf A}\isdef{\bf H}|_{\tau\times\sigma}\in\mathbb{R}^{m\times n}\), $m=\#\tau$, $n=\#\sigma$, always denote
a target matrix block, which might be implicitly given in terms of a \texttt{sum}-expression $\mathcal{S}(\tau,\sigma)$.\\

\subsection{Adaptive cross approximation}\label{sec:ACA}
In the context of boundary element methods, the \emph{adaptive cross
approximation} (ACA), see \cite{Beb}, is frequently used to find
$\mathcal{H}$-matrix approximations to system matrices. However,
one can prove, see \cite{DHS17}, that the same idea can also be
applied to the product of pseudodifferential operators. Since
the $\mathcal{H}$-matrix multiplication can be seen as a discrete
analogon to the multiplication of pseudodifferential operators, we
may use ACA to find the low-rank approximations for the admissible
matrix blocks. Concretely, we approximate
${\bf A}={\bf H}|_{\tau\times\sigma}$ by a partially pivoted Gaussian
elimination, see \cite{Beb} for further details.
To this end, we define the vectors 
\({\boldsymbol{\ell}}_r\in\mathbb{R}^m\) and \({\bf u}_r\in\mathbb{R}^n\)
by the iterative scheme shown in Algorithm~\ref{alg:ACA}, where \({\bf A}=[a_{i,j}]_{i,j}\) is the matrix-block under consideration.

\begin{algorithm}[htb]
	\caption{Adaptive cross approximation (ACA)}
	\label{alg:ACA}
	\begin{algorithmic}
		\For{$r =1,2,\ldots$}
		\State Choose the element in \((i_r,j_r)\)-position of the Schur compolement as pivot
		\State $\hat{\bf u}_r=[a_{i_r,j}]_{j=1}^n-\sum_{j=1}^{r-1}
	[{\boldsymbol\ell}_j]_{i_r}{\bf u}_j$
	\State \({\bf u}_r=\hat{\bf u}_r/[\hat{\bf u}_r]_{j_r}\)
	\State ${\boldsymbol{\ell}}_r=[a_{i,j_r}]_{i=1}^m-
		\sum_{i=1}^{r-1}[u_i]_{j_r}{\boldsymbol{\ell}_i}$
		\EndFor
		\State Set \({\bf L}_r\isdef[{\boldsymbol\ell}_1,\ldots,{\boldsymbol\ell}_r]\) and \({\bf U}_r\isdef[{\bf u}_1,\ldots,{\bf u}_r]^\intercal\)
	\end{algorithmic}
\end{algorithm}

A suitable criterion that guarantees the convergence of the algorithm
is to choose the pivot element located in the \((i_{r},j_{r})\)-position
as the maximum element in modulus of the remainder 
\({\bf A}-{\bf L}_{r-1}{\bf U}_{r-1}\).
Unfortunately, this would compromise the overall cost of the approximation.
Therefore, we resort to another pivoting strategy which is sufficient in
most cases: we choose \(j_r\) such that \([\hat{\bf u}_r]_{j_r}\)
is the largest element in modulus of the row \(\hat{{\bf u}}_r\).

Obviously, the cost for the computation of the 
rank-$k$-approximation ${\bf L}_{k}{\bf U}_{k}$ to the block 
${\bf A}$ is \(\mathcal{O}\big(k^2(m+n)\big)\) and
the storage cost is \(\mathcal{O}\big(k(m+n)\big)\). 
In addition, if \({\bf A}\) is given via a \texttt{sum}-expression,
we have to compute \({\bf e}_{i_m}^\intercal{\bf A}\)
and \({\bf A}{\bf e}_{j_m}\) in each step, where \({\bf e}_i\)
denots the \(i\)-th unit vector, in order to retrieve the row and column under consideration.
The respective computational cost for the multiplication of a \sumexpr-expression with a
vector is estimated in Lemma~\ref{lem:compsumexprmv}.

\subsection{Lanczos bidiagonalization}\label{sec:lanczos}
The second algorithm we consider for compressing a given matrix block is based on the Lanczos bidiagonalization (BiLanczos),
see Algorithm~\ref{alg:Bidiagonalization}.
This procedure is equivalent to the tridiagonalization of the corresponding, symmetric Jordan-Wielandt matrix
\[
\begin{bmatrix}{\bf 0} & {\bf A} \\ {\bf A}^\intercal & {\bf 0}\end{bmatrix},
\]
cf.\ \cite{golub2012}.

\begin{algorithm}[htb]
	\caption{Golub-Kahan-Lanczos bidiagonalization}
	\label{alg:Bidiagonalization}
	\begin{algorithmic}
	\State Choose a random vector \({\bf w}_1\) with \(\|{\bf w}_1\|=1\) and set \({\bf q}_0={\bf 0},\beta_0=0\)
		\For{$r =1,2,\ldots$}
			\State\({\bf q}_r = {\bf A}{\bf w}_r-\beta_{r-1}{\bf q}_{r-1}\)
			\State \(\alpha_r = \|{\bf q}_r\|_2\)
			\State \({\bf q}_r = {\bf q}_r / \alpha_r\)
			\State \({\bf w}_{r+1}={\bf A}^\intercal{\bf q}_r-\alpha_r{\bf w}_r\)
			\State \(\beta_r=\|{\bf w}_{r+1}\|_2\)
			\State \({\bf w}_{r+1}={\bf w}_{r+1}/\beta_r\)
		\EndFor
	\end{algorithmic}
\end{algorithm}

The algorithm leads to a decomposition of the given matrix block according to
\[
{\bf Q}^\intercal{\bf A}{\bf W}={\bf B}\isdef\begin{bmatrix}
\alpha_1 & \beta_1 & & & \\
 & \alpha_2 & \beta 2 & & \\
 & & \ddots & \ddots & \\
 & & & \alpha_{n-1} & \beta_{n-1}\\
 & & & & \alpha_n
\end{bmatrix}
\]
with orthogonal matrices \({\bf Q}^\intercal{\bf Q}={\bf I}_m\) and \({\bf W}^\intercal{\bf W}={\bf I}_n\), cf.\ \cite{golub2012}.
Note that although the algorithm yields orthogonal vectors \({\bf q}_r\) and \({\bf w}_r\) by construction, we have
to perform additional reorthogonalization steps to obtain numerical stability. Since the algorithm, like ACA,
only depends on matrix-vector multiplications, it is well suited to compress a given block \({\bf A}\).

Truncating the algorithm after \(k\) steps
results in a low-rank approximation 
\[
{\bf A}\approx {\bf Q}_k{\bf B}_k{\bf W}_k^\intercal={\bf U}_k\begin{bmatrix}
\alpha_1 & \beta_1 & & & \\
 & \alpha_2 & \beta_2 & & \\
 & & \ddots & \ddots & \\
 & & & \alpha_{k-1} & \beta_{k-1}\\
 & & & & \alpha_k
\end{bmatrix}{\bf V}_k^\intercal .
\]
It is then easy to compute a singular value decomposition \({\bf B}_k=\tilde{\bf U}{\bf S}\tilde{\bf V}^\intercal\) and to obtain the
separable decomposition 
\[{\bf A}\approx ({\bf Q}_k\tilde{\bf U}{\bf S})({\bf W}_k\tilde{\bf V})^\intercal={\bf L}_k{\bf U}_k.\]
As in the ACA case, the algorithm only requires two matrix-vector products with the block \({\bf A}\) in each step.

\subsection{Randomized low-rank approximation}\label{sec:randomized}
The third algorithm we consider for finding a low-rank approximation to \({\bf A}\)
is based on successive multiplication with
Gaussian random vectors. The algorithm can be motivated as follows, cf.\ \cite{HMT11}: 
Let \({\bf y}_i={\bf A}{\boldsymbol\omega}_i\) for \(i=1,\ldots,r\), where \({\boldsymbol\omega}_1,\ldots,{\boldsymbol\omega}_r\in\mathbb{R}^n\)
are independently drawn Gaussian random vectors. The collection of these random vectors is very likely to be linearly independent, whereas
it is very unlikely that any linear combination of them falls in the null space of \({\bf A}\). As a consequence, the collection \({\bf y}_1,\ldots,{\bf y}_r\)
is linearly independent and spans the range of \({\bf A}\). Thus, by orthogonalizing \([{\bf y}_1,\ldots,{\bf y}_r]={\bf L}_r{\bf R}\), with an orthogonal
matrix \({\bf L}_r\in\mathbb{R}^{m\times r}\), we obtain \({\bf A}\approx{\bf L}_r{\bf U}_r\) with \({\bf U}_r={\bf L}_r^\intercal{\bf A}\), see Algorithm~\ref{alg:randomized}. 
Employing oversampling,
the quality of the approximation can be increased even further. For all the details, we refer to \cite{HMT11}.
\begin{algorithm}[htb]
	\caption{Randomized low-rank approximation}
	\label{alg:randomized}
	\begin{algorithmic}
	\State Set \({\bf L}_0 = [\,]\)
		\For{$r =1,2,\ldots$}
		\State Generate a Gaussian random vector \({\boldsymbol\omega}\in\mathbb{R}^n\)
		\State ${\boldsymbol\ell}_r = ({\bf I}-{\bf L}_{r-1}{\bf L}_{r-1}^\intercal){\bf A}{\boldsymbol\omega}$
		\State ${\boldsymbol\ell}_r = {\boldsymbol\ell}_r$ / $\|{\boldsymbol\ell}_r\|_2$
		\State ${\bf L}_r=[{\bf L}_{r-1}, {\boldsymbol\ell}_r]$
		\State ${\bf U}_r={\bf L}_r^\intercal{\bf A}$
		\EndFor
	\end{algorithmic}
\end{algorithm}
As the previous two algorithms, the randomized low-rank approximation requires only two matrix-vector multiplications with \({\bf A}\) in each step.

In contrast to the other two presented compression schemes, the randomized approximation allows for a blocked version
in a straightforward manner, 
where instead of a single Gaussian random vector, a Gaussian matrix can be used to approximate the range of \({\bf A}\).
Although there also exist block versions of the ACA and the BiLanczos, as well, they are know to be numerically very unstable.

Note that there exist probabilistic error estimates for the approximation of a given matrix by Algorithm~\ref{alg:randomized}, cf.\ \cite{HMT11}. 
Unfortunately, these error estimates are with respect to the spectral norm and therefore only give insights on the largest singular value 
of the remainder. To have control on the actual approximation quality of a certain matrix block, we are thus rather interest in an error estimate
with respect to the Frobenius norm. To that end, we propose a different, adaptive criterion which estimates the error with respect to the Frobenius norm.

\subsection{Adaptive stopping criterion}
In our implementation, the proposed compression schemes rely on the following 
well known adaptive stopping criterion, which
aims at reflecting the approximation error with respect to the Frobenius norm. 
We terminate the approximation if the criterion
\begin{equation}\label{eq:truncaca}
%===============================================================================
\|{\boldsymbol\ell}_{k+1}\|_2\|{\bf u}_{k+1}\|_2
\leq\varepsilon\|{{\bf L}_k{\bf U}_k}\|_F
\end{equation}
is met for some desired accuracy \(\varepsilon>0\).
This criterion can be justified as follows. We assume that the error in each step is reduced by a constant rate
\(0<\vartheta<1\), i.e.,
\[
\|{\bf A}-{\bf L}_{k+1}{\bf U}_{k+1}\|_F\leq\vartheta
\|{\bf A}-{\bf L}_{k}{\bf U}_{k}\|_F.
\]
Then, there holds
\begin{align*}
\|{\boldsymbol\ell}_{k+1}\|_2\|{\bf u}_{k+1}\|_2&=
\|{\bf L}_{k+1}{\bf U}_{k+1}-{\bf L}_{k}{\bf U}_{k}\|_F\\
&\leq\|{\bf A}-{\bf L}_{k+1}{\bf U}_{k+1}\|_F+
\|{\bf A}-{\bf L}_{k}{\bf U}_{k}\|_F\\
&\leq (1+\vartheta)\|{\bf A}-{\bf L}_{k}{\bf U}_{k}\|_F
\end{align*}
and, vice versa,
\begin{align*}
\|{\bf L}_{k+1}{\bf U}_{k+1}-{\bf L}_{k}{\bf U}_{k}\|_F&\geq
\|{\bf A}-{\bf L}_{k}{\bf U}_{k}\|_F
-\|{\bf A}-{\bf L}_{k+1}{\bf U}_{k+1}\|_F\\
&\geq (1-\vartheta)\|{\bf A}-{\bf L}_{k}{\bf U}_{k}\|_F.
\end{align*}
Therefore, the approximation error is proportional to the norm
\[
\|{\boldsymbol\ell}_{k+1}{\bf u}_{k+1}\|_F=\|{\boldsymbol\ell}_{k+1}\|_2\|{\bf u}_{k+1}\|_2
\]
of the update vectors, i.e.,
\[
(1-\vartheta)\|{\bf A}-{\bf L}_{k}{\bf U}_{k}\|_F
\leq\|{\boldsymbol\ell}_{k+1}\|_2\|{\bf u}_{k+1}\|_2\leq(1+\vartheta)
\|{\bf A}-{\bf L}_{k}{\bf U}_{k}\|_F.
\]
Thus, together with \eqref{eq:truncaca}, we can guarantee a 
relative error bound
\begin{equation}\label{eq:ACAblockerror}
%===============================================================================
\|{\bf A}-{\bf L}_{k}{\bf U}_{k}\|_F\leq\frac{\varepsilon}{1-\vartheta}\|{{\bf L}_k{\bf U}_k}\|_F\leq\frac{\varepsilon}{1-\vartheta}\|{\bf A}\|_F.
\end{equation}

Based on this blockwise error estimate, it is straightforward to assess the
overall error for the approximation of a given $\mathcal{H}$-matrix.

\begin{theorem}
Let ${\bf H}$ be the uncompressed matrix and 
$\tilde{\bf H}$ be the matrix which is compressed by
one of the aforementioned compression schemes.
Then, with respect to the 
Frobenius norm, there holds the error estimate
\[
  \|{\bf H}-\tilde{\bf H}\|_F\lesssim\varepsilon\|{\bf H}\|_F
\]
provided that the blockwise error satisfies \eqref{eq:ACAblockerror}.
\end{theorem}

\begin{proof}
In view of \eqref{eq:ACAblockerror}, we have
\begin{align*}
  \|{\bf H}-\tilde{\bf H}\|_F^2
  	&=\sum_{\tau\times\sigma\in\mathcal{F}}
		\big\|{\bf H}|_{\tau\times\sigma}-\tilde{\bf H}|_{\tau\times\sigma}\big\|_F^2\\
	&\lesssim \sum_{\tau\times\sigma\in\mathcal{F}}
		\|{\bf H}|_{\tau\times\sigma}\|_F^2\\
	&=\varepsilon^2\|{\bf H}\|_F^2.
\end{align*}
Taking square roots on both sides yields the assertion.
\end{proof}

\subsection{Fixed rank approximation}
The traditional $\mathcal{H}$-matrix multiplication is based on low-rank approximations
to a fixed a-priori prescribed rank. We will therefore shortly comment on the fixed rank
versions of the introduced algorithms which we will later also use in the numerical
experiments. Since ACA and the BiLanczos algorithm are intrinsically iterative methods,
we stop the iteration whenever the prescribed rank is reached. For the randomized low-rank
approximation we may use a single iteration of its blocked variant. The corresponding
algorithm featuring also additional $q\in\mathbb{N}_0$ subspace iterations to increase
accuracies is listed in Algorithm~\ref{alg:randomizedfixed}, cp.\ \cite{HMT11}.
\begin{algorithm}[htb]
	\caption{Randomized rank-$k$ approximation with subspace iterations}
	\label{alg:randomizedfixed}
	\begin{algorithmic}
		\State Generate a Gaussian random matrix \({\boldsymbol\Omega}\in\mathbb{R}^{n\times k}\)
		\State $\mathbf{L} = {\bf A}{\boldsymbol\Omega}$
		\State Orthonormalize columns of $\mathbf{L}$
		\For{$\ell=1,\ldots,q$}
		\State $\mathbf{U} = {\bf A}^\intercal\mathbf{L}$
		\State Orthonormalize columns of $\mathbf{U}$
		\State $\mathbf{L} = {\bf A}\mathbf{U}$
		\State Orthonormalize columns of $\mathbf{L}$
		\EndFor
		\State $\mathbf{U} = {\bf A}^\intercal\mathbf{L}$
	\end{algorithmic}
\end{algorithm}
For practical purposes, when the singular values of $\mathbf{A}$ decay sufficiently fast,
a value of $q=1$ is usually sufficient. We will therefore use $q=1$ in the numerical
experiments. For a detailed discussion on the number of subspace
iterations and comments on \emph{oversampling}, i.e., sampling at a higher rank with a
subsequent truncation, we refer to \cite{HMT11}.

%!TEX root = paper.tex
\section{Cost of the \(\mathcal{H}\)-matrix multiplication}
\label{sec:cost}
The following section is dedicated to the cost analysis of the
$\mathcal{H}$-matrix multiplication as introduced in
Algorithm~\ref{alg:theHmult}. We first estimate the
cost for the computation of the \sumexpr-expressions and
then proceed by analyzing the multiplication of a \sumexpr-expression
with a vector. Having these estimates at our disposal, the main theorem of this
section confirms that the cost of the $\mathcal{H}$-matrix multiplication
scales almost linearly with the cardinality of the index set $\mathcal{I}$.

\begin{lemma}\label{lem:compsumexprrec}
Given a block-cluster $\tau\times\sigma\in\mathcal{B}$ with
\sumexpr-expression $\mathcal{S}(\tau,\sigma)$, the
\sumexpr-expression $\mathcal{S}(\tau',\sigma')$ for any
block-cluster
$\tau'\times\sigma'\in\children(\tau\times\sigma)$
can be computed in at most
\begin{align*}
N_{\mathrm{update},\mathcal{S}}(\tau',\sigma')\leq
\sum_{\substack{
\rho'\in\mathcal{T}_{\mathcal{I}}\colon\\
\tau'\times\rho'\in\mathcal{B\setminus F}\\
\rho'\times\sigma'\in\mathcal{F}
}}
kN_{\mathcal{H}\cdot v}(\tau'\times\rho',k)+{}&
\sum_{\substack{
\rho'\in\mathcal{T}_{\mathcal{I}}\colon\\
\tau'\times\rho'\in\mathcal{F}\\
\rho'\times\sigma'\in\mathcal{B\setminus F}
}}
kN_{\mathcal{H}\cdot v}(\rho'\times\sigma',k)\\
+{}&
\sum_{\substack{
\rho'\in\mathcal{T}_{\mathcal{I}}\colon\\
\tau'\times\rho'\in\mathcal{F}\\
\rho'\times\sigma'\in\mathcal{F}
}}
k\big(\#\rho'+\min\{\#\tau',\#\sigma'\}
\big)
\end{align*}
operations.
\end{lemma}
\begin{proof}
We start with recalling that $\mathcal{S}(\tau',\sigma')$ is
recursively given as
\begin{align*}
\mathcal{S}(\tau',\sigma')={}&
\mathcal{S}(\tau,\sigma)|_{\tau'\times\sigma'}\\
={}&
\mathcal{S}_\mathcal{R}(\tau,\sigma)|_{\tau'\times\sigma'}+
\mathcal{S}_\mathcal{H}(\tau,\sigma)|_{\tau'\times\sigma'}\\
={}&
\mathcal{S}_\mathcal{R}(\tau,\sigma)|_{\tau'\times\sigma'}+
\sum_{\substack{
\rho'\in\mathcal{T}_{\mathcal{I}}\colon\\
\tau'\times\rho'\in\mathcal{B}\\
\rho'\times\sigma'\in\mathcal{B}
}}
\mathbf{H}|_{\tau'\times\rho'}\mathbf{H}|_{\rho'\times\sigma'},
\numberthis\label{eq:recursion}
\end{align*}
with $\mathcal{S}_\mathcal{R}(\mathcal{I},\mathcal{I})=\emptyset$. Clearly,
the restriction of $\mathcal{S}_\mathcal{R}(\tau,\sigma)$ to
$\tau'\times\sigma'$ comes for free, and we only have to look at
the remaining sum. It can be decomposed
into
\begin{align}
\sum_{\substack{
\rho'\in\mathcal{T}_{\mathcal{I}}\colon\\
\tau'\times\rho'\in\mathcal{B}\\
\rho'\times\sigma'\in\mathcal{B}
}}
\mathbf{H}|_{\tau'\times\rho'}\mathbf{H}|_{\rho'\times\sigma'}
={}&
\sum_{\substack{
\rho'\in\mathcal{T}_{\mathcal{I}}\colon\\
\tau'\times\rho'\in\mathcal{B\setminus F}\\
\rho'\times\sigma'\in\mathcal{B\setminus F}
}}
\mathbf{H}|_{\tau'\times\rho'}\mathbf{H}|_{\rho'\times\sigma'}\label{eq:prodexrp1}\\\
&{}\quad+
\sum_{\substack{
\rho'\in\mathcal{T}_{\mathcal{I}}\colon\\
\tau'\times\rho'\in\mathcal{B\setminus F}\\
\rho'\times\sigma'\in\mathcal{F}
}}
\mathbf{H}|_{\tau'\times\rho'}\mathbf{H}|_{\rho'\times\sigma'}\label{eq:prodexrp2}\\\
&{}\quad+
\sum_{\substack{
\rho'\in\mathcal{T}_{\mathcal{I}}\colon\\
\tau'\times\rho'\in\mathcal{F}\\
\rho'\times\sigma'\in\mathcal{B\setminus F}
}}
\mathbf{H}|_{\tau'\times\rho'}\mathbf{H}|_{\rho'\times\sigma'}\label{eq:prodexrp3}\\
&{}\quad+
\sum_{\substack{
\rho'\in\mathcal{T}_{\mathcal{I}}\colon\\
\tau'\times\rho'\in\mathcal{F}\\
\rho'\times\sigma'\in\mathcal{F}
}}
\mathbf{H}|_{\tau'\times\rho'}\mathbf{H}|_{\rho'\times\sigma'}\label{eq:prodexrp4}.
\end{align}
The products on the right-hand side in \eqref{eq:prodexrp1} are not computed, thus, no
numerical effort has to be made. The products in
\eqref{eq:prodexrp2}, \eqref{eq:prodexrp3}, and
\eqref{eq:prodexrp4} must be computed and stored as low-rank matrices.
The computational effort for this operation is
$kN_{\mathcal{H}\cdot v}(\tau'\times\rho',k)$ for
\eqref{eq:prodexrp2}, $kN_{\mathcal{H}\cdot v}(\rho'\times\sigma',k)$
for \eqref{eq:prodexrp3}, and $k^2(\#\rho'+\min\{\#\tau',\#\sigma'\})$
for \eqref{eq:prodexrp4}. Thus, given $\mathcal{S}(\tau,
\sigma)$, the computational cost to compute $\mathcal{S}(\tau',
\sigma')$ is
\begin{align*}
\sum_{\substack{
\rho'\in\mathcal{T}_{\mathcal{I}}\colon\\
\tau'\times\rho'\in\mathcal{B\setminus F}\\
\rho'\times\sigma'\in\mathcal{F}
}}
kN_{\mathcal{H}\cdot v}(\tau'\times\rho',k)+
&{}
\sum_{\substack{
\rho'\in\mathcal{T}_{\mathcal{I}}\colon\\
\tau'\times\rho'\in\mathcal{F}\\
\rho'\times\sigma'\in\mathcal{B\setminus F}
}}
kN_{\mathcal{H}\cdot v}(\rho'\times\sigma',k)\\
&{}\qquad\qquad+
\sum_{\substack{
\rho'\in\mathcal{T}_{\mathcal{I}}\colon\\
\tau'\times\rho'\in\mathcal{F}\\
\rho'\times\sigma'\in\mathcal{F}
}}
k\big(\#\rho'+\min\{\#\tau',\#\sigma'\}
\big),
\end{align*}
which proves the assertion.
\end{proof}

\begin{lemma}\label{lem:compsumexprtot}
The $\mathcal{H}$-matrix multiplication as given by Algorithm~\ref{alg:theHmult}
requires at most
\[
N_{\mathcal{S}}(\mathcal{B})\leq 16\Csp^3k\max\{k,n_{\min}\}(p+1)^2\#\mathcal{I}
\]
operations for the computation of the \sumexpr-expressions.
\end{lemma}
\begin{proof}
We consider a block-cluster $\tau\times\sigma\in\mathcal{B}\setminus\mathcal{F}$
on level $j$. Then, the numerical effort to compute the \sumexpr-expression
$\mathcal{S}(\tau',\sigma')$ for $\tau'\times\sigma'\in\children(\tau\times\sigma)$
is estimated by Lemma~\ref{lem:compsumexprrec}, if the
\sumexpr-expression $\mathcal{S}(\tau,\sigma)$ is known. Therefore, it is sufficient
to sum over all block-clusters in $\mathcal{B}$ as follows:
\begin{align*}
&\sum_{\tau\times\sigma\in\mathcal{B}}N_{\mathrm{update},\mathcal{S}}(\tau,\sigma)\\
&\qquad\leq
\sum_{\tau\times\sigma\in\mathcal{B}}\bigg(
2\sum_{\substack{
\rho\in\mathcal{T}_{\mathcal{I}}\colon\\
\tau\times\rho\in\mathcal{B\setminus F}\\
\rho\times\sigma\in\mathcal{F}
}}
kN_{\mathcal{H}\cdot v}(\tau\times\rho,k)+
\sum_{\substack{
\rho\in\mathcal{T}_{\mathcal{I}}\colon\\
\tau\times\rho\in\mathcal{F}\\
\rho\times\sigma\in\mathcal{F}
}}
k\big(\#\rho+\min\{\#\tau,\#\sigma\}
\big)
\bigg)\\
&\qquad=
2\sum_{\tau\times\sigma\in\mathcal{B}}
\sum_{\substack{
\rho\in\mathcal{T}_{\mathcal{I}}\colon\\
\tau\times\rho\in\mathcal{B\setminus F}\\
\rho\times\sigma\in\mathcal{F}
}}
kN_{\mathcal{H}\cdot v}(\tau\times\rho,k)+
\sum_{\tau\times\sigma\in\mathcal{B}}
\sum_{\substack{
\rho\in\mathcal{T}_{\mathcal{I}}\colon\\
\tau\times\rho\in\mathcal{F}\\
\rho\times\sigma\in\mathcal{F}
}}
k\big(\#\rho+\min\{\#\tau,\#\sigma\}
\big).
\end{align*}
To estimate the first sum, consider
\begin{align*}
\sum_{\tau\times\sigma\in\mathcal{B}}
\sum_{\substack{
\rho\in\mathcal{T}_{\mathcal{I}}\colon\\
\tau\times\rho\in\mathcal{B\setminus F}\\
\rho\times\sigma\in\mathcal{F}
}}
kN_{\mathcal{H}\cdot v}(\tau\times\rho,k)
\leq{}&
\sum_{\tau\times\sigma\in\mathcal{B}}
\sum_{\substack{
	\rho\in\mathcal{T}_{\mathcal{I}}\colon\\
	\tau\times\rho\in\mathcal{B\setminus F}
}}
kN_{\mathcal{H}\cdot v}(\tau\times\rho,k)\\
\leq{}&
\Csp
\sum_{\tau\times\rho\in\mathcal{B}}
kN_{\mathcal{H}\cdot v}(\tau\times\rho,k)\\
\leq{}&
2\Csp^2 k\max\{k,n_{\min}\}(p+1)\sum_{\tau\times\rho\in\mathcal{B}}(\#\tau+\#\rho)\\
\leq{}&4\Csp^3k\max\{k,n_{\min}\}(p+1)^2\#\mathcal{I},
\end{align*}
due to the fact that
\[
\sum _{\tau\times\sigma\in\mathcal{B}}(\#\tau+\#\sigma)\leq 2\Csp (p+1)\#\mathcal{I},
\]
see, e.g., \cite[Lemma 2.4]{GH03}. Since the second sum can be estimated by
\begin{align*}
\sum_{\tau\times\sigma\in\mathcal{B}}
\sum_{\substack{
\rho\in\mathcal{T}_{\mathcal{I}}\colon\\
\tau\times\rho\in\mathcal{F}\\
\rho\times\sigma\in\mathcal{F}
}}
k\big(\#\rho+\min\{\#\tau,\#\sigma\}
\big)
\leq{}&
\sum_{\substack{
\tau\times\rho\in\mathcal{B}\\
\rho\times\sigma\in\mathcal{B}
}}
k\big(\#\rho+\min\{\#\tau,\#\sigma\}
\big)\\
\leq{}&
\sum_{\substack{
\tau\times\rho\in\mathcal{B}\\
\rho\times\sigma\in\mathcal{B}
}}
k\big(2\#\rho+\#\tau+\#\sigma
\big)\\
\leq{}&
2\Csp k(p+1)\#\mathcal{I},
\end{align*}
summing up yields the assertion.
\end{proof}

\begin{lemma}\label{lem:compsumexprmv}
For any block-cluster $\tau\times\sigma\in\mathcal{B}$, the
multiplication of $\mathcal{S}(\tau,\sigma)$ with a vector of
length $\#\sigma$ can be accomplished in at most
\[
4\Csp \max\{k, n_{\min}\} (p+1)\sum_{\substack{
\rho\in\mathcal{T}_{\mathcal{I}}\colon\\
\tau\times\rho\in\mathcal{B\setminus F}\\
\rho\times\sigma\in\mathcal{B\setminus F}
}}(\#\tau+\#\rho+\#\sigma)
\]
operations.
\end{lemma}
\begin{proof}
We first estimate the number of elements in
$\mathcal{S}_\mathcal{R}(\tau,\sigma)$ and $\mathcal{S}_\mathcal{H}(\tau,\sigma)$.
Therefore, we remark that, for fixed $\tau$, there are at most
$\Csp$ block-cluster pairs
$\tau\times\rho$ in
$\mathcal{B}$ and that the same consideration holds for $\sigma$.
Thus, looking at the recursion \eqref{eq:recursion},
the recursion step from $\mathcal{S}\big(\parent(\tau),\parent(\sigma)\big)$
to $\mathcal{S}(\tau,\sigma)$ adds at most $\Csp$ low-rank
matrices. Thus, considering that
$\mathcal{S}(\mathcal{I},\mathcal{I})=\emptyset$, we have at most
$\Csp\level(\tau\times\sigma)$ low-rank matrices in 
$\mathcal{S}(\tau,\sigma)$. Summing up, the multiplication
of $\mathcal{S}(\tau,\sigma)$ with a vector requires at most
\begin{align*}
&\Csp k\level(\tau\times\sigma)(\#\tau+\#\sigma)+
\sum_{\substack{
\rho\in\mathcal{T}_{\mathcal{I}}\colon\\
\tau\times\rho\in\mathcal{B\setminus F}\\
\rho\times\sigma\in\mathcal{B\setminus F}
}}\Big(N_{\mathcal{H}\cdot v}(\tau\times\rho,k)+N_{\mathcal{H}\cdot v}(\rho\times\sigma,k)\Big)\\
&\qquad\qquad\leq \Csp k(p+1)(\#\tau+\#\sigma)\\
&\qquad\qquad\qquad\qquad+\sum_{\substack{
\rho\in\mathcal{T}_{\mathcal{I}}\colon\\
\tau\times\rho\in\mathcal{B\setminus F}\\
\rho\times\sigma\in\mathcal{B\setminus F}
}}2\Csp \max\{k, n_{\min}\} (p+1)(\#\tau+2\#\rho+\#\sigma)\\
&\qquad\qquad\leq 4\Csp \max\{k, n_{\min}\} (p+1)\sum_{\substack{
\rho\in\mathcal{T}_{\mathcal{I}}\colon\\
\tau\times\rho\in\mathcal{B\setminus F}\\
\rho\times\sigma\in\mathcal{B\setminus F}
}}(\#\tau+\#\rho+\#\sigma)
\end{align*}
operations.
\end{proof}

With the help of the previous lemmata, we are now able to state our
main result, which estimates the number of operations for the
$\mathcal{H}$-matrix multiplication from  Algorithm~\ref{alg:theHmult}.
\begin{theorem}
Assuming that the range approximation scheme $\mathfrak{T}(\mathbf{M},k^\star)$
requires $\ell k^\star$, $\ell\geq 1$, matrix-vector multiplications of $\mathbf{M}$ and
additionally $N_{\mathfrak{T}}(k^{\star})$ operations to find a rank $k^\star$
approximation to a matrix $\mathbf{M}\in\mathbb{R}^{m\times n}$, then the
$\mathcal{H}$-matrix multiplication as stated in Algorithm~\ref{alg:theHmult}
requires at most
\[
8\Csp^2 (\ell k^{\star}+n_{\min}+2\Csp)k\max\{k, n_{\min}\} (p+1)^2\#\mathcal{I}+\Csp(2\#\mathcal{I}-1) N_{\mathfrak{T}}(k^{\star})
\]
operations.
\end{theorem}
\begin{proof}
We start by estimating the number of operations for a single far-field block
$\tau\times\sigma\in\mathcal{F}$. Using Lemma~\ref{lem:compsumexprmv}, this
requires at most
\[
4\Csp \ell k^{\star}\max\{k, n_{\min}\} (p+1)\sum_{\substack{
\rho\in\mathcal{T}_{\mathcal{I}}\colon\\
\tau\times\rho\in\mathcal{B\setminus F}\\
\rho\times\sigma\in\mathcal{B\setminus F}
}}(\#\tau+\#\rho+\#\sigma)
+N_{\mathfrak{T}}(k^{\star})
\]
operations. Summing up over all farfield blocks yields an effort of at most
\begin{align*}
&{}\sum_{\tau\times\sigma\in\mathcal{F}}
\bigg(
4\Csp \ell k^{\star}\max\{k, n_{\min}\} (p+1)\sum_{\substack{
\rho\in\mathcal{T}_{\mathcal{I}}\colon\\
\tau\times\rho\in\mathcal{B\setminus F}\\
\rho\times\sigma\in\mathcal{B\setminus F}
}}(\#\tau+\#\rho+\#\sigma)
+N_{\mathfrak{T}}(k^{\star})
\bigg)\\
&{}\qquad\qquad\leq
4\Csp \ell k^{\star}\max\{k, n_{\min}\} (p+1)\sum_{\substack{
\tau\times\rho\in\mathcal{B}\\
\rho\times\sigma\in\mathcal{B}
}}(\#\tau+2\#\rho+\#\sigma)+
\sum_{\tau\times\sigma\in\mathcal{F}}N_{\mathfrak{T}}(k^{\star})\\
&{}\qquad\qquad\leq 8\Csp^2 \ell k^{\star}k\max\{k, n_{\min}\} (p+1)^2\#\mathcal{I}+\Csp(2\#\mathcal{I}-1) N_{\mathfrak{T}}(k^{\star}),
\end{align*}
where we have used that $\#\mathcal{F}\leq \Csp(2\#\mathcal{I}-1)$,
cf.~\cite[Lemma 6.11]{Hac15}. Assuming that a nearfield block $\tau\times\sigma
\in\mathcal{N}$ is computed by applying its corresponding
\sumexpr-expression $\mathcal{S}(\tau,\sigma)$ to an identity matrix of size at
most $n_{\min}\times n_{\min}$, the nearfield blocks of the $\mathcal{H}$-matrix
product can be computed in at most
\begin{align*}
8\Csp^2 k\max\{k, n_{\min}\}n_{\min} (p+1)^2\#\mathcal{I}
\end{align*}
operations. Summing up the operations for the farfield, the nearfield, and the
operations of the \sumexpr-expressions from Lemma~\ref{lem:compsumexprtot}
yields the assertion.
\end{proof}
We remark that, with the convention $p=\depth(\mathcal{B})=\mathcal{O}(\log(\#\mathcal{I}))$,
the previous theorem states that the $\mathcal{H}$-matrix multiplication from
Algorithm~\ref{alg:theHmult} has an almost linear cost with respect to $\#\mathcal{I}$.
To that end, for given $\mathcal{I}$, we require that the block-wise ranks of the
$\mathcal{H}$-matrix are bounded by a constant $k^\star$. That constant may depend on
$\mathcal{I}$. According to \cite{DHS17}, this is reasonable for $\mathcal{H}$-matrices
which arise from the discretization of pseudodifferential operators. In particular, it is
well known that $k^\star$ depends poly-logarithmically on $\#\mathcal{I}$ for suitable
approximation accuracies $\varepsilon$.
%!TEX root = paper.tex
\section{Numerical examples}\label{sec:results}
The following section is dedicated to a comparison of the new
$\mathcal{H}$-matrix multiplication to the original algorithm from \cite{Hack1}, see also Section~\ref{sec:fasttrunc}. Besides the
comparison between these two algorithms, we also compare different
truncation operators. The considered configurations are listed in
Table~\ref{tab:expcases}.
Note that, in order to compute the dense SVD of a
\sumexpr-expression, we compute the SVD of the \sumexpr-expression
applied to a dense identity matrix.

\begin{table}[htb]
\begin{tabular}{|c|c|c|}
\hline
 & traditional multiplication & new multiplication\\\hline
\parbox[c]{0.2\textwidth}{\centering sum of low-rank matrices} &
\parbox[c]{0.35\textwidth}{\centering truncation with SVD of low-rank matrices, see Algorithm~\ref{alg:lrsvd}} & --- \\\hline
\parbox[c]{0.2\textwidth}{\centering conversion of products of $\mathcal{H}$-matrix blocks to low rank} & 
\parbox[c]{0.35\textwidth}{
\begin{itemize}[noitemsep,leftmargin=2ex]
\item Hierarchical approximation, see \cite{Hack1}
\item ACA, see Section~\ref{sec:ACA}
\item BiLanczos, see Section~\ref{sec:lanczos}
\item Randomized, see Setion~\ref{sec:randomized}
\item Reference: dense SVD
\end{itemize}
} & ---\\\hline
\parbox[c]{0.2\textwidth}{\centering conversion of \sumexpr-expressions to low rank}&
\parbox[c]{0.35\textwidth}{
\centering combination of the operations above, see discussion in
Section~\ref{sec:Hmult}
}&
\parbox[c]{0.35\textwidth}{
\begin{itemize}[noitemsep,leftmargin=2ex]
\item ACA, see Section~\ref{sec:ACA}
\item BiLanczos, see Section~\ref{sec:lanczos}
\item Randomized, see Setion~\ref{sec:randomized}
\item Reference: dense SVD
\end{itemize}
}\\\hline
\end{tabular}
\caption{\label{tab:expcases}Considered configurations of the $\mathcal{H}$-matrix multiplication for the numerical experiments.}
\end{table}

%The idea of considering other truncation operators than the hierarchical
%approximation in the traditional $\mathcal{H}$-matrix multiplication
%goes back to \cite{DHP15}, see also the remarks at the end
%of Section~\ref{sec:Hmult}.

We would like to compare the different configurations in terms of
speed and accuracy. For computational efficiency, the traditional
$\mathcal{H}$-matrix arithmetic puts an upper bound $k_{\max}$ on the
truncation rank. However, for computational accuracy, one should
rather put an upper bound on the truncation error. This is also referred
to as $\varepsilon$-rank, where the truncation is with respect to
the smallest possible rank such that the relative truncation error is smaller than $\varepsilon$.
We perform the experiments for both types of bounds, with $k_{\max}=16$ and
$\varepsilon=10^{-12}$.
%For the
%$\varepsilon$-rank truncation, we use the same error estimator for all three
%methods, i.e., we set the lookahead to one in the BiLanczos iteration.

\subsection{Example with exponential kernels}
For our first numerical example, we consider the Galerkin discretizations
$\mathbf{K}_1$ and $\mathbf{K}_2$ of an exponential kernel
\[
k_1(\mathbf{x},\mathbf{y})=\exp(-\|\mathbf{x}-\mathbf{y}\|)
\]
and a scaled exponential kernel
\[
k_2(\mathbf{x},\mathbf{y})=x_1\exp(-\|\mathbf{x}-\mathbf{y}\|)
\]
on the unit sphere with boundary $\Gamma$. This means that, for a given finite element space
$V_N=\spann\{\varphi_1,\ldots,\varphi _N\}$ on $\Gamma$, the system matrices are given by
\[
\big[\mathbf{K}_{\ell}\big]_{ij}
=
\int_\Gamma\int_\Gamma k_\ell(\mathbf{x},\mathbf{y})\varphi_j(\mathbf{x})\varphi_i(\mathbf{y})\de\sigma_\mathbf{x}\de\sigma_\mathbf{y}
\]
for all $i,j=1,\ldots, N$ and $\ell=1,2$. 

It is then well known that
$\mathbf{K}_1$, $\mathbf{K}_2$ and their product $\mathbf{K}_1\mathbf{K}_2$ is
compressible by means of $\mathcal{H}$-matrices, see \cite{DHS17,Hac15}. For our numerical experiments,
we assemble the matrices by using piecewise constant finite elements and adaptive
cross approximation as described in \cite{HP2013}.
The computations have been carried out on a single core of
a compute server with two Intel(R) Xeon(R) E5-2698v4 CPUs with a clock rate of
2.20GHz and a main memory of 756GB. The backend for the linear algebra routines
is version 3.2.8 of the software library Eigen, see \cite{eigenweb}.

Figure~\ref{fig:esetfr} depicts the computation times per degree of freedom
for the different kinds of $\mathcal{H}$-matrix multiplication for the fixed
rank truncation, whereas Figure~\ref{fig:eseter} shows the computations
times for the $\varepsilon$-rank truncation. The cost of the fixed rank
truncation seems to be $\mathcal{O}(N\log(N)^2)$, in accordance with the
theoretical cost estimates. We can also immediately infer that it pays off to
replace the hierarchical approximation by the alternative low-rank approximation
schemes to improve computation times. For the $\varepsilon$-rank truncation, no
cost estimates are known. While the asymptotic behaviour of the
computation times for the traditional multiplication seems to be in the preasymptotic
regime in this case, the new multiplication scales almost linearly.
Another important point to remark is that the new algorithm with
$\varepsilon$-rank truncation seems to outperform the frequently used traditional
$\mathcal{H}$-matrix multiplication in terms of computational efficiency.
Therefore, we shall now look whether it is also competitive in terms of
accuracy.
\begin{figure}
\scalefig{% This file was created by matlab2tikz.
\definecolor{mycolor1}{rgb}{0.00000,0.44700,0.74100}%
\definecolor{mycolor2}{rgb}{0.85000,0.32500,0.09800}%
\definecolor{mycolor3}{rgb}{0.92900,0.69400,0.12500}%
\definecolor{mycolor4}{rgb}{0.49400,0.18400,0.55600}%
\definecolor{mycolor5}{rgb}{0.46600,0.67400,0.18800}%
\begin{tikzpicture}

\begin{axis}[%
width=  0.5\textwidth,
height=       0.4\textwidth,
at={(         0\textwidth,         0\textwidth)},
scale only axis,
xmode=log,
xmin=1,
xmax=1000000,
xminorticks=true,
xlabel style={font=\color{white!15!black}},
xlabel={$\#\mathcal{I}$},
ymode=log,
ymin=1e-07,
ymax=1,
yminorticks=true,
ylabel style={font=\color{white!15!black}},
ylabel={s / $\#\mathcal{I}$},
axis background/.style={fill=white},
title={truncation to rank 16, traditional multiplication},
xmajorgrids,
ymajorgrids,
legend style={at={(      0.97,      0.03)}, anchor=south east, legend cell align=left, align=left, draw=white!15!black}
]
\addplot [color=mycolor1]
  table[row sep=crcr]{%
6	2.83333333333333e-06\\
24	1.16666666666667e-06\\
96.0000000000001	7.9375e-06\\
384.000000000001	0.000126559895833333\\
1536	0.00134795572916667\\
6144	0.00501630859375\\
24576	0.0130326334635417\\
98304	0.023697509765625\\
393216	0.0395561218261719\\
};
\addlegendentry{ACA}

\addplot [color=mycolor2]
  table[row sep=crcr]{%
6	3.5e-06\\
24	8.75e-07\\
96.0000000000001	7.56250000000001e-06\\
384.000000000001	0.000151052083333333\\
1536	0.0015441015625\\
6144	0.00554381510416667\\
24576	0.0133843180338542\\
98304	0.0240191650390625\\
393216	0.0391474405924479\\
};
\addlegendentry{BiLanczos}

\addplot [color=mycolor3]
  table[row sep=crcr]{%
6	2.83333333333333e-06\\
24	9.58333333333333e-07\\
96.0000000000001	7.66666666666667e-06\\
384.000000000001	0.000139776041666667\\
1536	0.00149927734375\\
6144	0.00567426757812499\\
24576	0.0129478759765625\\
98304	0.0233986409505208\\
393216	0.0388824462890625\\
};
\addlegendentry{Randomized}

\addplot [color=mycolor4]
  table[row sep=crcr]{%
6	4.16666666666666e-06\\
24	1.33333333333333e-06\\
96.0000000000001	6.13541666666666e-06\\
384.000000000001	0.000696039062499999\\
1536	0.0252180989583333\\
6144	0.820641276041667\\
};
\addlegendentry{SVD}

\addplot [color=mycolor5]
  table[row sep=crcr]{%
6	4.83333333333334e-06\\
24	1.5e-06\\
96.0000000000001	8.09375e-06\\
384.000000000001	0.00118475\\
1536	0.00796575520833333\\
6144	0.0182859700520833\\
24576	0.030990966796875\\
98304	0.0471506754557292\\
393216	0.0694157918294271\\
};
\addlegendentry{Hier.~Approx.}

\addplot [color=black, dashed, forget plot]
  table[row sep=crcr]{%
6	0.0016052009977842\\
6.71123680053168	0.00181219534288484\\
7.50678323213512	0.00203173901678489\\
8.39663331351989	0.00226383201948436\\
10.5052843607197	0.00276566601128153\\
13.1434820813104	0.00331769731827637\\
16.4442118166415	0.00391992594046886\\
20.5738555884738	0.00457235187785901\\
25.7405790253207	0.00527497513044681\\
32.204824502121	0.00602779569823228\\
40.2924394277295	0.00683081358121539\\
50.4111014463166	0.00768402877939617\\
63.0708685084453	0.0085874412927746\\
88.2638259384407	0.0100366800288379\\
123.519830208906	0.0115988627240958\\
172.858453535403	0.0132739893785486\\
241.904841579816	0.015062059992196\\
338.531041918473	0.0169630745650382\\
473.753504038874	0.0189770330970751\\
741.579447237127	0.0218380017429827\\
1160.81479477437	0.024899759649681\\
1817.0554655841	0.0281623068171699\\
2844.28711614656	0.0316256432454495\\
4980.0074931085	0.0362371736787858\\
8719.39913893663	0.0411624373321074\\
15266.627901524	0.0464014342054144\\
29898.6097837907	0.0531023583037634\\
58554.3102753	0.0602550582388914\\
114674.47070649	0.0678595340107983\\
251203.641819084	0.0773024168717296\\
393216	0.0829744351687061\\
};
\addplot [color=black, dashed, forget plot]
  table[row sep=crcr]{%
6	0.000223265685105974\\
6.71123680053168	0.000267815286716958\\
7.50678323213512	0.000317928695955608\\
8.39663331351989	0.000373933302401214\\
9.39196574904416	0.000436156495633068\\
10.5052843607197	0.000504925665230462\\
11.750575166952	0.000580568200772688\\
14.7015034385365	0.0007537829280088\\
18.3934999165971	0.000958419793975735\\
23.0126694590316	0.00119709791530783\\
28.7918535369543	0.00147243640863942\\
36.0223672255499	0.00178705439060482\\
45.0686837117619	0.00214357097783839\\
56.3868065303093	0.00254460528697445\\
70.5472556292287	0.00299277643464732\\
88.2638259384407	0.00349070353749135\\
110.42956809313	0.00404100571214088\\
138.161805015572	0.00464630207523021\\
193.349002441632	0.0056630900408265\\
270.580112506207	0.00681834746148049\\
378.660331107598	0.00812091385583309\\
529.911991781088	0.00957962874252515\\
741.579447237127	0.0112033316401975\\
1037.79511521549	0.0130008620674912\\
1452.33083950443	0.0149810595430469\\
2273.37358348436	0.0179207497658958\\
3558.57447181116	0.0212218490407579\\
5570.33492577888	0.024905310300708\\
8719.39913893663	0.0289920864788207\\
13648.7163441897	0.0335031305081706\\
23897.2743924192	0.0397704826946413\\
41841.2771564182	0.0467744160540108\\
73259.0857573098	0.0545558542836906\\
128267.921314482	0.0631557210810921\\
251203.641819084	0.0746135075133219\\
393216	0.0829744351687061\\
};
\addplot [color=black, dashed, forget plot]
  table[row sep=crcr]{%
6	3.10537846753452e-05\\
6.71123680053168	3.95790818472734e-05\\
7.50678323213512	4.97498226282944e-05\\
8.39663331351989	6.17652340991831e-05\\
9.39196574904416	7.58359313533958e-05\\
10.5052843607197	9.21839174970691e-05\\
11.750575166952	0.000111042583649021\\
13.1434820813104	0.00013265670894075\\
16.4442118166415	0.000185187393532941\\
20.5738555884738	0.000251961964579249\\
25.7405790253207	0.000335348623588181\\
32.204824502121	0.000437897780271133\\
40.2924394277295	0.000562342052542401\\
50.4111014463166	0.000711596266519174\\
63.0708685084453	0.000888757456521537\\
78.909889692566	0.00109710486507246\\
98.7265727989643	0.00134009994289783\\
123.519830208906	0.00162138634892641\\
154.539431708065	0.00194478995028987\\
193.349002441632	0.00231431882232276\\
241.904841579816	0.00273416324856254\\
302.654534757277	0.00320869572074955\\
423.54652483847	0.00403304774386247\\
592.727476714046	0.00500687945344094\\
829.485879469291	0.00614725408167652\\
1160.81479477437	0.00747215728978795\\
1624.48936277154	0.00900049716802108\\
2273.37358348436	0.010752104235649\\
3181.44739419349	0.0127477314409718\\
4452.24119588268	0.015009054161317\\
6969.21613636892	0.0184766460533132\\
10909.1065417438	0.0225129817791119\\
17076.3258321219	0.0271770927554179\\
26730.0445558012	0.0325309257301824\\
41841.2771564182	0.0386393427826031\\
65495.3069915573	0.0455701213231238\\
114674.47070649	0.0554980138978931\\
200781.320614451	0.0669659947501943\\
351544.144562608	0.0801256589061995\\
393216	0.0829744351687061\\
};
\end{axis}
\end{tikzpicture}%
}
\scalefig{% This file was created by matlab2tikz.
\definecolor{mycolor1}{rgb}{0.00000,0.44700,0.74100}%
\definecolor{mycolor2}{rgb}{0.85000,0.32500,0.09800}%
\definecolor{mycolor3}{rgb}{0.92900,0.69400,0.12500}%
\definecolor{mycolor4}{rgb}{0.49400,0.18400,0.55600}%
\begin{tikzpicture}

\begin{axis}[%
width=  0.5\textwidth,
height=       0.4\textwidth,
at={(         0\textwidth,         0\textwidth)},
scale only axis,
xmode=log,
xmin=1,
xmax=1000000,
xminorticks=true,
xlabel style={font=\color{white!15!black}},
xlabel={$\#\mathcal{I}$},
ymode=log,
ymin=1e-07,
ymax=1,
yminorticks=true,
ylabel style={font=\color{white!15!black}},
ylabel={s / $\#\mathcal{I}$},
axis background/.style={fill=white},
title={truncation to rank 16, new multiplication},
xmajorgrids,
ymajorgrids,
legend style={at={(      0.97,      0.03)}, anchor=south east, legend cell align=left, align=left, draw=white!15!black}
]
\addplot [color=mycolor1]
  table[row sep=crcr]{%
6	4e-06\\
24	1.79166666666667e-06\\
96.0000000000001	4.50208333333333e-05\\
384.000000000001	0.000313739583333333\\
1536	0.000809628906249999\\
6144	0.00206031901041667\\
24576	0.0043970947265625\\
98304	0.0085166524251302\\
393216	0.014923578898112\\
};
\addlegendentry{ACA}

\addplot [color=mycolor2]
  table[row sep=crcr]{%
6	5.16666666666667e-06\\
24	2.20833333333333e-06\\
96.0000000000001	2.934375e-05\\
384.000000000001	0.000231130208333333\\
1536	0.000806966145833333\\
6144	0.00179557291666667\\
24576	0.00357755533854166\\
98304	0.00634388224283854\\
393216	0.00993301391601562\\
};
\addlegendentry{BiLanczos}

\addplot [color=mycolor3]
  table[row sep=crcr]{%
6	4.66666666666667e-06\\
24	2.70833333333333e-06\\
96.0000000000001	2.35833333333333e-05\\
384.000000000001	0.000207627604166667\\
1536	0.000832949218749999\\
6144	0.00204913736979167\\
24576	0.00419571940104167\\
98304	0.00741547648111978\\
393216	0.0123853810628255\\
};
\addlegendentry{Randomized}

\addplot [color=mycolor4]
  table[row sep=crcr]{%
6	6.5e-06\\
24	1.83333333333333e-06\\
96.0000000000001	1.35625e-05\\
384.000000000001	0.000242619791666666\\
1536	0.0075349609375\\
6144	0.27962890625\\
};
\addlegendentry{SVD}

\addplot [color=black, dashed, forget plot]
  table[row sep=crcr]{%
6	0.0016052009977842\\
6.71123680053168	0.00181219534288484\\
7.50678323213512	0.00203173901678489\\
8.39663331351989	0.00226383201948436\\
10.5052843607197	0.00276566601128153\\
13.1434820813104	0.00331769731827637\\
16.4442118166415	0.00391992594046886\\
20.5738555884738	0.00457235187785901\\
25.7405790253207	0.00527497513044681\\
32.204824502121	0.00602779569823228\\
40.2924394277295	0.00683081358121539\\
50.4111014463166	0.00768402877939617\\
63.0708685084453	0.0085874412927746\\
88.2638259384407	0.0100366800288379\\
123.519830208906	0.0115988627240958\\
172.858453535403	0.0132739893785486\\
241.904841579816	0.015062059992196\\
338.531041918473	0.0169630745650382\\
473.753504038874	0.0189770330970751\\
741.579447237127	0.0218380017429827\\
1160.81479477437	0.024899759649681\\
1817.0554655841	0.0281623068171699\\
2844.28711614656	0.0316256432454495\\
4980.0074931085	0.0362371736787858\\
8719.39913893663	0.0411624373321074\\
15266.627901524	0.0464014342054144\\
29898.6097837907	0.0531023583037634\\
58554.3102753	0.0602550582388914\\
114674.47070649	0.0678595340107983\\
251203.641819084	0.0773024168717296\\
393216	0.0829744351687061\\
};
\addplot [color=black, dashed, forget plot]
  table[row sep=crcr]{%
6	0.000223265685105974\\
6.71123680053168	0.000267815286716958\\
7.50678323213512	0.000317928695955608\\
8.39663331351989	0.000373933302401214\\
9.39196574904416	0.000436156495633068\\
10.5052843607197	0.000504925665230462\\
11.750575166952	0.000580568200772688\\
14.7015034385365	0.0007537829280088\\
18.3934999165971	0.000958419793975735\\
23.0126694590316	0.00119709791530783\\
28.7918535369543	0.00147243640863942\\
36.0223672255499	0.00178705439060482\\
45.0686837117619	0.00214357097783839\\
56.3868065303093	0.00254460528697445\\
70.5472556292287	0.00299277643464732\\
88.2638259384407	0.00349070353749135\\
110.42956809313	0.00404100571214088\\
138.161805015572	0.00464630207523021\\
193.349002441632	0.0056630900408265\\
270.580112506207	0.00681834746148049\\
378.660331107598	0.00812091385583309\\
529.911991781088	0.00957962874252515\\
741.579447237127	0.0112033316401975\\
1037.79511521549	0.0130008620674912\\
1452.33083950443	0.0149810595430469\\
2273.37358348436	0.0179207497658958\\
3558.57447181116	0.0212218490407579\\
5570.33492577888	0.024905310300708\\
8719.39913893663	0.0289920864788207\\
13648.7163441897	0.0335031305081706\\
23897.2743924192	0.0397704826946413\\
41841.2771564182	0.0467744160540108\\
73259.0857573098	0.0545558542836906\\
128267.921314482	0.0631557210810921\\
251203.641819084	0.0746135075133219\\
393216	0.0829744351687061\\
};
\addplot [color=black, dashed, forget plot]
  table[row sep=crcr]{%
6	3.10537846753452e-05\\
6.71123680053168	3.95790818472734e-05\\
7.50678323213512	4.97498226282944e-05\\
8.39663331351989	6.17652340991831e-05\\
9.39196574904416	7.58359313533958e-05\\
10.5052843607197	9.21839174970691e-05\\
11.750575166952	0.000111042583649021\\
13.1434820813104	0.00013265670894075\\
16.4442118166415	0.000185187393532941\\
20.5738555884738	0.000251961964579249\\
25.7405790253207	0.000335348623588181\\
32.204824502121	0.000437897780271133\\
40.2924394277295	0.000562342052542401\\
50.4111014463166	0.000711596266519174\\
63.0708685084453	0.000888757456521537\\
78.909889692566	0.00109710486507246\\
98.7265727989643	0.00134009994289783\\
123.519830208906	0.00162138634892641\\
154.539431708065	0.00194478995028987\\
193.349002441632	0.00231431882232276\\
241.904841579816	0.00273416324856254\\
302.654534757277	0.00320869572074955\\
423.54652483847	0.00403304774386247\\
592.727476714046	0.00500687945344094\\
829.485879469291	0.00614725408167652\\
1160.81479477437	0.00747215728978795\\
1624.48936277154	0.00900049716802108\\
2273.37358348436	0.010752104235649\\
3181.44739419349	0.0127477314409718\\
4452.24119588268	0.015009054161317\\
6969.21613636892	0.0184766460533132\\
10909.1065417438	0.0225129817791119\\
17076.3258321219	0.0271770927554179\\
26730.0445558012	0.0325309257301824\\
41841.2771564182	0.0386393427826031\\
65495.3069915573	0.0455701213231238\\
114674.47070649	0.0554980138978931\\
200781.320614451	0.0669659947501943\\
351544.144562608	0.0801256589061995\\
393216	0.0829744351687061\\
};
\end{axis}
\end{tikzpicture}%
}
\caption{\label{fig:esetfr}Computation times in seconds per degree of freedom
for the product of the matrices occurring from the exponential kernels
using fixed rank truncation with corresponding asymptotics $N\log^2N$,
$N\log^3N$, and $N\log^4N$.}
\end{figure}
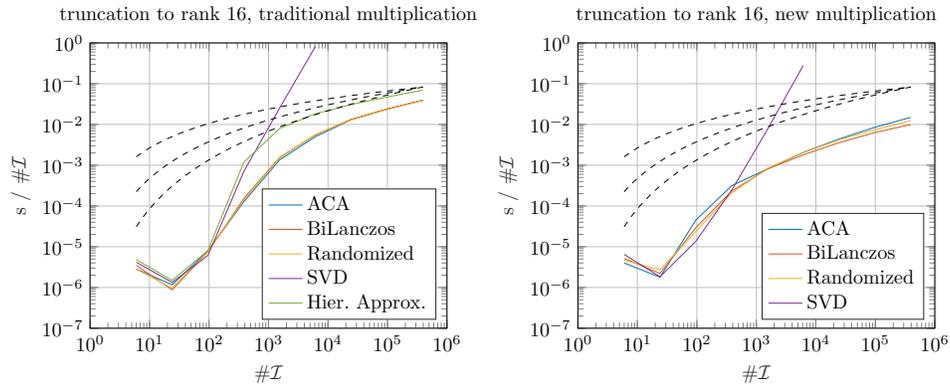
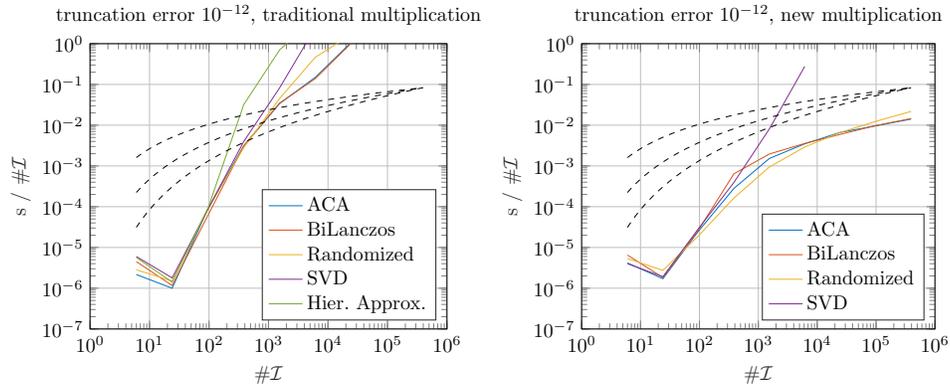
\begin{figure}
\scalefig{% This file was created by matlab2tikz.
\definecolor{mycolor1}{rgb}{0.00000,0.44700,0.74100}%
\definecolor{mycolor2}{rgb}{0.85000,0.32500,0.09800}%
\definecolor{mycolor3}{rgb}{0.92900,0.69400,0.12500}%
\definecolor{mycolor4}{rgb}{0.49400,0.18400,0.55600}%
\definecolor{mycolor5}{rgb}{0.46600,0.67400,0.18800}%
\begin{tikzpicture}

\begin{axis}[%
width=  0.5\textwidth,
height=       0.4\textwidth,
at={(         0\textwidth,         0\textwidth)},
scale only axis,
xmode=log,
xmin=1,
xmax=1000000,
xminorticks=true,
xlabel style={font=\color{white!15!black}},
xlabel={$\#\mathcal{I}$},
ymode=log,
ymin=1e-07,
ymax=1,
yminorticks=true,
ylabel style={font=\color{white!15!black}},
ylabel={s / $\#\mathcal{I}$},
axis background/.style={fill=white},
title={truncation error $10^{-12}$, traditional multiplication},
xmajorgrids,
ymajorgrids,
legend style={at={(      0.97,      0.03)}, anchor=south east, legend cell align=left, align=left, draw=white!15!black}
]
\addplot [color=mycolor1]
  table[row sep=crcr]{%
6	2.16666666666667e-06\\
24	1e-06\\
96.0000000000001	8.84166666666667e-05\\
384.000000000001	0.00311213541666667\\
1536	0.03380546875\\
6144	0.150038411458333\\
24576	1.07766520182292\\
};
\addlegendentry{ACA}

\addplot [color=mycolor2]
  table[row sep=crcr]{%
6	4.5e-06\\
24	1.16666666666667e-06\\
96.0000000000001	6.134375e-05\\
384.000000000001	0.00312947916666667\\
1536	0.0348748697916667\\
6144	0.139534830729167\\
24576	1.02823486328125\\
};
\addlegendentry{BiLanczos}

\addplot [color=mycolor3]
  table[row sep=crcr]{%
6	2.83333333333333e-06\\
24	1.5e-06\\
96.0000000000001	8.759375e-05\\
384.000000000001	0.00281614583333333\\
1536	0.0450075520833333\\
6144	0.467408854166666\\
24576	1.65286458333333\\
};
\addlegendentry{Randomized}

\addplot [color=mycolor4]
  table[row sep=crcr]{%
6	6e-06\\
24	1.79166666666667e-06\\
96.0000000000001	8.77395833333333e-05\\
384.000000000001	0.00399005208333333\\
1536	0.0822766927083334\\
6144	2.47330729166667\\
};
\addlegendentry{SVD}

\addplot [color=mycolor5]
  table[row sep=crcr]{%
6	5.66666666666667e-06\\
24	1.41666666666667e-06\\
96.0000000000001	9.16458333333334e-05\\
384.000000000001	0.032071875\\
1536	0.700514322916667\\
6144	4.41041666666667\\
};
\addlegendentry{Hier.~Approx.}

\addplot [color=black, dashed, forget plot]
  table[row sep=crcr]{%
6	0.0016052009977842\\
6.71123680053168	0.00181219534288484\\
7.50678323213512	0.00203173901678489\\
8.39663331351989	0.00226383201948436\\
10.5052843607197	0.00276566601128153\\
13.1434820813104	0.00331769731827637\\
16.4442118166415	0.00391992594046886\\
20.5738555884738	0.00457235187785901\\
25.7405790253207	0.00527497513044681\\
32.204824502121	0.00602779569823228\\
40.2924394277295	0.00683081358121539\\
50.4111014463166	0.00768402877939617\\
63.0708685084453	0.0085874412927746\\
88.2638259384407	0.0100366800288379\\
123.519830208906	0.0115988627240958\\
172.858453535403	0.0132739893785486\\
241.904841579816	0.015062059992196\\
338.531041918473	0.0169630745650382\\
473.753504038874	0.0189770330970751\\
741.579447237127	0.0218380017429827\\
1160.81479477437	0.024899759649681\\
1817.0554655841	0.0281623068171699\\
2844.28711614656	0.0316256432454495\\
4980.0074931085	0.0362371736787858\\
8719.39913893663	0.0411624373321074\\
15266.627901524	0.0464014342054144\\
29898.6097837907	0.0531023583037634\\
58554.3102753	0.0602550582388914\\
114674.47070649	0.0678595340107983\\
251203.641819084	0.0773024168717296\\
393216	0.0829744351687061\\
};
\addplot [color=black, dashed, forget plot]
  table[row sep=crcr]{%
6	0.000223265685105974\\
6.71123680053168	0.000267815286716958\\
7.50678323213512	0.000317928695955608\\
8.39663331351989	0.000373933302401214\\
9.39196574904416	0.000436156495633068\\
10.5052843607197	0.000504925665230462\\
11.750575166952	0.000580568200772688\\
14.7015034385365	0.0007537829280088\\
18.3934999165971	0.000958419793975735\\
23.0126694590316	0.00119709791530783\\
28.7918535369543	0.00147243640863942\\
36.0223672255499	0.00178705439060482\\
45.0686837117619	0.00214357097783839\\
56.3868065303093	0.00254460528697445\\
70.5472556292287	0.00299277643464732\\
88.2638259384407	0.00349070353749135\\
110.42956809313	0.00404100571214088\\
138.161805015572	0.00464630207523021\\
193.349002441632	0.0056630900408265\\
270.580112506207	0.00681834746148049\\
378.660331107598	0.00812091385583309\\
529.911991781088	0.00957962874252515\\
741.579447237127	0.0112033316401975\\
1037.79511521549	0.0130008620674912\\
1452.33083950443	0.0149810595430469\\
2273.37358348436	0.0179207497658958\\
3558.57447181116	0.0212218490407579\\
5570.33492577888	0.024905310300708\\
8719.39913893663	0.0289920864788207\\
13648.7163441897	0.0335031305081706\\
23897.2743924192	0.0397704826946413\\
41841.2771564182	0.0467744160540108\\
73259.0857573098	0.0545558542836906\\
128267.921314482	0.0631557210810921\\
251203.641819084	0.0746135075133219\\
393216	0.0829744351687061\\
};
\addplot [color=black, dashed, forget plot]
  table[row sep=crcr]{%
6	3.10537846753452e-05\\
6.71123680053168	3.95790818472734e-05\\
7.50678323213512	4.97498226282944e-05\\
8.39663331351989	6.17652340991831e-05\\
9.39196574904416	7.58359313533958e-05\\
10.5052843607197	9.21839174970691e-05\\
11.750575166952	0.000111042583649021\\
13.1434820813104	0.00013265670894075\\
16.4442118166415	0.000185187393532941\\
20.5738555884738	0.000251961964579249\\
25.7405790253207	0.000335348623588181\\
32.204824502121	0.000437897780271133\\
40.2924394277295	0.000562342052542401\\
50.4111014463166	0.000711596266519174\\
63.0708685084453	0.000888757456521537\\
78.909889692566	0.00109710486507246\\
98.7265727989643	0.00134009994289783\\
123.519830208906	0.00162138634892641\\
154.539431708065	0.00194478995028987\\
193.349002441632	0.00231431882232276\\
241.904841579816	0.00273416324856254\\
302.654534757277	0.00320869572074955\\
423.54652483847	0.00403304774386247\\
592.727476714046	0.00500687945344094\\
829.485879469291	0.00614725408167652\\
1160.81479477437	0.00747215728978795\\
1624.48936277154	0.00900049716802108\\
2273.37358348436	0.010752104235649\\
3181.44739419349	0.0127477314409718\\
4452.24119588268	0.015009054161317\\
6969.21613636892	0.0184766460533132\\
10909.1065417438	0.0225129817791119\\
17076.3258321219	0.0271770927554179\\
26730.0445558012	0.0325309257301824\\
41841.2771564182	0.0386393427826031\\
65495.3069915573	0.0455701213231238\\
114674.47070649	0.0554980138978931\\
200781.320614451	0.0669659947501943\\
351544.144562608	0.0801256589061995\\
393216	0.0829744351687061\\
};
\end{axis}
\end{tikzpicture}%
}
\scalefig{% This file was created by matlab2tikz.
\definecolor{mycolor1}{rgb}{0.00000,0.44700,0.74100}%
\definecolor{mycolor2}{rgb}{0.85000,0.32500,0.09800}%
\definecolor{mycolor3}{rgb}{0.92900,0.69400,0.12500}%
\definecolor{mycolor4}{rgb}{0.49400,0.18400,0.55600}%
\begin{tikzpicture}

\begin{axis}[%
width=  0.495\textwidth,
height=       0.4\textwidth,
at={(         0\textwidth,         0\textwidth)},
scale only axis,
xmode=log,
xmin=1,
xmax=1000000,
xminorticks=true,
xlabel style={font=\color{white!15!black}},
xlabel={$\#\mathcal{I}$},
ymode=log,
ymin=1e-07,
ymax=1,
yminorticks=true,
ylabel style={font=\color{white!15!black}},
ylabel={s / $\#\mathcal{I}$},
axis background/.style={fill=white},
title={truncation error $10^{-12}$, new multiplication},
xmajorgrids,
ymajorgrids,
legend style={at={(      0.97,      0.03)}, anchor=south east, legend cell align=left, align=left, draw=white!15!black}
]
\addplot [color=mycolor1]
  table[row sep=crcr]{%
6	4.16666666666666e-06\\
24	1.70833333333333e-06\\
96.0000000000001	2.49479166666667e-05\\
384.000000000001	0.000279432291666666\\
1536	0.0015226953125\\
6144	0.00349205729166667\\
24576	0.00664921061197917\\
98304	0.00960606892903645\\
393216	0.0141305796305338\\
};
\addlegendentry{ACA}

\addplot [color=mycolor2]
  table[row sep=crcr]{%
6	6.5e-06\\
24	1.83333333333333e-06\\
96.0000000000001	2.659375e-05\\
384.000000000001	0.000638770833333333\\
1536	0.00195805338541667\\
6144	0.00356305338541666\\
24576	0.00594238281249999\\
98304	0.0097773234049479\\
393216	0.0145653279622396\\
};
\addlegendentry{BiLanczos}

\addplot [color=mycolor3]
  table[row sep=crcr]{%
6	5.33333333333333e-06\\
24	2.75e-06\\
96.0000000000001	1.91458333333333e-05\\
384.000000000001	0.0001625546875\\
1536	0.000941100260416666\\
6144	0.00291796875\\
24576	0.00660803222656249\\
98304	0.0123855590820312\\
393216	0.0217745463053385\\
};
\addlegendentry{Randomized}

\addplot [color=mycolor4]
  table[row sep=crcr]{%
6	4e-06\\
24	1.91666666666667e-06\\
96.0000000000001	2.88229166666667e-05\\
384.000000000001	0.000390440104166667\\
1536	0.00802773437499999\\
6144	0.2765087890625\\
};
\addlegendentry{SVD}

\addplot [color=black, dashed, forget plot]
  table[row sep=crcr]{%
6	0.0016052009977842\\
6.71123680053168	0.00181219534288484\\
7.50678323213512	0.00203173901678489\\
8.39663331351989	0.00226383201948436\\
10.5052843607197	0.00276566601128153\\
13.1434820813104	0.00331769731827637\\
16.4442118166415	0.00391992594046886\\
20.5738555884738	0.00457235187785901\\
25.7405790253207	0.00527497513044681\\
32.204824502121	0.00602779569823228\\
40.2924394277295	0.00683081358121539\\
50.4111014463166	0.00768402877939617\\
63.0708685084453	0.0085874412927746\\
88.2638259384407	0.0100366800288379\\
123.519830208906	0.0115988627240958\\
172.858453535403	0.0132739893785486\\
241.904841579816	0.015062059992196\\
338.531041918473	0.0169630745650382\\
473.753504038874	0.0189770330970751\\
741.579447237127	0.0218380017429827\\
1160.81479477437	0.024899759649681\\
1817.0554655841	0.0281623068171699\\
2844.28711614656	0.0316256432454495\\
4980.0074931085	0.0362371736787858\\
8719.39913893663	0.0411624373321074\\
15266.627901524	0.0464014342054144\\
29898.6097837907	0.0531023583037634\\
58554.3102753	0.0602550582388914\\
114674.47070649	0.0678595340107983\\
251203.641819084	0.0773024168717296\\
393216	0.0829744351687061\\
};
\addplot [color=black, dashed, forget plot]
  table[row sep=crcr]{%
6	0.000223265685105974\\
6.71123680053168	0.000267815286716958\\
7.50678323213512	0.000317928695955608\\
8.39663331351989	0.000373933302401214\\
9.39196574904416	0.000436156495633068\\
10.5052843607197	0.000504925665230462\\
11.750575166952	0.000580568200772688\\
14.7015034385365	0.0007537829280088\\
18.3934999165971	0.000958419793975735\\
23.0126694590316	0.00119709791530783\\
28.7918535369543	0.00147243640863942\\
36.0223672255499	0.00178705439060482\\
45.0686837117619	0.00214357097783839\\
56.3868065303093	0.00254460528697445\\
70.5472556292287	0.00299277643464732\\
88.2638259384407	0.00349070353749135\\
110.42956809313	0.00404100571214088\\
138.161805015572	0.00464630207523021\\
193.349002441632	0.0056630900408265\\
270.580112506207	0.00681834746148049\\
378.660331107598	0.00812091385583309\\
529.911991781088	0.00957962874252515\\
741.579447237127	0.0112033316401975\\
1037.79511521549	0.0130008620674912\\
1452.33083950443	0.0149810595430469\\
2273.37358348436	0.0179207497658958\\
3558.57447181116	0.0212218490407579\\
5570.33492577888	0.024905310300708\\
8719.39913893663	0.0289920864788207\\
13648.7163441897	0.0335031305081706\\
23897.2743924192	0.0397704826946413\\
41841.2771564182	0.0467744160540108\\
73259.0857573098	0.0545558542836906\\
128267.921314482	0.0631557210810921\\
251203.641819084	0.0746135075133219\\
393216	0.0829744351687061\\
};
\addplot [color=black, dashed, forget plot]
  table[row sep=crcr]{%
6	3.10537846753452e-05\\
6.71123680053168	3.95790818472734e-05\\
7.50678323213512	4.97498226282944e-05\\
8.39663331351989	6.17652340991831e-05\\
9.39196574904416	7.58359313533958e-05\\
10.5052843607197	9.21839174970691e-05\\
11.750575166952	0.000111042583649021\\
13.1434820813104	0.00013265670894075\\
16.4442118166415	0.000185187393532941\\
20.5738555884738	0.000251961964579249\\
25.7405790253207	0.000335348623588181\\
32.204824502121	0.000437897780271133\\
40.2924394277295	0.000562342052542401\\
50.4111014463166	0.000711596266519174\\
63.0708685084453	0.000888757456521537\\
78.909889692566	0.00109710486507246\\
98.7265727989643	0.00134009994289783\\
123.519830208906	0.00162138634892641\\
154.539431708065	0.00194478995028987\\
193.349002441632	0.00231431882232276\\
241.904841579816	0.00273416324856254\\
302.654534757277	0.00320869572074955\\
423.54652483847	0.00403304774386247\\
592.727476714046	0.00500687945344094\\
829.485879469291	0.00614725408167652\\
1160.81479477437	0.00747215728978795\\
1624.48936277154	0.00900049716802108\\
2273.37358348436	0.010752104235649\\
3181.44739419349	0.0127477314409718\\
4452.24119588268	0.015009054161317\\
6969.21613636892	0.0184766460533132\\
10909.1065417438	0.0225129817791119\\
17076.3258321219	0.0271770927554179\\
26730.0445558012	0.0325309257301824\\
41841.2771564182	0.0386393427826031\\
65495.3069915573	0.0455701213231238\\
114674.47070649	0.0554980138978931\\
200781.320614451	0.0669659947501943\\
351544.144562608	0.0801256589061995\\
393216	0.0829744351687061\\
};
\end{axis}
\end{tikzpicture}%
}
\caption{\label{fig:eseter}Computation times in seconds per degree of freedom
for the product of the matrices occurring from the exponential kernels
using $\varepsilon$-rank truncation with corresponding asymptotics $N\log^2N$,
$N\log^3N$, and $N\log^4N$.}
\end{figure}

To estimate the error of the $\mathcal{H}$-matrix multiplication, we apply
ten subspace iterations to a subspace of size 100, using the matrix-vector
product
\[
(\mathbf{K}_1\mathbf{K}_2)\mathbf{v}-\mathbf{K}_1(\mathbf{K}_2\mathbf{v}),
\]
and compute an approximation to the Frobenius norm.
Looking at the corresponding errors in Figures~\ref{fig:eseefr} and \ref{fig:eseeer},
we see that the accuracies of the fixed rank arithmetic behave similar. However,
the new multiplication reaches these accuracies in a shorter computation time.
For the $\varepsilon$-rank truncation, we observe that the traditional
multiplication cannot achieve the prescribed $\varepsilon$-rank of
$\varepsilon=10^{-12}$. This accuracy is only achieved by the new multiplication
with an appropriate low-rank approximation method. The computation times for these
accuracies are even smaller than the fixed-rank version of the
traditional $\mathcal{H}$-matrix multiplication. Concerning the low-rank
approximation methods, it seems as if ACA and the randomized algorithm
are more robust when applied to smaller matrix sizes, although this difference
vanishes for large matrices.
\begin{figure}
\scalefig{% This file was created by matlab2tikz.
\definecolor{mycolor1}{rgb}{0.00000,0.44700,0.74100}%
\definecolor{mycolor2}{rgb}{0.85000,0.32500,0.09800}%
\definecolor{mycolor3}{rgb}{0.92900,0.69400,0.12500}%
\definecolor{mycolor4}{rgb}{0.49400,0.18400,0.55600}%
\definecolor{mycolor5}{rgb}{0.46600,0.67400,0.18800}%
\begin{tikzpicture}

\begin{axis}[%
width=  0.5\textwidth,
height=       0.4\textwidth,
at={(         0\textwidth,         0\textwidth)},
scale only axis,
xmode=log,
xmin=1,
xmax=1000000,
xminorticks=true,
xlabel style={font=\color{white!15!black}},
xlabel={$\#\mathcal{I}$},
ymode=log,
ymin=1e-17,
ymax=1e-05,
ytick={1e-5,1e-7,1e-9,1e-11,1e-13,1e-15,1e-17},
yminorticks=true,
ylabel style={font=\color{white!15!black}},
ylabel={error},
axis background/.style={fill=white},
title={truncation to rank 16, traditional multiplication},
xmajorgrids,
ymajorgrids,
legend style={at={(      0.97,      0.03)}, anchor=south east, legend cell align=left, align=left, draw=white!15!black}
]
\addplot [color=mycolor1]
  table[row sep=crcr]{%
6	1.92824e-16\\
24	3.00744999999999e-16\\
96.0000000000001	2.76113e-16\\
384.000000000001	3.82215e-07\\
1536	1.80116e-07\\
6144.00000000001	2.16112e-07\\
24576.0000000001	2.23088999999999e-07\\
98304.0000000002	2.25617e-07\\
393216.000000001	2.25953999999999e-07\\
};
\addlegendentry{ACA}

\addplot [color=mycolor2]
  table[row sep=crcr]{%
6	2.06081e-16\\
24	3.0977e-16\\
96.0000000000001	2.8044e-16\\
384.000000000001	3.81906999999999e-07\\
1536	1.8126e-07\\
6144.00000000001	2.16327e-07\\
24576.0000000001	2.22671999999999e-07\\
98304.0000000002	2.2522e-07\\
393216.000000001	2.25551999999999e-07\\
};
\addlegendentry{BiLanczos}

\addplot [color=mycolor3]
  table[row sep=crcr]{%
6	1.93823e-16\\
24	3.09115000000001e-16\\
96.0000000000001	2.79345e-16\\
384.000000000001	3.81876e-07\\
1536	1.81339e-07\\
6144.00000000001	2.16347e-07\\
24576.0000000001	2.22669e-07\\
98304.0000000002	2.25218e-07\\
393216.000000001	2.25551e-07\\
};
\addlegendentry{Randomized}

\addplot [color=mycolor4]
  table[row sep=crcr]{%
6	2.09086e-16\\
24	2.99674e-16\\
96.0000000000001	2.85332e-16\\
384.000000000001	3.81864e-07\\
1536	1.81483e-07\\
6144.00000000001	2.16362e-07\\
};
\addlegendentry{SVD}

\addplot [color=mycolor5]
  table[row sep=crcr]{%
6	1.88443e-16\\
24	3.06568e-16\\
96.0000000000001	2.83968999999999e-16\\
384.000000000001	3.91301e-07\\
1536	1.81123e-07\\
6144.00000000001	2.18155e-07\\
24576.0000000001	2.25869e-07\\
98304.0000000002	2.28405e-07\\
393216.000000001	2.28695999999999e-07\\
};
\addlegendentry{Hier.~Approx.}

\end{axis}
\end{tikzpicture}%
}
\scalefig{% This file was created by matlab2tikz.
\definecolor{mycolor1}{rgb}{0.00000,0.44700,0.74100}%
\definecolor{mycolor2}{rgb}{0.85000,0.32500,0.09800}%
\definecolor{mycolor3}{rgb}{0.92900,0.69400,0.12500}%
\definecolor{mycolor4}{rgb}{0.49400,0.18400,0.55600}%
\begin{tikzpicture}

\begin{axis}[%
width=  0.5\textwidth,
height=       0.4\textwidth,
at={(         0\textwidth,         0\textwidth)},
scale only axis,
xmode=log,
xmin=1,
xmax=1000000,
xminorticks=true,
xlabel style={font=\color{white!15!black}},
xlabel={$\#\mathcal{I}$},
ymode=log,
ymin=1e-17,
ymax=1e-05,
ytick={1e-5,1e-7,1e-9,1e-11,1e-13,1e-15,1e-17},
yminorticks=true,
ylabel style={font=\color{white!15!black}},
ylabel={error},
axis background/.style={fill=white},
title={truncation to rank 16, new multiplication},
xmajorgrids,
ymajorgrids,
legend style={at={(      0.97,      0.03)}, anchor=south east, legend cell align=left, align=left, draw=white!15!black}
]
\addplot [color=mycolor1]
  table[row sep=crcr]{%
6	1.78671e-16\\
24	3.10095e-16\\
96.0000000000001	2.81555e-16\\
384.000000000001	2.40768e-06\\
1536	1.22553e-06\\
6144.00000000001	9.32552e-07\\
24576.0000000001	1.50491e-06\\
98304.0000000002	1.83797e-06\\
393216.000000001	1.78566e-06\\
};
\addlegendentry{ACA}

\addplot [color=mycolor2]
  table[row sep=crcr]{%
6	2.40383e-16\\
24	3.07905e-16\\
96.0000000000001	1.74604e-15\\
384.000000000001	3.524e-07\\
1536	1.13568e-07\\
6144.00000000001	1.47897e-07\\
24576.0000000001	1.62633e-07\\
98304.0000000002	1.88504e-07\\
393216.000000001	1.70769e-07\\
};
\addlegendentry{BiLanczos}

\addplot [color=mycolor3]
  table[row sep=crcr]{%
6	1.96871e-16\\
24	3.02722e-16\\
96.0000000000001	2.92371e-16\\
384.000000000001	3.72874999999999e-07\\
1536	1.01478e-07\\
6144.00000000001	1.49288e-07\\
24576.0000000001	1.75243e-07\\
393216.000000001	1.63382e-07\\
};
\addlegendentry{Randomized}

\addplot [color=mycolor4]
  table[row sep=crcr]{%
6	1.97368e-16\\
24	3.06593999999999e-16\\
96.0000000000001	2.84733e-16\\
384.000000000001	3.52325999999999e-07\\
1536	1.27153e-07\\
6144.00000000001	1.98134e-07\\
};
\addlegendentry{SVD}

\end{axis}
\end{tikzpicture}%
}
\caption{\label{fig:eseefr}Error using fixed rank truncation
for the product of the matrices occurring from the exponential kernels.}
\end{figure}
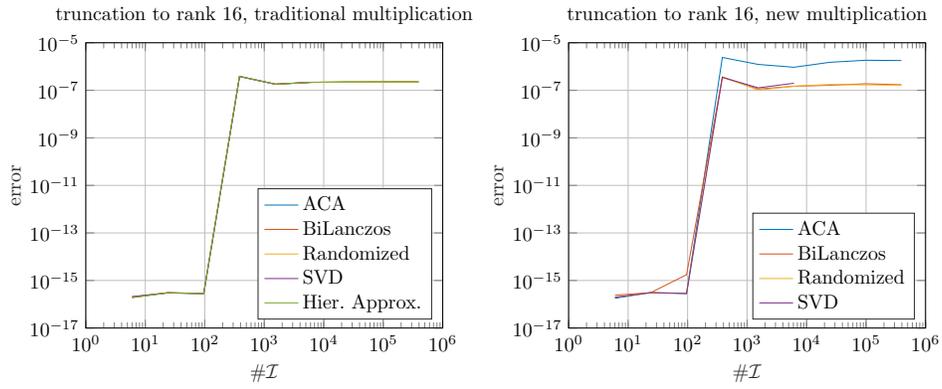
\begin{figure}
\scalefig{% This file was created by matlab2tikz.
\definecolor{mycolor1}{rgb}{0.00000,0.44700,0.74100}%
\definecolor{mycolor2}{rgb}{0.85000,0.32500,0.09800}%
\definecolor{mycolor3}{rgb}{0.92900,0.69400,0.12500}%
\definecolor{mycolor4}{rgb}{0.49400,0.18400,0.55600}%
\definecolor{mycolor5}{rgb}{0.46600,0.67400,0.18800}%
\begin{tikzpicture}

\begin{axis}[%
width=  0.5\textwidth,
height=       0.4\textwidth,
at={(         0\textwidth,         0\textwidth)},
scale only axis,
xmode=log,
xmin=1,
xmax=1000000,
xminorticks=true,
xlabel style={font=\color{white!15!black}},
xlabel={$\#\mathcal{I}$},
ymode=log,
ymin=1e-17,
ymax=1e-05,
ytick={1e-5,1e-7,1e-9,1e-11,1e-13,1e-15,1e-17},
yminorticks=true,
ylabel style={font=\color{white!15!black}},
ylabel={error},
axis background/.style={fill=white},
title={truncation error $10^{-12}$, traditional multiplication},
xmajorgrids,
ymajorgrids,
legend style={at={(      0.97,      0.03)}, anchor=south east, legend cell align=left, align=left, draw=white!15!black}
]
\addplot [color=mycolor1]
  table[row sep=crcr]{%
6	1.92742e-16\\
24	3.06532e-16\\
96.0000000000001	5.12297e-15\\
384.000000000001	5.73902e-11\\
1536	1.04966e-10\\
6144.00000000001	1.16333e-10\\
24576.0000000001	6.49643000000001e-11\\
};
\addlegendentry{ACA}

\addplot [color=mycolor2]
  table[row sep=crcr]{%
6	2.19054e-16\\
24	3.24651e-16\\
96.0000000000001	5.12707000000001e-15\\
384.000000000001	5.58396999999999e-11\\
1536	8.54886999999999e-11\\
6144.00000000001	9.07649999999999e-11\\
24576.0000000001	1.49186e-10\\
};
\addlegendentry{BiLanczos}

\addplot [color=mycolor3]
  table[row sep=crcr]{%
6	1.95154e-16\\
24	3.15543e-16\\
96.0000000000001	5.12651999999999e-15\\
384.000000000001	5.75384000000001e-11\\
1536	8.42790999999999e-11\\
6144.00000000001	1.60507e-10\\
24576.0000000001	1.37834e-10\\
};
\addlegendentry{Randomized}

\addplot [color=mycolor4]
  table[row sep=crcr]{%
6	1.4862e-16\\
24	3.23214000000001e-16\\
96.0000000000001	5.12524000000001e-15\\
384.000000000001	5.88538e-11\\
1536	2.28905e-10\\
6144.00000000001	7.68007e-11\\
};
\addlegendentry{SVD}

\addplot [color=mycolor5]
  table[row sep=crcr]{%
6	1.72786e-16\\
24	3.07925e-16\\
96.0000000000001	5.12765e-15\\
384.000000000001	5.17945e-11\\
1536	2.85593e-10\\
6144.00000000001	7.61287000000001e-11\\
};
\addlegendentry{Hier.~Approx.}

\end{axis}
\end{tikzpicture}%
}
\scalefig{% This file was created by matlab2tikz.
\definecolor{mycolor1}{rgb}{0.00000,0.44700,0.74100}%
\definecolor{mycolor2}{rgb}{0.85000,0.32500,0.09800}%
\definecolor{mycolor3}{rgb}{0.92900,0.69400,0.12500}%
\definecolor{mycolor4}{rgb}{0.49400,0.18400,0.55600}%
\begin{tikzpicture}

\begin{axis}[%
width=  0.5\textwidth,
height=       0.4\textwidth,
at={(         0\textwidth,         0\textwidth)},
scale only axis,
xmode=log,
xmin=1,
xmax=1000000,
xminorticks=true,
xlabel style={font=\color{white!15!black}},
xlabel={$\#\mathcal{I}$},
ymode=log,
ymin=1e-17,
ymax=1e-05,
ytick={1e-5,1e-7,1e-9,1e-11,1e-13,1e-15,1e-17},
yminorticks=true,
ylabel style={font=\color{white!15!black}},
ylabel={error},
axis background/.style={fill=white},
title={truncation error $10^{-12}$, new multiplication},
xmajorgrids,
ymajorgrids,
legend style={legend cell align=left, align=left, draw=white!15!black}
]
\addplot [color=mycolor1]
  table[row sep=crcr]{%
6	2.11533e-16\\
24	3.11013e-16\\
96.0000000000001	2.81138e-16\\
384.000000000001	2.69669e-16\\
1536	2.03682e-13\\
6144.00000000001	2.96182e-13\\
24576.0000000001	3.41823e-13\\
98304.0000000002	4.02508999999999e-13\\
393216.000000001	4.49078e-13\\
};
\addlegendentry{ACA}

\addplot [color=mycolor2]
  table[row sep=crcr]{%
6	2.33112e-16\\
24	3.14061e-16\\
96.0000000000001	2.25897e-08\\
384.000000000001	2.6636e-10\\
1536	9.23559000000001e-12\\
6144.00000000001	3.39399e-13\\
24576.0000000001	3.75521e-13\\
98304.0000000002	3.89015999999999e-13\\
393216.000000001	4.85945000000001e-13\\
};
\addlegendentry{BiLanczos}

\addplot [color=mycolor3]
  table[row sep=crcr]{%
6	1.82161e-16\\
24	2.95109000000001e-16\\
96.0000000000001	2.89143e-16\\
384.000000000001	3.20657000000001e-16\\
1536	1.32804e-13\\
6144.00000000001	1.24925e-13\\
24576.0000000001	1.50893e-13\\
98304.0000000002	8.53253000000002e-14\\
393216.000000001	1.72406e-14\\
};
\addlegendentry{Randomized}

\addplot [color=mycolor4]
  table[row sep=crcr]{%
6	1.91094e-16\\
24	2.87213e-16\\
96.0000000000001	1.73992e-15\\
384.000000000001	3.26057e-12\\
1536	4.40628999999999e-14\\
6144.00000000001	6.30798000000001e-12\\
};
\addlegendentry{SVD}

\end{axis}
\end{tikzpicture}%
}
\caption{\label{fig:eseeer}Error using $\varepsilon$-rank truncation
for the product of the matrices occurring from the exponential kernels.}
\end{figure}
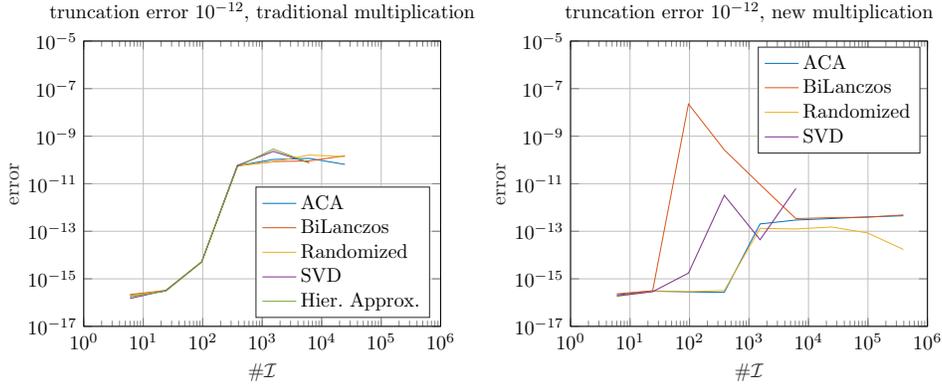
\FloatBarrier
Since the approximation quality of the low-rank methods crucially depends on the
decay of the singular values of the off-diagonal blocks, we repeat the
experiments on a different example.

\subsection{Example with matrices from the boundary element method}
The second example is concerned with the discretized boundary
integral operator
\[
[\mathbf{V}]_{ij}
=
\int_\Gamma\int_\Gamma \frac{\varphi_j(\mathbf{x})\varphi_i(\mathbf{y})}{\|\mathbf{x}-\mathbf{y}\|}\de\sigma_\mathbf{x}\de\sigma_\mathbf{y},
\]
which is frequently used in boundary element methods, see also \cite{Ste08}.
The computation times for the operation $\mathbf{V}\mathbf{V}$ can be
found in Figures~\ref{fig:vvtfr} and \ref{fig:vvter} and the corresponding
accuracies can be found in Figures~\ref{fig:vvefr} and \ref{fig:vveer}. We
can see that the 
behaviour of the traditional and the new $\mathcal{H}$-matrix multiplication
is in large parts the same as in the previous example, i.e., the new
multiplication with $\varepsilon$-rank truncation can reach higher accuracies
in shorter computation time than the traditional multiplication with
an upper bound on the used rank. However,
we shortly comment on the right figure of Figure~\ref{fig:vveer}. There,
we see, as in the previous example, that the ACA and the randomized
low-rank approximation method are more robust on small matrix sizes than
the BiLanczos algorithm. We also see that the ACA and the
randomized low-rank approximation method are even more robust than assembling
the full matrix and applying a dense SVD to obtain the low-rank blocks.
\begin{figure}
\scalefig{% This file was created by matlab2tikz.
\definecolor{mycolor1}{rgb}{0.00000,0.44700,0.74100}%
\definecolor{mycolor2}{rgb}{0.85000,0.32500,0.09800}%
\definecolor{mycolor3}{rgb}{0.92900,0.69400,0.12500}%
\definecolor{mycolor4}{rgb}{0.49400,0.18400,0.55600}%
\definecolor{mycolor5}{rgb}{0.46600,0.67400,0.18800}%
\begin{tikzpicture}

\begin{axis}[%
width=  0.5\textwidth,
height=       0.4\textwidth,
at={(         0\textwidth,         0\textwidth)},
scale only axis,
xmode=log,
xmin=1,
xmax=1000000,
xminorticks=true,
xlabel style={font=\color{white!15!black}},
xlabel={$\#\mathcal{I}$},
ymode=log,
ymin=1e-07,
ymax=1,
yminorticks=true,
ylabel style={font=\color{white!15!black}},
ylabel={s / $\#\mathcal{I}$},
axis background/.style={fill=white},
title={truncation to rank 16, traditional multiplication},
xmajorgrids,
ymajorgrids,
legend style={at={(      0.97,      0.03)}, anchor=south east, legend cell align=left, align=left, draw=white!15!black}
]
\addplot [color=mycolor1]
  table[row sep=crcr]{%
6	3.5e-06\\
24	8.75e-07\\
96.0000000000001	8.17708333333333e-06\\
384.000000000001	0.0001325\\
1536	0.00142408854166667\\
6144	0.00560195312499999\\
24576	0.0148797200520833\\
98304	0.0283530680338542\\
393216	0.0509254455566406\\
};
\addlegendentry{ACA}

\addplot [color=mycolor2]
  table[row sep=crcr]{%
6	4.83333333333334e-06\\
24	1.29166666666667e-06\\
96.0000000000001	8.04166666666666e-06\\
384.000000000001	0.0002073046875\\
1536	0.00156171223958333\\
6144	0.00597550455729167\\
24576	0.0152289632161458\\
98304	0.0283121744791666\\
393216	0.0490259806315104\\
};
\addlegendentry{BiLanczos}

\addplot [color=mycolor3]
  table[row sep=crcr]{%
6	2.16666666666667e-06\\
24	1.45833333333333e-06\\
96.0000000000001	1.09375e-05\\
384.000000000001	0.0002015390625\\
1536	0.0015446484375\\
6144	0.00618435872395833\\
24576	0.0147998453776042\\
98304	0.0276812744140625\\
393216	0.0493700663248698\\
};
\addlegendentry{Randomized}

\addplot [color=mycolor4]
  table[row sep=crcr]{%
6	4.5e-06\\
24	1.83333333333333e-06\\
96.0000000000001	1.06041666666667e-05\\
384.000000000001	0.000876302083333332\\
1536	0.0256916015625\\
6144	0.851274414062499\\
};
\addlegendentry{SVD}

\addplot [color=mycolor5]
  table[row sep=crcr]{%
6	2.66666666666667e-06\\
24	1.16666666666667e-06\\
96.0000000000001	8.35416666666667e-06\\
384.000000000001	0.00125895833333333\\
1536	0.00867819010416667\\
6144	0.0206385091145833\\
24576	0.0370115559895833\\
98304	0.0596094767252604\\
393216	0.0929964701334635\\
};
\addlegendentry{Hier.~Approx.}

\addplot [color=black, dashed, forget plot]
  table[row sep=crcr]{%
6	0.0016052009977842\\
6.71123680053168	0.00181219534288484\\
7.50678323213512	0.00203173901678489\\
8.39663331351989	0.00226383201948436\\
10.5052843607197	0.00276566601128153\\
13.1434820813104	0.00331769731827637\\
16.4442118166415	0.00391992594046886\\
20.5738555884738	0.00457235187785901\\
25.7405790253207	0.00527497513044681\\
32.204824502121	0.00602779569823228\\
40.2924394277295	0.00683081358121539\\
50.4111014463166	0.00768402877939617\\
63.0708685084453	0.0085874412927746\\
88.2638259384407	0.0100366800288379\\
123.519830208906	0.0115988627240958\\
172.858453535403	0.0132739893785486\\
241.904841579816	0.015062059992196\\
338.531041918473	0.0169630745650382\\
473.753504038874	0.0189770330970751\\
741.579447237127	0.0218380017429827\\
1160.81479477437	0.024899759649681\\
1817.0554655841	0.0281623068171699\\
2844.28711614656	0.0316256432454495\\
4980.0074931085	0.0362371736787858\\
8719.39913893663	0.0411624373321074\\
15266.627901524	0.0464014342054144\\
29898.6097837907	0.0531023583037634\\
58554.3102753	0.0602550582388914\\
114674.47070649	0.0678595340107983\\
251203.641819084	0.0773024168717296\\
393216	0.0829744351687061\\
};
\addplot [color=black, dashed, forget plot]
  table[row sep=crcr]{%
6	0.000223265685105974\\
6.71123680053168	0.000267815286716958\\
7.50678323213512	0.000317928695955608\\
8.39663331351989	0.000373933302401214\\
9.39196574904416	0.000436156495633068\\
10.5052843607197	0.000504925665230462\\
11.750575166952	0.000580568200772688\\
14.7015034385365	0.0007537829280088\\
18.3934999165971	0.000958419793975735\\
23.0126694590316	0.00119709791530783\\
28.7918535369543	0.00147243640863942\\
36.0223672255499	0.00178705439060482\\
45.0686837117619	0.00214357097783839\\
56.3868065303093	0.00254460528697445\\
70.5472556292287	0.00299277643464732\\
88.2638259384407	0.00349070353749135\\
110.42956809313	0.00404100571214088\\
138.161805015572	0.00464630207523021\\
193.349002441632	0.0056630900408265\\
270.580112506207	0.00681834746148049\\
378.660331107598	0.00812091385583309\\
529.911991781088	0.00957962874252515\\
741.579447237127	0.0112033316401975\\
1037.79511521549	0.0130008620674912\\
1452.33083950443	0.0149810595430469\\
2273.37358348436	0.0179207497658958\\
3558.57447181116	0.0212218490407579\\
5570.33492577888	0.024905310300708\\
8719.39913893663	0.0289920864788207\\
13648.7163441897	0.0335031305081706\\
23897.2743924192	0.0397704826946413\\
41841.2771564182	0.0467744160540108\\
73259.0857573098	0.0545558542836906\\
128267.921314482	0.0631557210810921\\
251203.641819084	0.0746135075133219\\
393216	0.0829744351687061\\
};
\addplot [color=black, dashed, forget plot]
  table[row sep=crcr]{%
6	3.10537846753452e-05\\
6.71123680053168	3.95790818472734e-05\\
7.50678323213512	4.97498226282944e-05\\
8.39663331351989	6.17652340991831e-05\\
9.39196574904416	7.58359313533958e-05\\
10.5052843607197	9.21839174970691e-05\\
11.750575166952	0.000111042583649021\\
13.1434820813104	0.00013265670894075\\
16.4442118166415	0.000185187393532941\\
20.5738555884738	0.000251961964579249\\
25.7405790253207	0.000335348623588181\\
32.204824502121	0.000437897780271133\\
40.2924394277295	0.000562342052542401\\
50.4111014463166	0.000711596266519174\\
63.0708685084453	0.000888757456521537\\
78.909889692566	0.00109710486507246\\
98.7265727989643	0.00134009994289783\\
123.519830208906	0.00162138634892641\\
154.539431708065	0.00194478995028987\\
193.349002441632	0.00231431882232276\\
241.904841579816	0.00273416324856254\\
302.654534757277	0.00320869572074955\\
423.54652483847	0.00403304774386247\\
592.727476714046	0.00500687945344094\\
829.485879469291	0.00614725408167652\\
1160.81479477437	0.00747215728978795\\
1624.48936277154	0.00900049716802108\\
2273.37358348436	0.010752104235649\\
3181.44739419349	0.0127477314409718\\
4452.24119588268	0.015009054161317\\
6969.21613636892	0.0184766460533132\\
10909.1065417438	0.0225129817791119\\
17076.3258321219	0.0271770927554179\\
26730.0445558012	0.0325309257301824\\
41841.2771564182	0.0386393427826031\\
65495.3069915573	0.0455701213231238\\
114674.47070649	0.0554980138978931\\
200781.320614451	0.0669659947501943\\
351544.144562608	0.0801256589061995\\
393216	0.0829744351687061\\
};
\end{axis}
\end{tikzpicture}%
}
\scalefig{% This file was created by matlab2tikz.
\definecolor{mycolor1}{rgb}{0.00000,0.44700,0.74100}%
\definecolor{mycolor2}{rgb}{0.85000,0.32500,0.09800}%
\definecolor{mycolor3}{rgb}{0.92900,0.69400,0.12500}%
\definecolor{mycolor4}{rgb}{0.49400,0.18400,0.55600}%
\begin{tikzpicture}

\begin{axis}[%
width=  0.5\textwidth,
height=       0.4\textwidth,
at={(         0\textwidth,         0\textwidth)},
scale only axis,
xmode=log,
xmin=1,
xmax=1000000,
xminorticks=true,
xlabel style={font=\color{white!15!black}},
xlabel={$\#\mathcal{I}$},
ymode=log,
ymin=1e-07,
ymax=1,
yminorticks=true,
ylabel style={font=\color{white!15!black}},
ylabel={s / $\#\mathcal{I}$},
axis background/.style={fill=white},
title={truncation to rank 16, new multiplication},
xmajorgrids,
ymajorgrids,
legend style={at={(      0.97,      0.03)}, anchor=south east, legend cell align=left, align=left, draw=white!15!black}
]
\addplot [color=mycolor1]
  table[row sep=crcr]{%
6	2.5e-06\\
24	1.33333333333333e-06\\
96.0000000000001	2.65625e-05\\
384.000000000001	0.000198291666666667\\
1536	0.000835032552083333\\
6144	0.00208678385416667\\
24576	0.00494095865885417\\
98304	0.00905853271484374\\
393216	0.0190629069010417\\
};
\addlegendentry{ACA}

\addplot [color=mycolor2]
  table[row sep=crcr]{%
6	4.16666666666666e-06\\
24	1.70833333333333e-06\\
96.0000000000001	2.74791666666667e-05\\
384.000000000001	0.0002451953125\\
1536	0.0008428125\\
6144	0.00188037109375\\
24576	0.00406024576822917\\
98304	0.00762589518229166\\
393216	0.0141360473632812\\
};
\addlegendentry{BiLanczos}

\addplot [color=mycolor3]
  table[row sep=crcr]{%
6	5.5e-06\\
24	1.70833333333333e-06\\
96.0000000000001	1.71354166666667e-05\\
384.000000000001	0.0002895703125\\
1536	0.00081376953125\\
6144	0.002156640625\\
24576	0.00467451985677083\\
98304	0.00863125610351562\\
393216	0.0157177734375\\
};
\addlegendentry{Randomized}

\addplot [color=mycolor4]
  table[row sep=crcr]{%
6	6e-06\\
24	1.95833333333333e-06\\
96.0000000000001	1.66979166666667e-05\\
384.000000000001	0.000316888020833333\\
1536	0.00706790364583332\\
6144	0.270257161458334\\
};
\addlegendentry{SVD}

\addplot [color=black, dashed, forget plot]
  table[row sep=crcr]{%
6	0.0016052009977842\\
6.71123680053168	0.00181219534288484\\
7.50678323213512	0.00203173901678489\\
8.39663331351989	0.00226383201948436\\
10.5052843607197	0.00276566601128153\\
13.1434820813104	0.00331769731827637\\
16.4442118166415	0.00391992594046886\\
20.5738555884738	0.00457235187785901\\
25.7405790253207	0.00527497513044681\\
32.204824502121	0.00602779569823228\\
40.2924394277295	0.00683081358121539\\
50.4111014463166	0.00768402877939617\\
63.0708685084453	0.0085874412927746\\
88.2638259384407	0.0100366800288379\\
123.519830208906	0.0115988627240958\\
172.858453535403	0.0132739893785486\\
241.904841579816	0.015062059992196\\
338.531041918473	0.0169630745650382\\
473.753504038874	0.0189770330970751\\
741.579447237127	0.0218380017429827\\
1160.81479477437	0.024899759649681\\
1817.0554655841	0.0281623068171699\\
2844.28711614656	0.0316256432454495\\
4980.0074931085	0.0362371736787858\\
8719.39913893663	0.0411624373321074\\
15266.627901524	0.0464014342054144\\
29898.6097837907	0.0531023583037634\\
58554.3102753	0.0602550582388914\\
114674.47070649	0.0678595340107983\\
251203.641819084	0.0773024168717296\\
393216	0.0829744351687061\\
};
\addplot [color=black, dashed, forget plot]
  table[row sep=crcr]{%
6	0.000223265685105974\\
6.71123680053168	0.000267815286716958\\
7.50678323213512	0.000317928695955608\\
8.39663331351989	0.000373933302401214\\
9.39196574904416	0.000436156495633068\\
10.5052843607197	0.000504925665230462\\
11.750575166952	0.000580568200772688\\
14.7015034385365	0.0007537829280088\\
18.3934999165971	0.000958419793975735\\
23.0126694590316	0.00119709791530783\\
28.7918535369543	0.00147243640863942\\
36.0223672255499	0.00178705439060482\\
45.0686837117619	0.00214357097783839\\
56.3868065303093	0.00254460528697445\\
70.5472556292287	0.00299277643464732\\
88.2638259384407	0.00349070353749135\\
110.42956809313	0.00404100571214088\\
138.161805015572	0.00464630207523021\\
193.349002441632	0.0056630900408265\\
270.580112506207	0.00681834746148049\\
378.660331107598	0.00812091385583309\\
529.911991781088	0.00957962874252515\\
741.579447237127	0.0112033316401975\\
1037.79511521549	0.0130008620674912\\
1452.33083950443	0.0149810595430469\\
2273.37358348436	0.0179207497658958\\
3558.57447181116	0.0212218490407579\\
5570.33492577888	0.024905310300708\\
8719.39913893663	0.0289920864788207\\
13648.7163441897	0.0335031305081706\\
23897.2743924192	0.0397704826946413\\
41841.2771564182	0.0467744160540108\\
73259.0857573098	0.0545558542836906\\
128267.921314482	0.0631557210810921\\
251203.641819084	0.0746135075133219\\
393216	0.0829744351687061\\
};
\addplot [color=black, dashed, forget plot]
  table[row sep=crcr]{%
6	3.10537846753452e-05\\
6.71123680053168	3.95790818472734e-05\\
7.50678323213512	4.97498226282944e-05\\
8.39663331351989	6.17652340991831e-05\\
9.39196574904416	7.58359313533958e-05\\
10.5052843607197	9.21839174970691e-05\\
11.750575166952	0.000111042583649021\\
13.1434820813104	0.00013265670894075\\
16.4442118166415	0.000185187393532941\\
20.5738555884738	0.000251961964579249\\
25.7405790253207	0.000335348623588181\\
32.204824502121	0.000437897780271133\\
40.2924394277295	0.000562342052542401\\
50.4111014463166	0.000711596266519174\\
63.0708685084453	0.000888757456521537\\
78.909889692566	0.00109710486507246\\
98.7265727989643	0.00134009994289783\\
123.519830208906	0.00162138634892641\\
154.539431708065	0.00194478995028987\\
193.349002441632	0.00231431882232276\\
241.904841579816	0.00273416324856254\\
302.654534757277	0.00320869572074955\\
423.54652483847	0.00403304774386247\\
592.727476714046	0.00500687945344094\\
829.485879469291	0.00614725408167652\\
1160.81479477437	0.00747215728978795\\
1624.48936277154	0.00900049716802108\\
2273.37358348436	0.010752104235649\\
3181.44739419349	0.0127477314409718\\
4452.24119588268	0.015009054161317\\
6969.21613636892	0.0184766460533132\\
10909.1065417438	0.0225129817791119\\
17076.3258321219	0.0271770927554179\\
26730.0445558012	0.0325309257301824\\
41841.2771564182	0.0386393427826031\\
65495.3069915573	0.0455701213231238\\
114674.47070649	0.0554980138978931\\
200781.320614451	0.0669659947501943\\
351544.144562608	0.0801256589061995\\
393216	0.0829744351687061\\
};
\end{axis}
\end{tikzpicture}%
}
\caption{\label{fig:vvtfr}Computation times in seconds per degree of freedom
for the product of the boundary integral operator matrices
using fixed rank truncation with corresponding asymptotics $N\log^2N$,
$N\log^3N$, and $N\log^4N$.}
\end{figure}
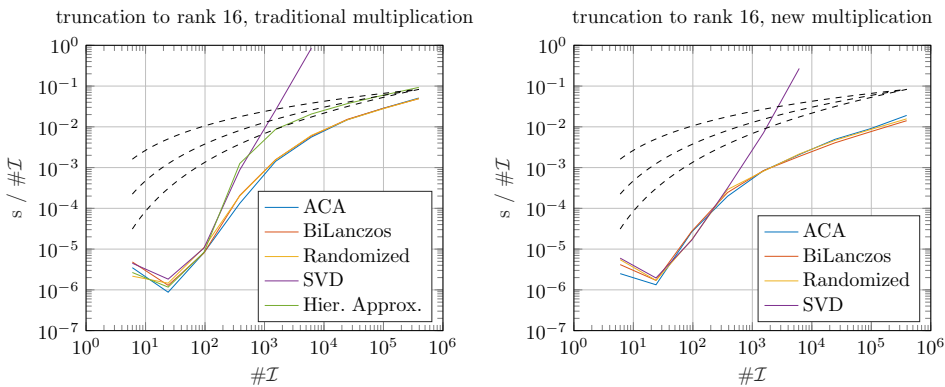
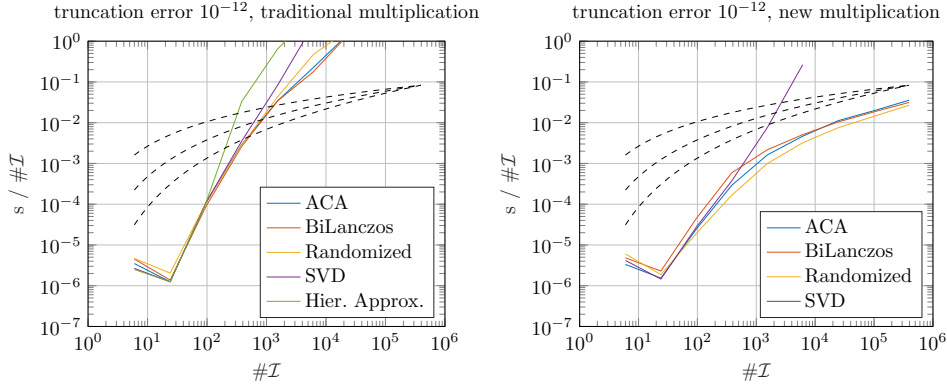
\begin{figure}
\scalefig{% This file was created by matlab2tikz.
\definecolor{mycolor1}{rgb}{0.00000,0.44700,0.74100}%
\definecolor{mycolor2}{rgb}{0.85000,0.32500,0.09800}%
\definecolor{mycolor3}{rgb}{0.92900,0.69400,0.12500}%
\definecolor{mycolor4}{rgb}{0.49400,0.18400,0.55600}%
\definecolor{mycolor5}{rgb}{0.46600,0.67400,0.18800}%
\begin{tikzpicture}

\begin{axis}[%
width=  0.5\textwidth,
height=       0.4\textwidth,
at={(         0\textwidth,         0\textwidth)},
scale only axis,
xmode=log,
xmin=1,
xmax=1000000,
xminorticks=true,
xlabel style={font=\color{white!15!black}},
xlabel={$\#\mathcal{I}$},
ymode=log,
ymin=1e-07,
ymax=1,
yminorticks=true,
ylabel style={font=\color{white!15!black}},
ylabel={s / $\#\mathcal{I}$},
axis background/.style={fill=white},
title={truncation error $10^{-12}$, traditional multiplication},
xmajorgrids,
ymajorgrids,
legend style={at={(      0.97,      0.03)}, anchor=south east, legend cell align=left, align=left, draw=white!15!black}
]
\addplot [color=mycolor1]
  table[row sep=crcr]{%
6	3.5e-06\\
24	1.25e-06\\
96.0000000000001	0.000108\\
384.000000000001	0.002918671875\\
1536	0.0342668619791667\\
6144	0.226663411458333\\
24576	1.58812255859375\\
};
\addlegendentry{ACA}

\addplot [color=mycolor2]
  table[row sep=crcr]{%
6	4.5e-06\\
24	1.375e-06\\
96.0000000000001	8.78437499999999e-05\\
384.000000000001	0.00263424479166667\\
1536	0.0340916666666666\\
6144	0.178269856770833\\
24576	1.53498128255208\\
};
\addlegendentry{BiLanczos}

\addplot [color=mycolor3]
  table[row sep=crcr]{%
6	4.66666666666667e-06\\
24	2.04166666666667e-06\\
96.0000000000001	0.00010771875\\
384.000000000001	0.00280489583333333\\
1536	0.0431377604166667\\
6144	0.472513020833334\\
24576	1.96863199869792\\
};
\addlegendentry{Randomized}

\addplot [color=mycolor4]
  table[row sep=crcr]{%
6	2.66666666666667e-06\\
24	1.25e-06\\
96.0000000000001	0.000109083333333333\\
384.000000000001	0.00356653645833333\\
1536	0.0828795572916667\\
6144	2.53533528645833\\
};
\addlegendentry{SVD}

\addplot [color=mycolor5]
  table[row sep=crcr]{%
6	2.5e-06\\
24	1.25e-06\\
96.0000000000001	0.000108697916666667\\
384.000000000001	0.0332708333333333\\
1536	0.64291015625\\
6144	4.93097330729167\\
};
\addlegendentry{Hier.~Approx.}

\addplot [color=black, dashed, forget plot]
  table[row sep=crcr]{%
6	0.0016052009977842\\
6.71123680053168	0.00181219534288484\\
7.50678323213512	0.00203173901678489\\
8.39663331351989	0.00226383201948436\\
10.5052843607197	0.00276566601128153\\
13.1434820813104	0.00331769731827637\\
16.4442118166415	0.00391992594046886\\
20.5738555884738	0.00457235187785901\\
25.7405790253207	0.00527497513044681\\
32.204824502121	0.00602779569823228\\
40.2924394277295	0.00683081358121539\\
50.4111014463166	0.00768402877939617\\
63.0708685084453	0.0085874412927746\\
88.2638259384407	0.0100366800288379\\
123.519830208906	0.0115988627240958\\
172.858453535403	0.0132739893785486\\
241.904841579816	0.015062059992196\\
338.531041918473	0.0169630745650382\\
473.753504038874	0.0189770330970751\\
741.579447237127	0.0218380017429827\\
1160.81479477437	0.024899759649681\\
1817.0554655841	0.0281623068171699\\
2844.28711614656	0.0316256432454495\\
4980.0074931085	0.0362371736787858\\
8719.39913893663	0.0411624373321074\\
15266.627901524	0.0464014342054144\\
29898.6097837907	0.0531023583037634\\
58554.3102753	0.0602550582388914\\
114674.47070649	0.0678595340107983\\
251203.641819084	0.0773024168717296\\
393216	0.0829744351687061\\
};
\addplot [color=black, dashed, forget plot]
  table[row sep=crcr]{%
6	0.000223265685105974\\
6.71123680053168	0.000267815286716958\\
7.50678323213512	0.000317928695955608\\
8.39663331351989	0.000373933302401214\\
9.39196574904416	0.000436156495633068\\
10.5052843607197	0.000504925665230462\\
11.750575166952	0.000580568200772688\\
14.7015034385365	0.0007537829280088\\
18.3934999165971	0.000958419793975735\\
23.0126694590316	0.00119709791530783\\
28.7918535369543	0.00147243640863942\\
36.0223672255499	0.00178705439060482\\
45.0686837117619	0.00214357097783839\\
56.3868065303093	0.00254460528697445\\
70.5472556292287	0.00299277643464732\\
88.2638259384407	0.00349070353749135\\
110.42956809313	0.00404100571214088\\
138.161805015572	0.00464630207523021\\
193.349002441632	0.0056630900408265\\
270.580112506207	0.00681834746148049\\
378.660331107598	0.00812091385583309\\
529.911991781088	0.00957962874252515\\
741.579447237127	0.0112033316401975\\
1037.79511521549	0.0130008620674912\\
1452.33083950443	0.0149810595430469\\
2273.37358348436	0.0179207497658958\\
3558.57447181116	0.0212218490407579\\
5570.33492577888	0.024905310300708\\
8719.39913893663	0.0289920864788207\\
13648.7163441897	0.0335031305081706\\
23897.2743924192	0.0397704826946413\\
41841.2771564182	0.0467744160540108\\
73259.0857573098	0.0545558542836906\\
128267.921314482	0.0631557210810921\\
251203.641819084	0.0746135075133219\\
393216	0.0829744351687061\\
};
\addplot [color=black, dashed, forget plot]
  table[row sep=crcr]{%
6	3.10537846753452e-05\\
6.71123680053168	3.95790818472734e-05\\
7.50678323213512	4.97498226282944e-05\\
8.39663331351989	6.17652340991831e-05\\
9.39196574904416	7.58359313533958e-05\\
10.5052843607197	9.21839174970691e-05\\
11.750575166952	0.000111042583649021\\
13.1434820813104	0.00013265670894075\\
16.4442118166415	0.000185187393532941\\
20.5738555884738	0.000251961964579249\\
25.7405790253207	0.000335348623588181\\
32.204824502121	0.000437897780271133\\
40.2924394277295	0.000562342052542401\\
50.4111014463166	0.000711596266519174\\
63.0708685084453	0.000888757456521537\\
78.909889692566	0.00109710486507246\\
98.7265727989643	0.00134009994289783\\
123.519830208906	0.00162138634892641\\
154.539431708065	0.00194478995028987\\
193.349002441632	0.00231431882232276\\
241.904841579816	0.00273416324856254\\
302.654534757277	0.00320869572074955\\
423.54652483847	0.00403304774386247\\
592.727476714046	0.00500687945344094\\
829.485879469291	0.00614725408167652\\
1160.81479477437	0.00747215728978795\\
1624.48936277154	0.00900049716802108\\
2273.37358348436	0.010752104235649\\
3181.44739419349	0.0127477314409718\\
4452.24119588268	0.015009054161317\\
6969.21613636892	0.0184766460533132\\
10909.1065417438	0.0225129817791119\\
17076.3258321219	0.0271770927554179\\
26730.0445558012	0.0325309257301824\\
41841.2771564182	0.0386393427826031\\
65495.3069915573	0.0455701213231238\\
114674.47070649	0.0554980138978931\\
200781.320614451	0.0669659947501943\\
351544.144562608	0.0801256589061995\\
393216	0.0829744351687061\\
};
\end{axis}
\end{tikzpicture}%
}
\scalefig{% This file was created by matlab2tikz.
\definecolor{mycolor1}{rgb}{0.00000,0.44700,0.74100}%
\definecolor{mycolor2}{rgb}{0.85000,0.32500,0.09800}%
\definecolor{mycolor3}{rgb}{0.92900,0.69400,0.12500}%
\definecolor{mycolor4}{rgb}{0.49400,0.18400,0.55600}%
\begin{tikzpicture}

\begin{axis}[%
width=  0.495\textwidth,
height=       0.4\textwidth,
at={(         0\textwidth,         0\textwidth)},
scale only axis,
xmode=log,
xmin=1,
xmax=1000000,
xminorticks=true,
xlabel style={font=\color{white!15!black}},
xlabel={$\#\mathcal{I}$},
ymode=log,
ymin=1e-07,
ymax=1,
yminorticks=true,
ylabel style={font=\color{white!15!black}},
ylabel={s / $\#\mathcal{I}$},
axis background/.style={fill=white},
title={truncation error $10^{-12}$, new multiplication},
xmajorgrids,
ymajorgrids,
legend style={at={(      0.97,      0.03)}, anchor=south east, legend cell align=left, align=left, draw=white!15!black}
]
\addplot [color=mycolor1]
  table[row sep=crcr]{%
6	3.33333333333333e-06\\
24	1.54166666666667e-06\\
96.0000000000001	2.45729166666667e-05\\
384.000000000001	0.000291283854166666\\
1536	0.00161681640625\\
6144	0.00466544596354166\\
24576	0.0112059326171875\\
98304	0.0195994059244791\\
393216	0.0356096903483073\\
};
\addlegendentry{ACA}

\addplot [color=mycolor2]
  table[row sep=crcr]{%
6	4.83333333333334e-06\\
24	2.29166666666667e-06\\
96.0000000000001	4.490625e-05\\
384.000000000001	0.000590830729166666\\
1536	0.00216278645833333\\
6144	0.00509912109375\\
24576	0.0104352620442708\\
98304	0.0183473714192708\\
393216	0.031433359781901\\
};
\addlegendentry{BiLanczos}

\addplot [color=mycolor3]
  table[row sep=crcr]{%
6	6e-06\\
24	1.875e-06\\
96.0000000000001	1.92083333333334e-05\\
384.000000000001	0.000166721354166667\\
1536	0.00099318359375\\
6144	0.00321539713541667\\
24576	0.00753470865885416\\
393216	0.0268613179524739\\
};
\addlegendentry{Randomized}

\addplot [color=mycolor4]
  table[row sep=crcr]{%
6	4.16666666666666e-06\\
24	1.45833333333333e-06\\
96.0000000000001	2.75104166666667e-05\\
384.000000000001	0.000374385416666666\\
1536	0.00756126302083333\\
6144	0.265460611979167\\
};
\addlegendentry{SVD}

\addplot [color=black, dashed, forget plot]
  table[row sep=crcr]{%
6	0.0016052009977842\\
6.71123680053168	0.00181219534288484\\
7.50678323213512	0.00203173901678489\\
8.39663331351989	0.00226383201948436\\
10.5052843607197	0.00276566601128153\\
13.1434820813104	0.00331769731827637\\
16.4442118166415	0.00391992594046886\\
20.5738555884738	0.00457235187785901\\
25.7405790253207	0.00527497513044681\\
32.204824502121	0.00602779569823228\\
40.2924394277295	0.00683081358121539\\
50.4111014463166	0.00768402877939617\\
63.0708685084453	0.0085874412927746\\
88.2638259384407	0.0100366800288379\\
123.519830208906	0.0115988627240958\\
172.858453535403	0.0132739893785486\\
241.904841579816	0.015062059992196\\
338.531041918473	0.0169630745650382\\
473.753504038874	0.0189770330970751\\
741.579447237127	0.0218380017429827\\
1160.81479477437	0.024899759649681\\
1817.0554655841	0.0281623068171699\\
2844.28711614656	0.0316256432454495\\
4980.0074931085	0.0362371736787858\\
8719.39913893663	0.0411624373321074\\
15266.627901524	0.0464014342054144\\
29898.6097837907	0.0531023583037634\\
58554.3102753	0.0602550582388914\\
114674.47070649	0.0678595340107983\\
251203.641819084	0.0773024168717296\\
393216	0.0829744351687061\\
};
\addplot [color=black, dashed, forget plot]
  table[row sep=crcr]{%
6	0.000223265685105974\\
6.71123680053168	0.000267815286716958\\
7.50678323213512	0.000317928695955608\\
8.39663331351989	0.000373933302401214\\
9.39196574904416	0.000436156495633068\\
10.5052843607197	0.000504925665230462\\
11.750575166952	0.000580568200772688\\
14.7015034385365	0.0007537829280088\\
18.3934999165971	0.000958419793975735\\
23.0126694590316	0.00119709791530783\\
28.7918535369543	0.00147243640863942\\
36.0223672255499	0.00178705439060482\\
45.0686837117619	0.00214357097783839\\
56.3868065303093	0.00254460528697445\\
70.5472556292287	0.00299277643464732\\
88.2638259384407	0.00349070353749135\\
110.42956809313	0.00404100571214088\\
138.161805015572	0.00464630207523021\\
193.349002441632	0.0056630900408265\\
270.580112506207	0.00681834746148049\\
378.660331107598	0.00812091385583309\\
529.911991781088	0.00957962874252515\\
741.579447237127	0.0112033316401975\\
1037.79511521549	0.0130008620674912\\
1452.33083950443	0.0149810595430469\\
2273.37358348436	0.0179207497658958\\
3558.57447181116	0.0212218490407579\\
5570.33492577888	0.024905310300708\\
8719.39913893663	0.0289920864788207\\
13648.7163441897	0.0335031305081706\\
23897.2743924192	0.0397704826946413\\
41841.2771564182	0.0467744160540108\\
73259.0857573098	0.0545558542836906\\
128267.921314482	0.0631557210810921\\
251203.641819084	0.0746135075133219\\
393216	0.0829744351687061\\
};
\addplot [color=black, dashed, forget plot]
  table[row sep=crcr]{%
6	3.10537846753452e-05\\
6.71123680053168	3.95790818472734e-05\\
7.50678323213512	4.97498226282944e-05\\
8.39663331351989	6.17652340991831e-05\\
9.39196574904416	7.58359313533958e-05\\
10.5052843607197	9.21839174970691e-05\\
11.750575166952	0.000111042583649021\\
13.1434820813104	0.00013265670894075\\
16.4442118166415	0.000185187393532941\\
20.5738555884738	0.000251961964579249\\
25.7405790253207	0.000335348623588181\\
32.204824502121	0.000437897780271133\\
40.2924394277295	0.000562342052542401\\
50.4111014463166	0.000711596266519174\\
63.0708685084453	0.000888757456521537\\
78.909889692566	0.00109710486507246\\
98.7265727989643	0.00134009994289783\\
123.519830208906	0.00162138634892641\\
154.539431708065	0.00194478995028987\\
193.349002441632	0.00231431882232276\\
241.904841579816	0.00273416324856254\\
302.654534757277	0.00320869572074955\\
423.54652483847	0.00403304774386247\\
592.727476714046	0.00500687945344094\\
829.485879469291	0.00614725408167652\\
1160.81479477437	0.00747215728978795\\
1624.48936277154	0.00900049716802108\\
2273.37358348436	0.010752104235649\\
3181.44739419349	0.0127477314409718\\
4452.24119588268	0.015009054161317\\
6969.21613636892	0.0184766460533132\\
10909.1065417438	0.0225129817791119\\
17076.3258321219	0.0271770927554179\\
26730.0445558012	0.0325309257301824\\
41841.2771564182	0.0386393427826031\\
65495.3069915573	0.0455701213231238\\
114674.47070649	0.0554980138978931\\
200781.320614451	0.0669659947501943\\
351544.144562608	0.0801256589061995\\
393216	0.0829744351687061\\
};
\end{axis}
\end{tikzpicture}%
}
\caption{\label{fig:vvter}Computation times in seconds per degree of freedom
	for the product of the boundary integral operator matrices
using $\varepsilon$-rank truncation with corresponding asymptotics $N\log^2N$,
$N\log^3N$, and $N\log^4N$.}
\end{figure}
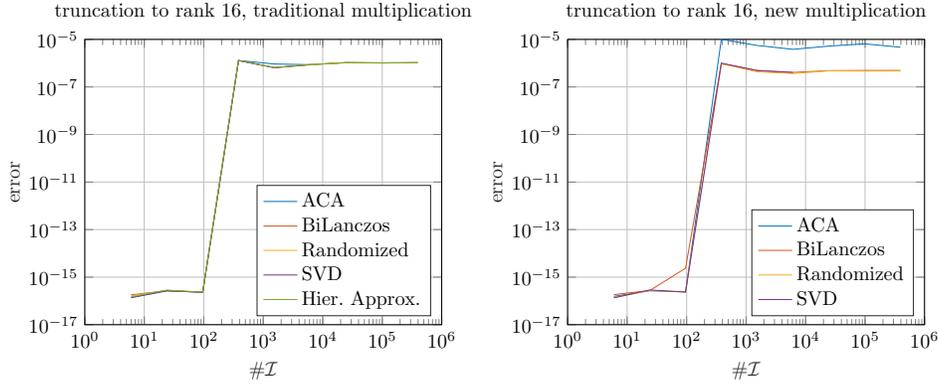
\begin{figure}
\scalefig{% This file was created by matlab2tikz.
\definecolor{mycolor1}{rgb}{0.00000,0.44700,0.74100}%
\definecolor{mycolor2}{rgb}{0.85000,0.32500,0.09800}%
\definecolor{mycolor3}{rgb}{0.92900,0.69400,0.12500}%
\definecolor{mycolor4}{rgb}{0.49400,0.18400,0.55600}%
\definecolor{mycolor5}{rgb}{0.46600,0.67400,0.18800}%
\begin{tikzpicture}

\begin{axis}[%
width=  0.5\textwidth,
height=       0.4\textwidth,
at={(         0\textwidth,         0\textwidth)},
scale only axis,
xmode=log,
xmin=1,
xmax=1000000,
xminorticks=true,
xlabel style={font=\color{white!15!black}},
xlabel={$\#\mathcal{I}$},
ymode=log,
ymin=1e-17,
ymax=1e-05,
ytick={1e-5,1e-7,1e-9,1e-11,1e-13,1e-15,1e-17},
yminorticks=true,
ylabel style={font=\color{white!15!black}},
ylabel={error},
axis background/.style={fill=white},
title={truncation to rank 16, traditional multiplication},
xmajorgrids,
ymajorgrids,
legend style={at={(      0.97,      0.03)}, anchor=south east, legend cell align=left, align=left, draw=white!15!black}
]
\addplot [color=mycolor1]
  table[row sep=crcr]{%
6	1.35682e-16\\
24	2.62659e-16\\
96.0000000000001	2.31492e-16\\
384.000000000001	1.29858e-06\\
1536	9.22875000000001e-07\\
6144.00000000001	8.63126e-07\\
24576.0000000001	1.05431e-06\\
98304.0000000002	1.02099e-06\\
393216.000000001	1.05088e-06\\
};
\addlegendentry{ACA}

\addplot [color=mycolor2]
  table[row sep=crcr]{%
6	1.80007e-16\\
24	2.67249e-16\\
96.0000000000001	2.30306e-16\\
384.000000000001	1.28732e-06\\
1536	6.46073e-07\\
6144.00000000001	8.64728999999998e-07\\
24576.0000000001	1.05109e-06\\
98304.0000000002	1.0393e-06\\
393216.000000001	1.04462e-06\\
};
\addlegendentry{BiLanczos}

\addplot [color=mycolor3]
  table[row sep=crcr]{%
6	1.63921e-16\\
24	2.7863e-16\\
96.0000000000001	2.34415e-16\\
384.000000000001	1.2874e-06\\
1536	6.46136999999999e-07\\
6144.00000000001	8.61528e-07\\
24576.0000000001	1.051e-06\\
98304.0000000002	1.03233e-06\\
393216.000000001	1.04428e-06\\
};
\addlegendentry{Randomized}

\addplot [color=mycolor4]
  table[row sep=crcr]{%
6	1.40388e-16\\
24	2.69977e-16\\
96.0000000000001	2.31894e-16\\
384.000000000001	1.28779e-06\\
1536	6.46064999999999e-07\\
6144.00000000001	8.63163999999998e-07\\
};
\addlegendentry{SVD}

\addplot [color=mycolor5]
  table[row sep=crcr]{%
6	1.46999e-16\\
24	2.85084e-16\\
96.0000000000001	2.31204e-16\\
384.000000000001	1.29641e-06\\
1536	6.55112e-07\\
6144.00000000001	8.76726999999998e-07\\
24576.0000000001	1.05384e-06\\
98304.0000000002	1.01298e-06\\
393216.000000001	1.04472e-06\\
};
\addlegendentry{Hier.~Approx.}

\end{axis}
\end{tikzpicture}%
}
\scalefig{% This file was created by matlab2tikz.
\definecolor{mycolor1}{rgb}{0.00000,0.44700,0.74100}%
\definecolor{mycolor2}{rgb}{0.85000,0.32500,0.09800}%
\definecolor{mycolor3}{rgb}{0.92900,0.69400,0.12500}%
\definecolor{mycolor4}{rgb}{0.49400,0.18400,0.55600}%
\begin{tikzpicture}

\begin{axis}[%
width=  0.5\textwidth,
height=       0.4\textwidth,
at={(         0\textwidth,         0\textwidth)},
scale only axis,
xmode=log,
xmin=1,
xmax=1000000,
xminorticks=true,
xlabel style={font=\color{white!15!black}},
xlabel={$\#\mathcal{I}$},
ymode=log,
ymin=1e-17,
ymax=1e-05,
ytick={1e-5,1e-7,1e-9,1e-11,1e-13,1e-15,1e-17},
yminorticks=true,
ylabel style={font=\color{white!15!black}},
ylabel={error},
axis background/.style={fill=white},
title={truncation to rank 16, new multiplication},
xmajorgrids,
ymajorgrids,
legend style={at={(      0.97,      0.03)}, anchor=south east, legend cell align=left, align=left, draw=white!15!black}
]
\addplot [color=mycolor1]
  table[row sep=crcr]{%
6	1.55332e-16\\
24	2.66501e-16\\
96.0000000000001	2.38905e-16\\
384.000000000001	1.0089e-05\\
1536	5.56477e-06\\
6144.00000000001	3.81941999999999e-06\\
24576.0000000001	5.22047e-06\\
98304.0000000002	6.50406999999999e-06\\
393216.000000001	4.67848e-06\\
};
\addlegendentry{ACA}

\addplot [color=mycolor2]
  table[row sep=crcr]{%
6	1.84201e-16\\
24	2.72473e-16\\
96.0000000000001	2.37206e-15\\
384.000000000001	9.45520000000001e-07\\
1536	4.90354999999999e-07\\
6144.00000000001	4.08197999999999e-07\\
24576.0000000001	4.83696e-07\\
98304.0000000002	4.93467e-07\\
393216.000000001	4.93854e-07\\
};
\addlegendentry{BiLanczos}

\addplot [color=mycolor3]
  table[row sep=crcr]{%
6	1.35336e-16\\
24	2.74101e-16\\
96.0000000000001	2.40365e-16\\
384.000000000001	9.44006000000001e-07\\
1536	4.28209e-07\\
6144.00000000001	3.65138999999999e-07\\
24576.0000000001	4.74733999999999e-07\\
98304.0000000002	4.66216999999999e-07\\
393216.000000001	4.70131999999999e-07\\
};
\addlegendentry{Randomized}

\addplot [color=mycolor4]
  table[row sep=crcr]{%
6	1.35138e-16\\
24	2.89015e-16\\
96.0000000000001	2.35462e-16\\
384.000000000001	1.0066e-06\\
1536	4.90878999999999e-07\\
6144.00000000001	4.08268e-07\\
};
\addlegendentry{SVD}

\end{axis}
\end{tikzpicture}%
}
\caption{\label{fig:vvefr}Error using fixed rank truncation
	for the product of the boundary integral operator matrices.}
\end{figure}
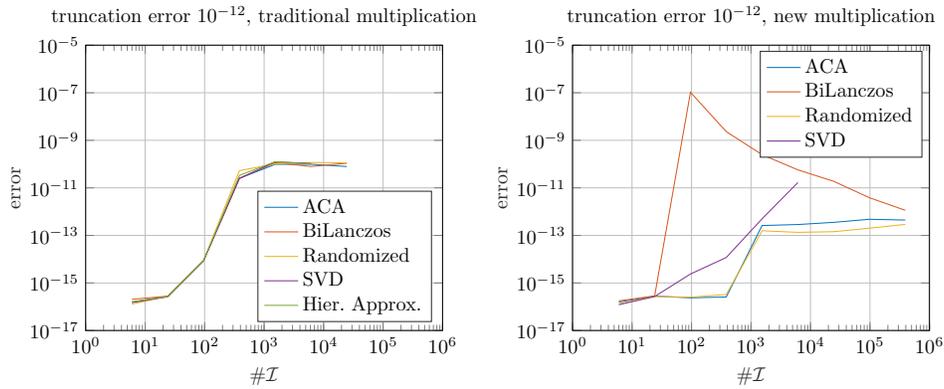
\begin{figure}
\scalefig{% This file was created by matlab2tikz.
\definecolor{mycolor1}{rgb}{0.00000,0.44700,0.74100}%
\definecolor{mycolor2}{rgb}{0.85000,0.32500,0.09800}%
\definecolor{mycolor3}{rgb}{0.92900,0.69400,0.12500}%
\definecolor{mycolor4}{rgb}{0.49400,0.18400,0.55600}%
\definecolor{mycolor5}{rgb}{0.46600,0.67400,0.18800}%
\begin{tikzpicture}

\begin{axis}[%
width=  0.5\textwidth,
height=       0.4\textwidth,
at={(         0\textwidth,         0\textwidth)},
scale only axis,
xmode=log,
xmin=1,
xmax=1000000,
xminorticks=true,
xlabel style={font=\color{white!15!black}},
xlabel={$\#\mathcal{I}$},
ymode=log,
ymin=1e-17,
ymax=1e-05,
ytick={1e-5,1e-7,1e-9,1e-11,1e-13,1e-15,1e-17},
yminorticks=true,
ylabel style={font=\color{white!15!black}},
ylabel={error},
axis background/.style={fill=white},
title={truncation error $10^{-12}$, traditional multiplication},
xmajorgrids,
ymajorgrids,
legend style={at={(      0.97,      0.03)}, anchor=south east, legend cell align=left, align=left, draw=white!15!black}
]
\addplot [color=mycolor1]
  table[row sep=crcr]{%
6	1.54901e-16\\
24	2.76458000000001e-16\\
96.0000000000001	8.69513999999999e-15\\
384.000000000001	2.49956e-11\\
1536	9.73781000000001e-11\\
6144.00000000001	9.83599999999999e-11\\
24576.0000000001	7.90079e-11\\
};
\addlegendentry{ACA}

\addplot [color=mycolor2]
  table[row sep=crcr]{%
6	2.05036e-16\\
24	2.78111e-16\\
96.0000000000001	8.70022000000002e-15\\
384.000000000001	2.54506e-11\\
1536	1.13784e-10\\
6144.00000000001	7.92805e-11\\
24576.0000000001	1.08748e-10\\
};
\addlegendentry{BiLanczos}

\addplot [color=mycolor3]
  table[row sep=crcr]{%
6	1.29315e-16\\
24	2.75209e-16\\
96.0000000000001	8.69583e-15\\
384.000000000001	5.30019e-11\\
1536	1.01341e-10\\
6144.00000000001	1.16845e-10\\
24576.0000000001	1.12772e-10\\
};
\addlegendentry{Randomized}

\addplot [color=mycolor4]
  table[row sep=crcr]{%
6	1.49823e-16\\
24	2.62035e-16\\
96.0000000000001	8.69635000000001e-15\\
384.000000000001	2.48418e-11\\
1536	1.24758e-10\\
6144.00000000001	1.05736e-10\\
};
\addlegendentry{SVD}

\addplot [color=mycolor5]
  table[row sep=crcr]{%
6	1.62787e-16\\
24	2.64489e-16\\
96.0000000000001	8.69315000000001e-15\\
384.000000000001	3.32707999999999e-11\\
1536	1.26079e-10\\
6144.00000000001	1.00748e-10\\
};
\addlegendentry{Hier.~Approx.}

\end{axis}
\end{tikzpicture}%
}
\scalefig{% This file was created by matlab2tikz.
\definecolor{mycolor1}{rgb}{0.00000,0.44700,0.74100}%
\definecolor{mycolor2}{rgb}{0.85000,0.32500,0.09800}%
\definecolor{mycolor3}{rgb}{0.92900,0.69400,0.12500}%
\definecolor{mycolor4}{rgb}{0.49400,0.18400,0.55600}%
\begin{tikzpicture}

\begin{axis}[%
width=  0.5\textwidth,
height=       0.4\textwidth,
at={(         0\textwidth,         0\textwidth)},
scale only axis,
xmode=log,
xmin=1,
xmax=1000000,
xminorticks=true,
xlabel style={font=\color{white!15!black}},
xlabel={$\#\mathcal{I}$},
ymode=log,
ymin=1e-17,
ymax=1e-05,
ytick={1e-5,1e-7,1e-9,1e-11,1e-13,1e-15,1e-17},
yminorticks=true,
ylabel style={font=\color{white!15!black}},
ylabel={error},
axis background/.style={fill=white},
title={truncation error $10^{-12}$, new multiplication},
xmajorgrids,
ymajorgrids,
legend style={legend cell align=left, align=left, draw=white!15!black}
]
\addplot [color=mycolor1]
  table[row sep=crcr]{%
6	1.56625e-16\\
24	2.81015e-16\\
96.0000000000001	2.36584000000001e-16\\
384.000000000001	2.57695e-16\\
1536	2.60614e-13\\
6144.00000000001	2.86716e-13\\
24576.0000000001	3.54386e-13\\
98304.0000000002	4.77121e-13\\
393216.000000001	4.41998000000001e-13\\
};
\addlegendentry{ACA}

\addplot [color=mycolor2]
  table[row sep=crcr]{%
6	1.76429e-16\\
24	2.83583000000001e-16\\
96.0000000000001	1.04978e-07\\
384.000000000001	2.31287e-09\\
1536	2.49396e-10\\
6144.00000000001	5.71738000000001e-11\\
24576.0000000001	1.88509e-11\\
98304.0000000002	3.83590999999999e-12\\
393216.000000001	1.13462e-12\\
};
\addlegendentry{BiLanczos}

\addplot [color=mycolor3]
  table[row sep=crcr]{%
6	1.31978e-16\\
24	2.69053e-16\\
96.0000000000001	2.52348e-16\\
384.000000000001	3.28076000000001e-16\\
1536	1.56857e-13\\
6144.00000000001	1.34304e-13\\
24576.0000000001	1.43547e-13\\
98304.0000000002	2.01016e-13\\
393216.000000001	2.91043e-13\\
};
\addlegendentry{Randomized}

\addplot [color=mycolor4]
  table[row sep=crcr]{%
6	1.2173e-16\\
24	2.68056e-16\\
96.0000000000001	2.37787e-15\\
384.000000000001	1.1766e-14\\
1536	4.99659e-13\\
6144.00000000001	1.67449e-11\\
};
\addlegendentry{SVD}

\end{axis}
\end{tikzpicture}%
}
\caption{\label{fig:vveer}Error using $\varepsilon$-rank truncation
	for the product of the boundary integral operator matrices.}
\end{figure}

%!TEX root = paper.tex
\section{Conclusion}\label{sec:conclusion}
The multiplication of hierarchical matrices is a widely used algorithm in
engineering. The recursive scheme of the original implementation
and the required upper threshold for the rank make an a-priori error
analysis of the algorithm difficult. Although several approaches to reduce
the number of error-introducing operations have been made in the literature,
an algorithm providing a fast multiplication of $\mathcal{H}$-matrices with
a-priori error-bounds is still not available nowadays. By introducing
\sumexpr-expressions, which can be seen as a queuing system of low-rank
matrices and $\mathcal{H}$-matrix products, we can postpone the
error-introducing low-rank approximation until the last stage of the
algorithm. We have discussed several adaptive low-rank approximation methods
based on matrix-vector multiplications which make an a-priori error analysis
of the $\mathcal{H}$-matrix multiplication possible. The cost analysis
shows that the cost of our new $\mathcal{H}$-matrix multiplication
algorithm is almost linear with respect to the size of the matrix. In particular,
the numerical experiments show that the new approach can compute the best
approximation of the $\mathcal{H}$-matrix product while being
computationally more efficient than the traditional product.

Parallelization is an important topic on modern computer architectures.
Therefore, we remark that the computation of the low-rank approximations
for the target blocks is easily parallelizeable, once the necessary
\sumexpr-expression is available. We also remark that the computation
of the \sumexpr-expressions does not require concurrent write access,
such that the parallelization on a shared memory machine should be
comparably easy.

\section*{Acknowledgements}
The authors would like to thank Daniel Kressner for many fruitful
discussions on iterative and randomized eigensolvers.

The work of J\"urgen D\"olz is supported by the Swiss National 
Science Foundation (SNSF) through the project 174987 and 
by the Excellence Initiative of the German Federal and State 
Governments and the Graduate School of Computational 
Engineering at TU Darmstadt.

%===============================================================================
%===============================================================================
\bibliographystyle{plain}
\bibliography{bibl}
\end{document}